\pdfoutput=1
\documentclass[
	,12pt%
	,pagesize%
	,headings=small%
	,paper=a4%
	,parskip=false%
	,abstracton
	,toc=bibliography
]{scrartcl}
\addtokomafont{sectioning}{\normalfont\bfseries}
\setkomafont{title}{\normalfont}
\setkomafont{subtitle}{\normalfont}
\setkomafont{author}{\normalfont}
\setkomafont{date}{\normalfont}
\addtokomafont{pageheadfoot}{\scshape\small}
\setkomafont{caption}{\footnotesize}
\setkomafont{captionlabel}{\usekomafont{caption}\bfseries}
\setcapindent{0pt}
\usepackage{etoolbox}
\AtBeginEnvironment{abstract}{\footnotesize}

\usepackage{xpatch}
\makeatletter
\patchcmd{\@maketitle}{\huge}{\Large}{}{}
\makeatother

\usepackage[utf8]{inputenc}
\usepackage[T1]{fontenc}
\usepackage{csquotes}
\usepackage[english]{babel}
\usepackage{graphicx}
\graphicspath{{Pictures/}{/}}
\usepackage[export]{adjustbox}

\usepackage[usenames,dvipsnames]{xcolor}
\usepackage{tikz-cd}
\usetikzlibrary{backgrounds}

\usepackage[draft]{fixme}
\usepackage[nointlimits]{amsmath}
\usepackage{amssymb,mathrsfs,mathtools}
\usepackage{txfonts}
\usepackage{upgreek}
\usepackage{braket}
\usepackage{cancel}
\usepackage{siunitx}
\usepackage{aliascnt} %alias counter
\usepackage[amsmath,thmmarks,hyperref]{ntheorem}
\usepackage[expansion=true,protrusion=true]{microtype}
\usepackage{xkeyval}

\usepackage[final,%
	pdftex,%
	bookmarks,%
	bookmarksdepth=3,%
	breaklinks=true,%
	colorlinks=true,%
	urlcolor=NavyBlue,%
	linkcolor=NavyBlue,%
	citecolor=ForestGreen,%	
]{hyperref}%
\usepackage[all]{hypcap}
\newcommand{\cref}[1]{\autoref{#1}}
\addto\extrasenglish{%
}

\definecolor{RoyalBlue}{RGB}{0,32,96}
\definecolor{ForestGreen}{RGB}{34,139,34}
\definecolor{BrickRed}{RGB}{203, 65, 84}

\newcommand{\MeritFun}{f}

\newcommand{\ext}[1]{#1}
\newcommand{\regparam}{\alpha}
\newcommand{\DouadyEarle}{E}

\newcommand{\Shift}{\sigma}

\DeclareMathAlphabet{\mathcal}{OMS}{cmsy}{m}{n}

\newcommand{\OpenBall}[2]{B(#1;#2)}
\newcommand{\ClosedBall}[2]{\bar B(#1;#2)}

\newcommand{\AmbDim}{d} % dimension of ambient space

\newcommand{\Sphere}{\mathbb{S}}
\newcommand{\Disk}{\mathbb{B}}
\newcommand{\UnitTangent}{\operatorname{UT}}
\newcommand{\TangentBundle}{\operatorname{T}}

\DeclareDocumentCommand{\Hess}{ O{} }{\operatorname{Hess}_{#1}}

\newcommand{\Var}{\operatorname{Var}}

\newcommand{\qand}{\quad \text{and} \quad}

\DeclareMathOperator*{\esssup}{ess\,sup}

\newcommand{\converges}[1]{\stackrel{#1}{\longrightarrow}}

\newcommand{\transp}{^{\mathsf{T\!}}}

\newcommand{\dual}{'^{\!}}

\newcommand{\dd}{{\operatorname{d}}}

\newcommand{\at}{|}

\newcommand{\ceq}{\coloneqq}

\newcommand{\R}{{\mathbb{R}}}
\newcommand{\C}{{\mathbb{C}}}
\newcommand{\N}{\mathbb{N}}

\DeclareMathOperator{\id}{id}

\newcommand{\abs}[1]{\left\lvert#1\right\rvert} %absolute value
\newcommand{\nabs}[1]{\lvert{#1}\rvert} %absolute value
\newcommand{\nnorm}[1]{\lVert{#1}\rVert}

\newcommand{\ninnerprod}[1]{\langle{#1}\rangle}

\newcommand{\paren}[1]{\left(#1\right)}
\newcommand{\nparen}[1]{(#1)}
\newcommand{\bigparen}[1]{\big(#1\big)}
\newcommand{\biggparen}[1]{\bigg(#1\bigg)}

\newcommand{\intervalco}[1]{\left[#1\right)}
\newcommand{\intervalcc}[1]{\left[#1\right]}

\newcommand{\nintervalcc}[1]{{[#1]}}

\DeclareMathOperator{\Aut}{Aut}
\DeclareMathOperator{\supp}{supp}   %support
\DeclareMathOperator{\dist}{dist}  %Distance
\DeclareMathOperator{\End}{End}    %Endomorphisms
\DeclareMathOperator{\grad}{grad}

\newcommand{\pull}{\#}
\newcommand{\push}{\#}
\newcommand{\mynewtheorem}[4] %{BEZEICHNER}{COUNTER}{TITEL} - Für Numerierung mit \autoref aus dem \hyperref-Packet
{
\newaliascnt{#1}{#2}
\newtheorem{#1}[#1]{#3}
\aliascntresetthe{#1}
\expandafter\def\csname #1autorefname\endcsname{%
#4%
}%
}

\newtheorem{theorem}{Theorem}[section]
\mynewtheorem{lemma}{theorem}{Lemma}{Lemma} 
\mynewtheorem{proposition}{theorem}{Proposition}{Proposition} 
\mynewtheorem{corollary}{theorem}{Corollary}{Corollary}
\mynewtheorem{quest}{theorem}{Question}{Question}

\mynewtheorem{problem}{theorem}{Problem}{Problem}

\theoremstyle{break}
\mynewtheorem{btheorem}{theorem}{Theorem}{Theorem} 
\mynewtheorem{blemma}{theorem}{Lemma}{Lemma}
\mynewtheorem{bproposition}{theorem}{Proposition}{Proposition}
\mynewtheorem{bcorollary}{theorem}{Corollary}{Corollary}
\mynewtheorem{bproblem}{theorem}{Problem}{Problem}

\theoremstyle{plain}
\theorembodyfont{\normalfont}
\mynewtheorem{definition}{theorem}{Definition}{Definition}
\mynewtheorem{example}{theorem}{Example}{Example}
\mynewtheorem{remark}{theorem}{Remark}{Remark}

\theoremstyle{break}
\mynewtheorem{bdfn}{theorem}{Definition}{Definition}
\mynewtheorem{bex}{theorem}{Example}{Example}
\mynewtheorem{brem}{theorem}{Remark}{Remark}
\theoremheaderfont{\itshape}
\theorembodyfont{\upshape}
\theoremstyle{nonumberplain}
\theoremseparator{.}
\theoremsymbol{\ensuremath{\Box}}
\newtheorem{proof}{\textsc{Proof}}

\begin{document}
\title{Computing the conformal barycenter}
\author{%
Jason Cantarella
\and
Henrik Schumacher
}

\maketitle

\begin{abstract}
\begin{small}
The conformal barycenter of a point cloud on the sphere at infinity of the Poincar\'e ball model of hyperbolic space is a hyperbolic analogue of the geometric median of a point cloud in Euclidean space. It was defined by Douady and Earle as part of a construction of a conformally natural way to extend homeomorphisms of the circle to homeomorphisms of the disk, and it plays a central role in Millson and Kapovich's model of the configuration space of cyclic linkages with fixed edgelengths.
 
In this paper we consider the problem of computing the conformal barycenter. Abikoff and Ye have given an iterative algorithm for measures on $\Sphere^1$ which is guaranteed to converge. We analyze Riemannian versions of Newton's method computed in the intrinsic geometry of the Poincare ball model. We give Newton-Kantorovich (NK) conditions under which we show that Newton's method with fixed step size is guaranteed to converge quadratically to the conformal barycenter for measures on any $\Sphere^d$ (including infinite-dimensional spheres). For measures given by $n$ atoms on a finite dimensional sphere which obey the NK conditions, we give an explicit linear bound on the computation time required to approximate the conformal barycenter to fixed error. We prove that our NK conditions hold for all but exponentially few $n$ atom measures. For all measures with a unique conformal barycenter we show that a regularized Newton's method with line search will always converge (eventually superlinearly) to the conformal barycenter. Though we do not have hard time bounds for this algorithm, experiments show that it is extremely efficient in practice and in particular much faster than the Abikoff-Ye iteration.
\\

\noindent
\textbf{MSC-2020 classification:} 
65D18, % Numerical aspects of computer graphics, image analysis, and computational geometry
65E10, %	Numerical methods in conformal mappings
53-08 % Computational methods for problems pertaining to differential geometry
\end{small}
\end{abstract}

\section{Introduction}\label{sec:intro}

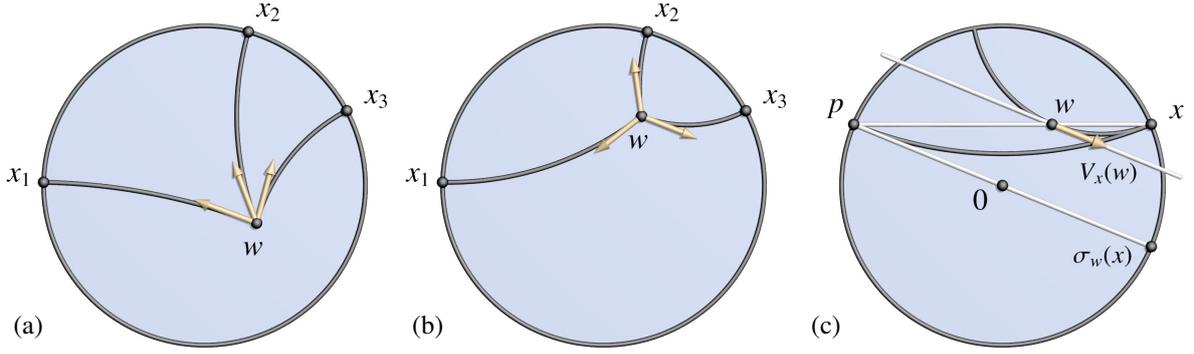
\begin{figure}
\begin{center}
	\capstart %ensures that hyperlink will jump to the top of this image
\newcommand{\inc}[2]{\begin{tikzpicture}
    \node[inner sep=0pt] (fig) at (0,0) {\includegraphics{#1}};
	\node[above right= 0ex] at (fig.south west) {\begin{footnotesize}(#2)\end{footnotesize}};    
\end{tikzpicture}}%
\presetkeys{Gin}{
	trim = 0 100 15 20, 
	clip = true,  
	width = 0.333\textwidth
%	height= 0.165\textheight
}{}
	\hfill
	\inc{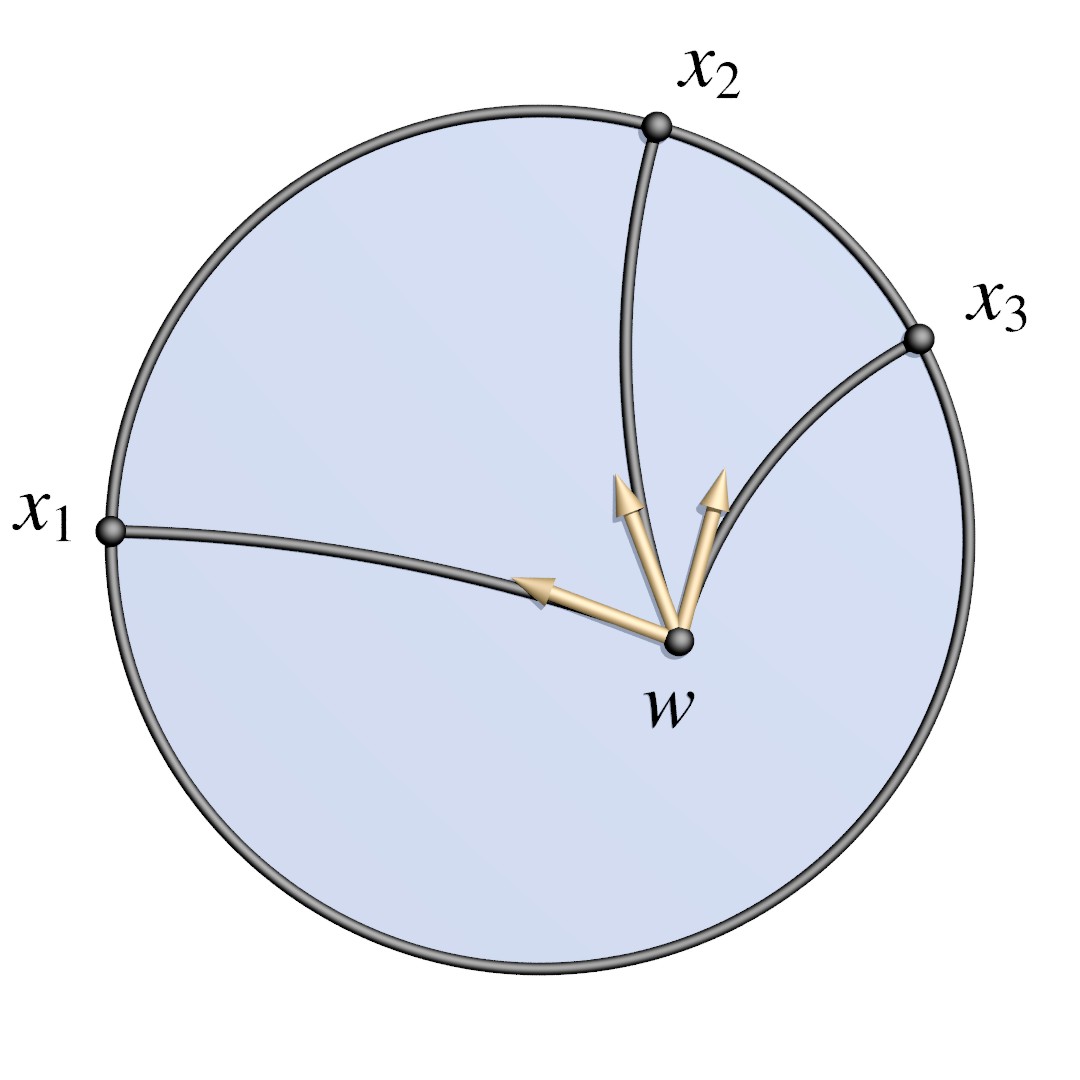}{a}%
	\hfill
	\inc{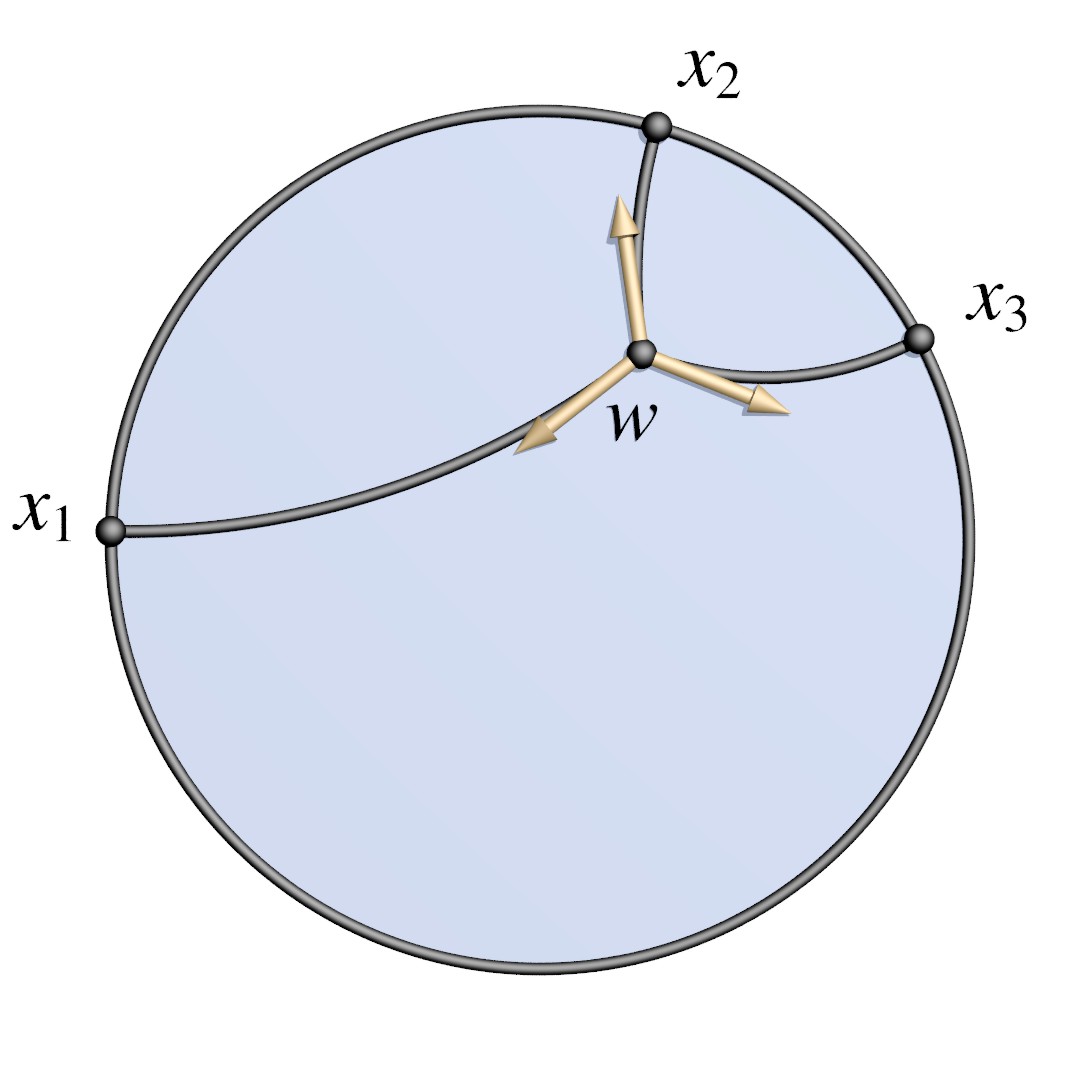}{b}%
	\hfill	
	\inc{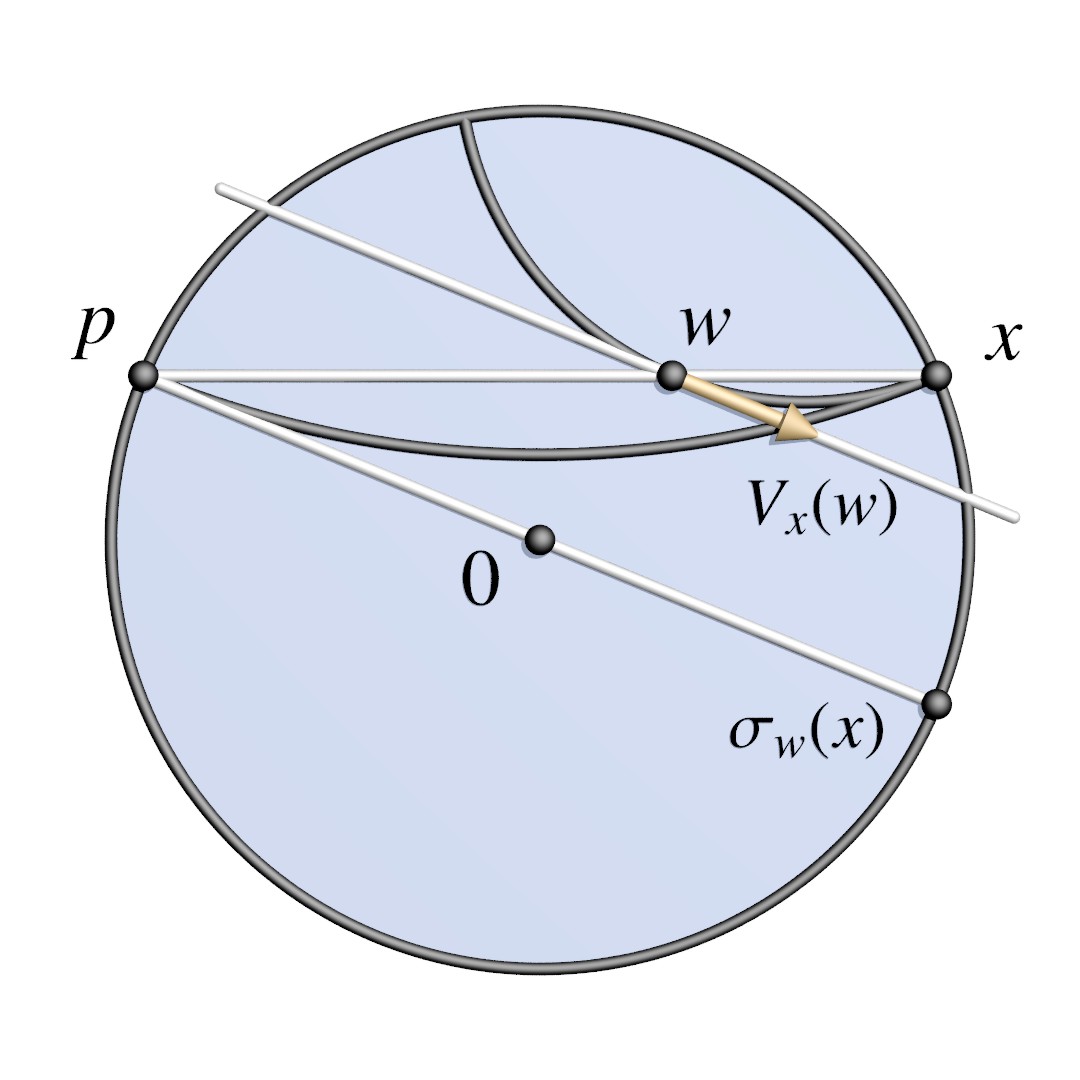}{c}%	
	\hfill{}	
	\caption{%
	(a)	
	Geodesics joining some $w$ in $\Disk$ to three points $x_1, x_2, x_3$ in~$\Sphere$ and the corresponding conformal directors.
	The directors $V_{x_i}(w)$ do not sum up to $0$.
	(b)
	Same as (a), but here the sum of directors $V_{x_i}(w)$ vanishes;  thus $w$ is the conformal barycenter of the $x_i$.
	(c) Geometric construction of the directors: 
	Each geodesic emanating from $x$ intersects the secant $p x$ in the same angle. Thus
	all directors $V_x(w)$ for $w$ on the secant $p x$  point in the same direction.
	}
\label{fig:conformalbary}
\end{center}
\end{figure}
In~\cite{MR0857678}, Douady and Earle defined the~\emph{conformal barycenter} of a measure on $\Sphere^{\AmbDim-1}$. Suppose we give the (open) ball $\Disk^\AmbDim$ the geometry of the Poincar\'e ball model of hyperbolic space. For each pair of points $w \in \Disk^\AmbDim$ and $x$ on the sphere $\Sphere^{\AmbDim-1}$ at infinity, there is a unique geodesic joining $w$ to $x$. The unit tangent vector to this geodesic at $w$ will be called the \emph{director} $V_x(w)$. The conformal barycenter (see~\cref{fig:conformalbary}) of $\mu$ is the point $w_*$ where the weighted average of directors vanishes:
\begin{align*}
	\textstyle
	F_\mu (w_*) \ceq \int_{x \in \Sphere^{\AmbDim-1}} V_x(w_*) \, \dd \mu(x) = 0. 
\end{align*}
The conformal barycenter is comparable to the geometric median~\cite{MR1933966} of a weighted point cloud $x_1,\dotsc x_n$ in Euclidean space with weights $\omega_1,\dotsc,\omega_n$, which is (generically) the point $w_*$ where
$
     \textstyle \sum_{i=1}^n \omega_i \, (x_i - w_*) / \nabs{x_i - w_*} = 0.
$
The conformal barycenter is clearly isometry-equivariant in hyperbolic geometry; this means that if $w_*(\mu)$ is the conformal barycenter of $\mu$ and $\varphi$ is a M\"obius transformation of $\Disk^\AmbDim$,
then $w_*(\varphi_\push \mu) = \varphi(w_*(\mu))$, where $\varphi_\push \mu$ is the push-forward of~$\mu$.  Douady and Earle's original motivation for their construction was to provide a canonical extension of homeomorphisms of the circle to homeomorphisms of the disk that is also conformally equivariant.

The conformal barycenter for discrete measures composed of $n$ atoms with weights $\omega_1, \dots, \omega_n$ plays a central role in the symplectic model for the space of closed polygons with fixed edgelengths in $\R^{3}$ given by Kapovich and Millson~\cite{MR1431002}. They identify the atoms $x_1, \dots, x_n$ of $\mu$ with the directions of the edges of the polygon and $\omega_1, \dots, \omega_n$ with the lengths of the edges. Such a polygon is closed if and only if $\sum_i \omega_i \, x_i = 0$. They show that for a ``stable'' (see~\cref{def:stable}) measure $\mu$ of this type, there is a M\"obius transformation $\varphi$ so that the conformal barycenter $w_*(\varphi_\push \mu) = 0$. Further, this $\varphi$ is unique up to postcomposition with an element of $\operatorname{SO}(3)$, and $w_*(\mu) = 0$ if and only if $\sum_i \omega_i \, x_i = 0$.

This construction allows them to show (Theorem 2.7) that the quotient space of ``nice semi-stable'' (\cref{def:stable}) $n$-atom measures $\mu$ with weights $\omega_1, \dots, \omega_n$ with respect to all M\"obius transformations is homeomorphic (and even complex-analytically equivalent) to the quotient space of $n$-edge closed polygons with edgelengths $\omega_1, \dots, \omega_n$ by the action of~$\operatorname{SO}(3)$. 

Generalizing their construction in the obvious way, this means that the space of nice semi-stable measures $\mu$ (or weighted point clouds on $\Sphere^{\AmbDim-1}$) is an \emph{exact} (but redundant) system of coordinates for the space of closed polygonal linkages with fixed edgelengths in $\R^\AmbDim$. That is, every stable weighted point cloud $\mu$ on $\Sphere^{\AmbDim-1}$ exactly represents a unique (up to the action of $\operatorname{SO}(\AmbDim)$) closed polygon in $\R^\AmbDim$ with the corresponding edgelengths. (We lose complex-analytic equivalence, as the spaces do not generally carry a complex structure.)

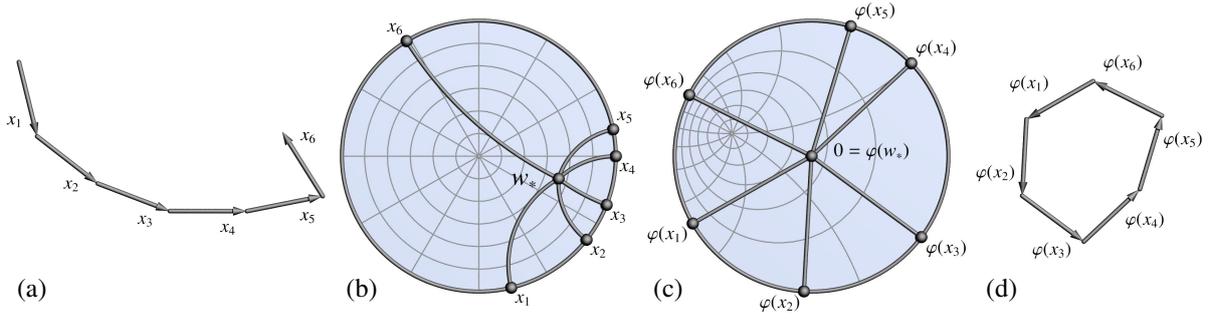
\begin{figure}
\newlength{\subht}
\newsavebox{\subbox}
\sbox\subbox{%
  \resizebox{\dimexpr1.0\textwidth}{!}{%
    \includegraphics[trim = 0 0 25 50, clip = true, height= 0.5\textheight]{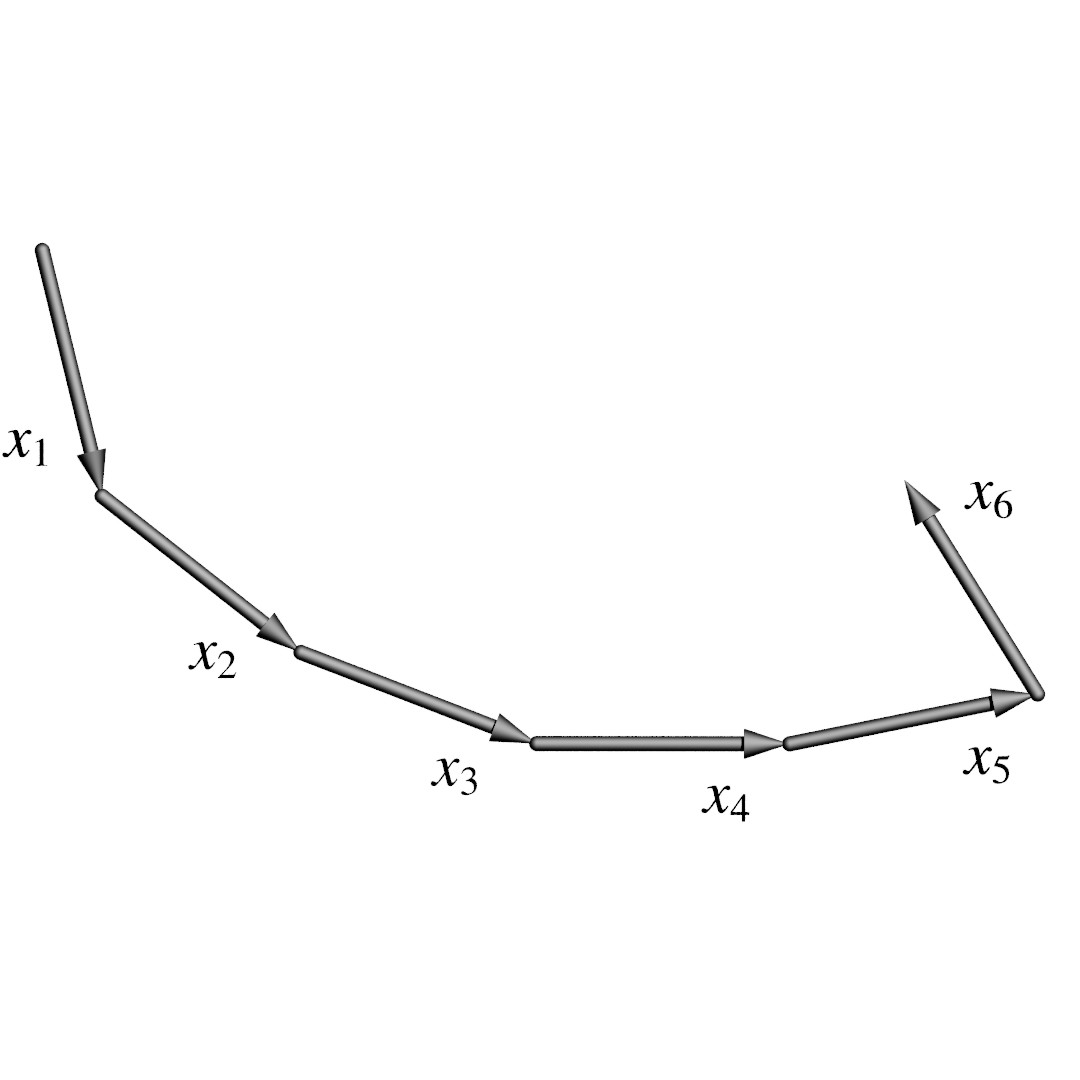}%
    \includegraphics[trim = 10 45 50 50, clip = true, height= 0.5\textheight]{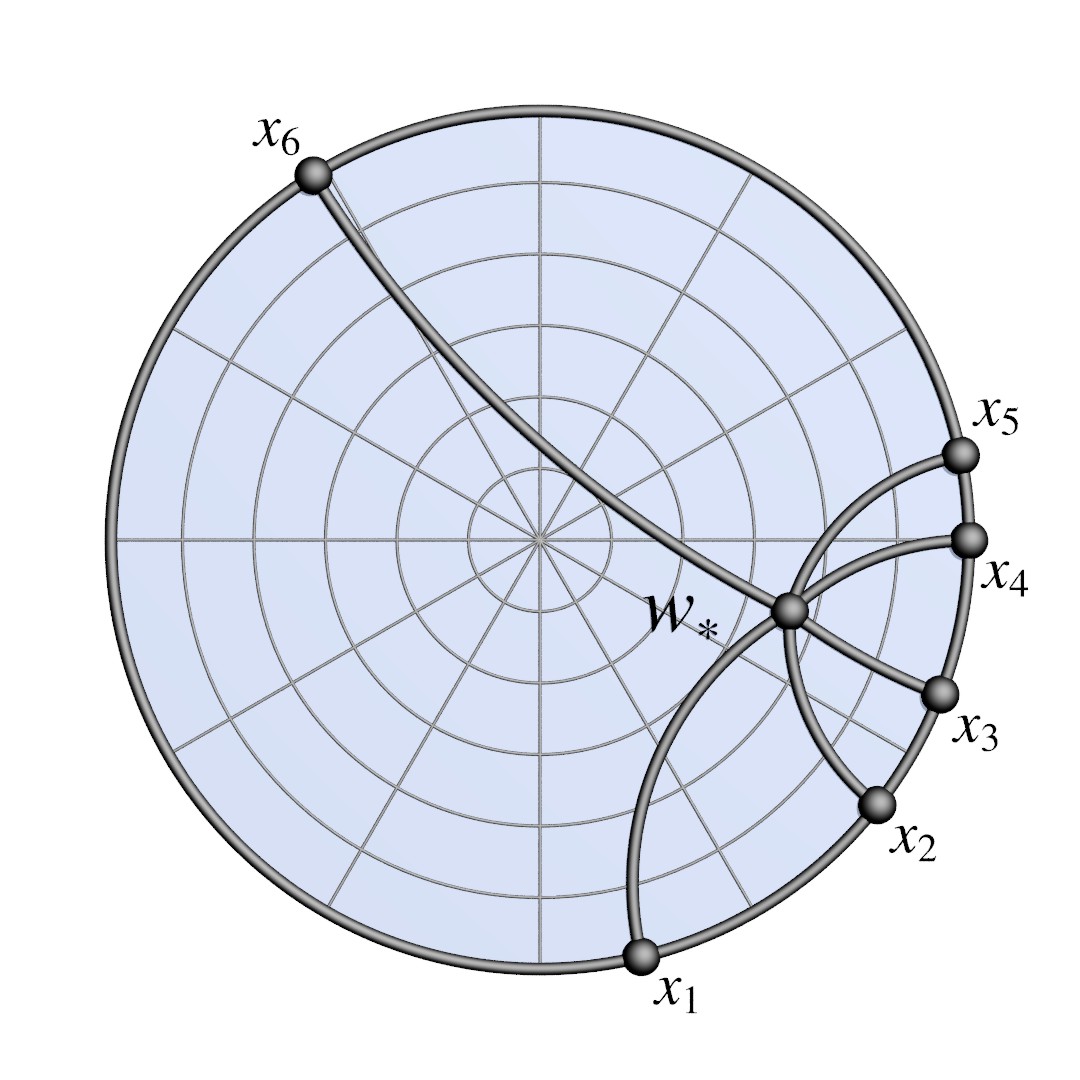}% 
	\includegraphics[trim = 10 45 50 50, clip = true, height= 0.5\textheight]{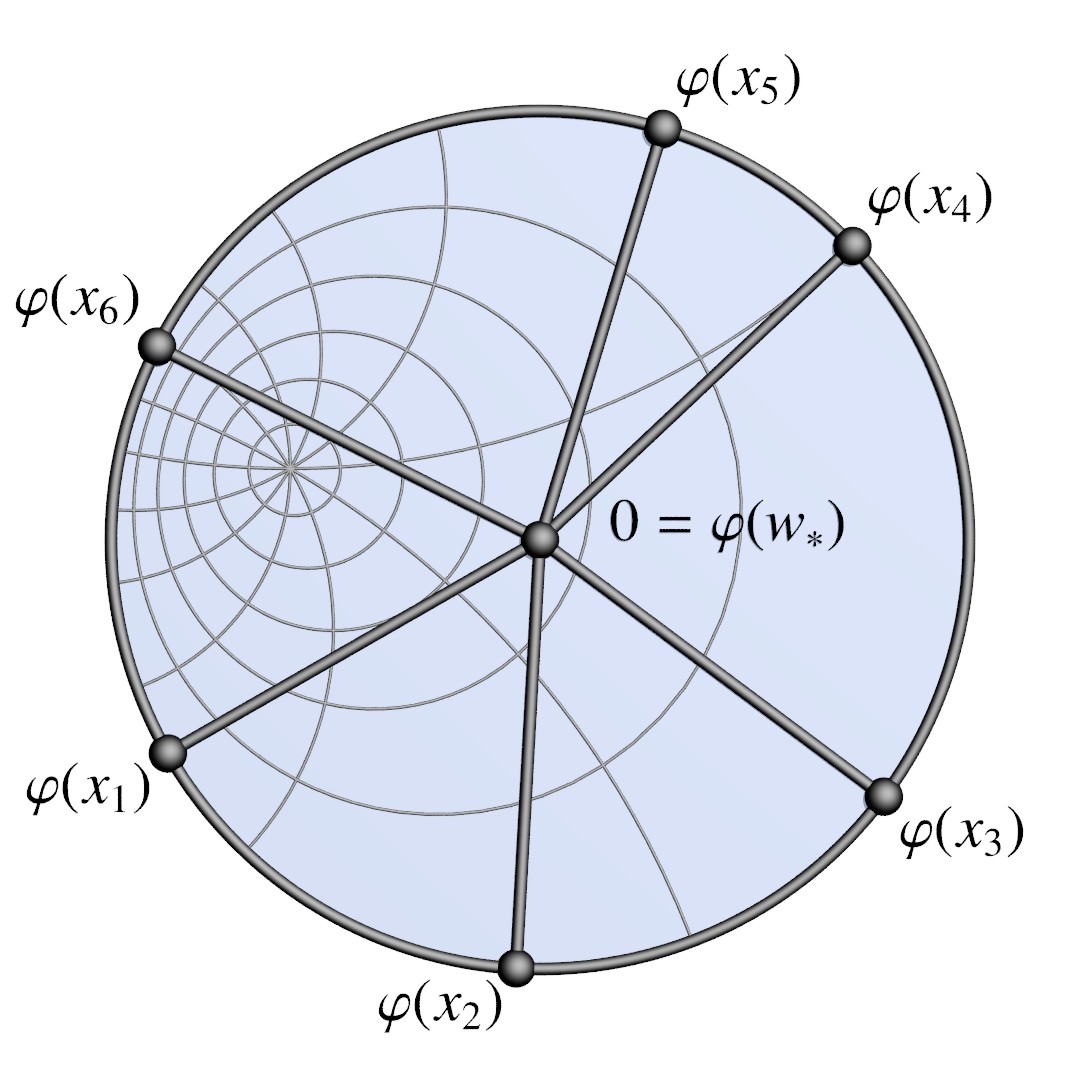}%  
	\includegraphics[trim = 170 25 160 25, clip = true, height= 0.5\textheight]{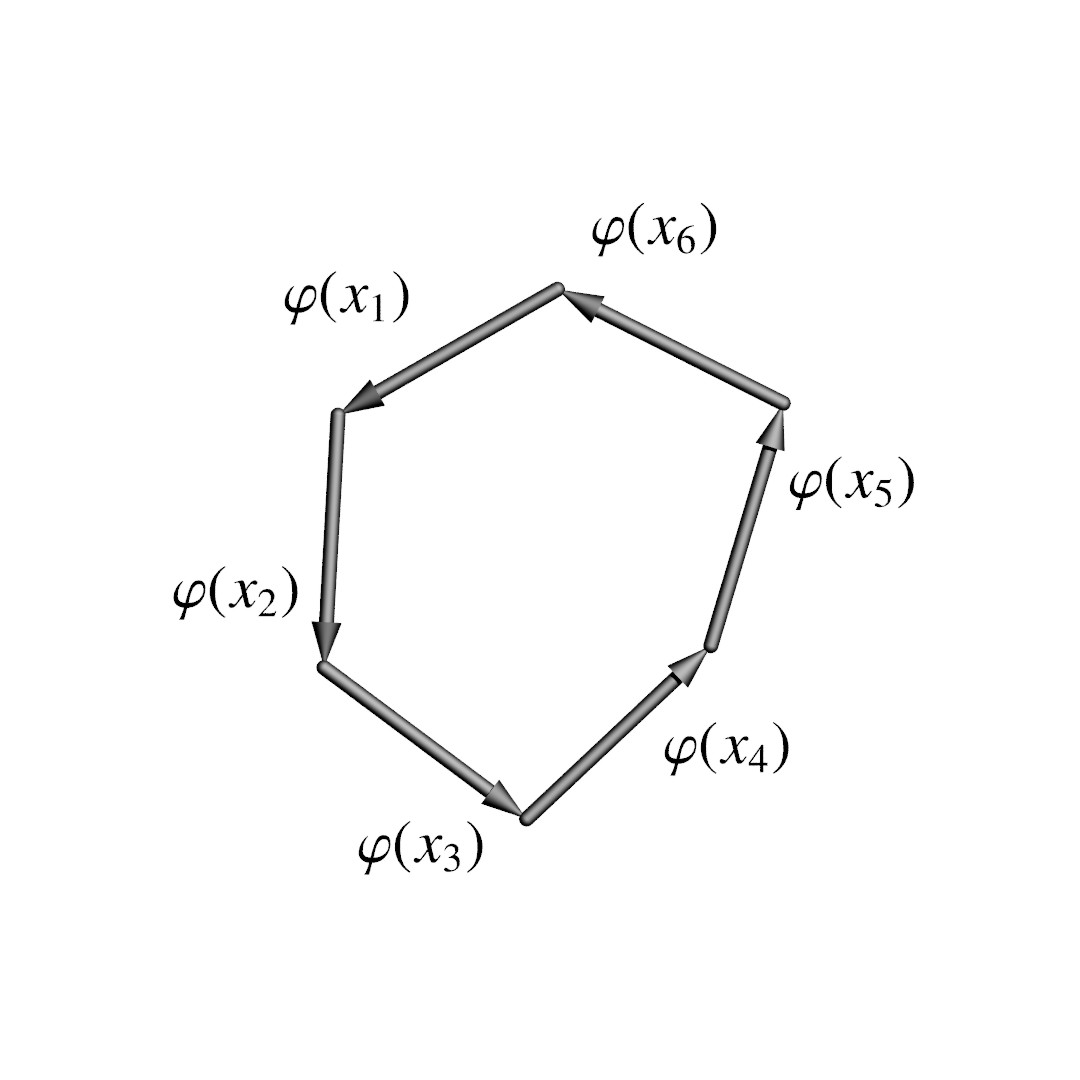}%	  
  }%
}
\setlength{\subht}{\ht\subbox}
\begin{center}
\newcommand{\inc}[2]{\begin{tikzpicture}
    \node[inner sep=0pt] (fig) at (0,0) {\includegraphics[clip = true, height= \subht]{#1}};
	\node[above right= 0ex] at (fig.south west){\begin{footnotesize}(#2)\end{footnotesize}};    
\end{tikzpicture}}%
	\presetkeys{Gin}{trim = 0 0 25 50}{}%
	\inc{ConformalClosure_Explained_1}{a}%
	\hfill
	\presetkeys{Gin}{trim = 90 45 50 50}{}%	
	\inc{ConformalClosure_Explained_2}{b}%
	\hfill
	\presetkeys{Gin}{trim = 10 45 50 50}{}%
	\inc{ConformalClosure_Explained_3}{c}%
	\hfill	
	\presetkeys{Gin}{trim = 170 25 160 25}{}%	
	\inc{ConformalClosure_Explained_4}{d}%	
\caption{
Millson and Kapovich's construction:
(a)~An open polygon. 
(b)~The corresponding point measure $\mu$ on $\Sphere^1$ where the masses $x_i$ of the measure are located at the directions of the edges (and weighted by their lengths), together with the conformal barycenter $w_*$ of $\mu$. 
(c)~We see $\mu$ shifted by a conformal transformation $\varphi$ to bring the conformal barycenter (and hence the center of mass) to the origin. 
(d)~The transformed edges make up a closed polygon with the same edgelengths. 
}
\label{fig:ConformalClosure_Explained}
\end{center}
\end{figure}

Such a system of coordinates promises to be extremely useful in studying linkage problems in computational geometry and in implementing robust numerical methods for linkage reconfiguration. The linkage reconfiguration problem has a vast array of applications, including robotics, protein folding, polymer science and computational origami (see~\cite{MR2354878} for an overview of the field). However, to make these coordinates computationally practical, we must solve one remaining problem: If the $\mu$ representing our polygon has $\sum_i \omega_i \, x_i \neq 0$, to recover the unique closed polygon represented by $\mu$ we must find the  M\"obius transformation $\varphi$ so that $w_*(\varphi(\mu)) = 0$. Since $w_*(\varphi_\push \mu) = \varphi(w_*(\mu))$, this is equivalent to finding $w_*(\mu)$. \cref{fig:ConformalClosure_Explained} shows an example of a stable weighted point cloud $\mu$ on $\Sphere^1$ with $\sum \omega_i \, x_i \neq 0$, the corresponding open polygon, the M\"obius transformed $\varphi_\push \mu$ with $\sum \omega_i \, \varphi(x_i) = 0$, and the corresponding closed polygon.
%In this case, the conformal barycenter plays a role in hyperbolic geometry analogous to the geometric median or Fermat-Weber point in Euclidean geometry, which is the point that minimizes the $L^1$ sum of distances to a fixed set of points. The geometric median has a long history and has proved useful both as a robust statistic and in a variety of applications~(cf.\ \cite{Hamacher:2002vp,Vardi:2000kq}).

This motivates us to consider the question of numerically computing the conformal barycenter of a measure on $\Sphere^{\AmbDim-1}$ and to pay particular attention to the case of $n$-atom measures with variable weights. Abikoff and Ye~\cite{MR1476978} and Abikoff~\cite{MR1894482} have given an iterative algorithm based on a suggestion of Milnor and an algorithm based on Newton's method in the Euclidean geometry of $\R^2$ for computing the conformal barycenter of such a measure.
However, they only prove that their method converges for measures on $\Sphere^1$ and were not able to give effective bounds for the rate of convergence. 
In practice, their method can be very slow for highly concentrated measures.
The primary insight of this paper is that switching the setting of Newton's method from the Euclidean geometry of the unit ball to the hyperbolic geometry of the Poincar\'e ball model leads to greatly improved theoretical and numerical results. 

After recording some preliminary calculations (\cref{sec:background}), we will analyze the Riemannian Newton's method with fixed step size (\cref{sec:newton with fixed stepsize}). We give Newton-Kantorovich (NK) conditions under which the iteration converges quadratically for all steps (\cref{thm:quadratic convergence}) using a theorem of Ferreira and Svaiter~\cite{MR1895088}. Our convergence results hold for all $d$, including infinite-dimensional spheres. We then show that under NK conditions, our algorithm solves the problem ``compute the conformal barycenter of an $n$-atom measure on $\Sphere^d$ to within error $\varepsilon$'' in $O(n)$ time~(\cref{cor:time bound under NK conditions}). 

We then turn to the question of how often~\cref{cor:time bound under NK conditions} applies to $n$ atom measures (\cref{sec:how often do NK conditions hold}). The answer is ``in all but exponentially few (in $n$) cases'' (\cref{thm:how often}). The key tool is an estimate on the eigenvalues of a certain random matrix previously computed by one of us (Cantarella) using the Matrix Bernstein inequality~\cite{MR3871192}. 

We then consider cases where the NK conditions may not hold (\cref{sec:NewtonRegularized}). We analyze the behavior of a regularized Newton's method with line search~\eqref{eq:Newton3.1}--\eqref{eq:Newton3.2}. We show that if $\mu$ has a unique conformal barycenter $w_*$, the algorithm converges (eventually Q-quadratically) to $w_*$~(\cref{thm:main}) using results of Ring and Wirth on Riemannian Newton methods~\cite{MR2968868}. The key idea is to recast the problem as finding the minimizer of a function $\varPsi_\mu$~\eqref{eq:Potential} which we can prove to be uniformly convex on balls of finite radius in hyperbolic geometry~(\cref{lem:Uniform Convexity}). Again, our results prove convergence and bound the rate for all finite $\AmbDim$ and even for infinite dimensional spheres.

We conclude by using our algorithm to compute examples of polygon closures and Douady-Earle extensions of maps from $\Sphere^1$ into $\Sphere^2$ (\cref{sec:experiments}). Our algorithm performs very well even when the Milnor-Abikoff-Ye iteration struggles to converge. 

\section{Background} 
\label{sec:background}

Although we currently only have applications for the computation of conformal barycenters for measures on $\Sphere^{\AmbDim-1}$ where $\AmbDim$ is finite, all our methods work just as well when the ambient space is an arbitrary Hilbert space. Therefore, we will work in that context for generality. 
 
Let $H$ be a real Hilbert space with inner product $\ninnerprod{\cdot , \cdot } \colon H \times H \to \R$, norm $\nabs{u} \ceq \sqrt{\ninnerprod{u,u}}$, and with Riesz isomorphism $(\cdot)\transp \colon H \to H'$ given by $u\transp(v) \ceq \ninnerprod{u ,v}$ for $u$, $v \in H$. We now recall the Poincar\'e ball model for  hyperbolic space on $H$: $\Disk$ will denote the (open) ball $\set{z \in H | \nabs{z}< 1}$ with Riemannian metric $g \colon \Disk \to L^2(\TangentBundle\Disk;\R)$ given by
\begin{align}
	g \at_w(X, Y) \ceq \tfrac{4}{(1- \nabs{w}^2)^{2}} \,  \ninnerprod{X,Y},
	\qquad
	\text{for $w \in \Disk$ and $X$, $Y \in \TangentBundle_w \Disk \cong H$.}
	\label{eq:metric}
\end{align}
The length of a tangent vector $X \in \TangentBundle_w \Disk$ \emph{with respect to $g$} will be denoted by $\nabs{X}_g \ceq \sqrt{g\at_w(X,X)}$; we have to distinguish it carefully from $\nabs{X}$.
The induced Riesz isomorphism will be denoted by $(\cdot)^\flat \colon T_w \Disk \to T_w\dual \Disk$, $X \mapsto X^\flat$ and
the geodesic distance between two points $w_1$, $w_2 \in \Disk$ will be denoted by $d_g(w_1,w_2)$.
We point out that even if $H$ is infinite-dimensional, the Poincar\'e ball model $(\Disk,g)$ is a (strongly) Riemannian manifold in the sense that $g\at_w$ turns $\TangentBundle_w \Disk$ into a Hilbert space (see, e.g., \cite[Chapter VII]{MR1335233}). Moreover, $(\Disk,g)$ is \emph{geodesically complete} and \emph{geodesically convex}, properties inherited basically from the two-dimensional Poincar\'e model. One can deduce these properties also from the observation that the Riemannian exponential map $\exp_0 \colon \TangentBundle_0 \Disk \to \Disk$ 
is just a reparameterized linear ray of the form 
\begin{align}
	\exp_0(X) 
	= 
	\tfrac{\tanh(\nabs{X}_g /2)}{\nabs{X}_g/2} \, X
	= 
	\tfrac{\tanh(\nabs{X})}{\nabs{X}} \, X	
	\quad
	\text{for $X \in \TangentBundle_0\Disk \cong H$.}
	\label{eq:RadialGeodesic}
\end{align}
Thus, geodesics emanating from the origin are just reparameterized straight lines, and geodesic completeness and geodesic convexity follow directly from the transitivity of the isometry group $\Aut(\Disk,g)$ (see \eqref{sec:ShiftTransformation}) and from the fact that straight lines can intersect at most once (due to negative curvature).

The unit sphere in $H$ will be denoted by $\Sphere \ceq \set{ w \in H | \nabs{w}  = 1}$. Likewise, the unit sphere with respect to the Poincar\'e metric $g$ in the tangent space $\TangentBundle_w \Disk$ will be denoted by $\UnitTangent_w \Disk \ceq \set{ X \in \TangentBundle_w \Disk | \nabs{X}_g = 1}$.
For each $w \in \Disk$ and each $x \in \Sphere$, there is a unique unit tangent vector $V_x(w) = V(w,x) \in \UnitTangent_w \Disk$ satisfying
\begin{align*}
	\lim_{t \to \infty} \exp_w (t \, V_x(w)) = x,
\end{align*}
where $\exp \colon\! \TangentBundle\Disk \to \Disk$ denotes the Riemannian exponential map with respect to the metric $g$ and where the limit is to be interpreted as limit in the topology of $H$.
Soon, we will derive a concrete expression for $V_x(w)$ and realize that the \emph{director mapping} $V \colon \Disk \times \Sphere \to \UnitTangent \Disk$ is a smooth mapping into the sphere bundle $\UnitTangent \Disk$ (see \cref{sec:DirectorField}).
For now, we define:

\begin{definition} Let $\mu$ be a Borel probability measure on $\Sphere$. We say $w \in \Disk$ is a \emph{conformal barycenter of $\mu$} if and only if
\begin{align}
	\textstyle
	F_\mu (w)  = 0,
	\quad
	\text{where}
	\quad
	F_\mu (w) \ceq \int_\Sphere V_x(w) \, \dd \mu(x).
	\label{eq:ConformalBarycenter}
\end{align}
We call $\mu$ \emph{conformally centralized} if $F_\mu(0) = 0$.
\end{definition}

Our aim is to derive conditions on $\mu$ that ensure that $F_\mu(w)  = 0$ has a unique solution $w_*$ (in which case we are allowed to talk about \emph{the} conformal barycenter) and use Newton's method in hyperbolic space to compute $w_*$ iteratively.

\subsection{Shift Transformation}\label{sec:ShiftTransformation}

From \eqref{eq:RadialGeodesic}, it follows that $V_x(0) = \frac{1}{2} \, x$. 
In order to obtain an explicit expression for $V_x(w)$ for general $w \in \Disk$, it is useful to understand the isometry group $\Aut(\Disk,g)$ of the Riemannian manifold $(\Disk,g)$ first.
This group is characterized as the group of those M\"obius transformations of $H \cup \{\infty\}$ that map the unit ball $\Disk$ onto itself.
This is precisely the group generated by inversions in those spheres that meet $\Sphere$ perpendicularly. 
Such spheres cannot have their center on $\Sphere$. Hence each isometry $\varphi \colon \Disk \to \Disk$ induces also a unique conformal diffeomorphism $\ext \varphi \colon \Sphere \to \Sphere$.
By the construction of $V$, we have for each $\varphi \in \Aut(\Disk,g)$ that
\begin{align}
	V_{\varphi(x)}(\varphi(w)) = \dd \varphi(w) \, V_x(w)	
	\quad \text{for all $w \in \Disk$ and $x \in \Sphere$.}
	\label{eq:TransformationRulesforDirectors}
\end{align}
This induces the following transformation rule for $F$ under $\varphi \in \Aut(\Disk,g)$:
\begin{align}
	F_{\ext \varphi_\push \mu}(\varphi(w))
%	=
%	\textstyle
%	\int_\Sphere V(\varphi(w), y) \, \dd (\ext \varphi_\push \mu)(y)
%	=
%	\textstyle
%	\int_\Sphere V(\varphi(w), \ext \varphi(x)) \, \dd \mu(x)	
	=
	\dd \varphi(w)\, F_\mu(w).
	\label{eq:TransformationRulesforF}
\end{align}
If we find an isometry $\varphi$
that maps a conformal barycenter $w_*$ of $\mu$ (if existent) to $0$, 
we can construct a centralized measure $\ext \varphi_\push \mu$ as we have
$F_{\ext \varphi_\push \mu}(0) = F_{\ext \varphi_\push \mu}(\varphi(w_*))
= 	\dd \varphi(w_*)\, F_\mu(w_*) = 0$.
To this end, we associate a unique hyperbolic translation $\Shift_w \in \Aut(\Disk,g)$ with each 
$w \in \Disk$.
It is characterized by $\Shift_w(w) = 0$ and can be expressed by
\begin{align}
	\Shift \colon \Disk \times \Disk \to \Disk,
	\quad
	\Shift(w,z) 
	\ceq 
	\Shift_w(z)
	\ceq 	
	\frac{(1- \nabs{w}^2) \, z - (1 + \nabs{z}^2 - 2 \, \ninnerprod{w,z}) \, w}{1- 2 \ninnerprod{w,z} + \nabs{w}^2 \nabs{z}^2}.
	\label{eq:ShiftTrafo}
\end{align}
In the special case $H = \R^2 \cong \C$, we may rewrite $\Shift$ in complex arithmetic as
$
	\Shift(w,z) = \frac{z-w}{1- \bar w \, z}.
$
This reveals that $\Shift$ is merely an orientation-preserving version of the map $\eta$ from~\cite{MR1894482}.

As $\Disk \subset H$ is a an open set, we have a canonical identification $T_z \Disk \cong H$.
If we let $C(w,z) = (1 - \nabs{w}^2) \, (1 - 2 \, \ninnerprod{w,z} + \nabs{w}^2 \nabs{z}^2)^{-1}$, then we can easily check that $\Shift_w$ is orientation preserving as its differential is given by
\begin{equation}
	\dd \Shift_w(z)
	=
	C(w,z) \, 
	\biggparen{
	\id_H 
	- 
	2
	\,
	\frac{
		\nabs{z}^2 \, w \, w\transp + (w \, z\transp - z  \, w\transp) - 2 \, \ninnerprod{w,z} \, w \, z\transp + \nabs{w}^2  \, z  \, z\transp
	}{
		1 - 2 \, \ninnerprod{w,z}  + \nabs{w}^2 \nabs{z}^2
	}  
	}
	.
	\label{eq:pd2f}
\end{equation}
It is now straightforward to compute that $\dd  \Shift_w(z) \, (\dd  \Shift_w(z))\transp = C^2(w,z) \, \id_H$, and hence that $\dd  \Shift_w(z) $ is a similarity matrix with conformal factor $C(w,z)$.
The mapping  $\Shift_w$ is a hyperbolic translation in the sense that it maps the geodesic though $w$ and $0$ to itself; more precisely, we have
$\Shift(w,t \, w) =  (t - 1)(1 - t \, \abs{w}^2)^{-1} \, w$.
Its inverse has the simple form $(\Shift_w)^{-1} = \Shift_{(-w)}$.
For each linear subspace $E \subset H$ and for $w \in E \cap \Disk$, the transformation $\Shift_w$ maps $E \cap \Disk$ to itself.
As for all elements of $\Aut(\Disk,g)$, $\Shift_w$ can be extended to $\Sphere$, and we have
\begin{align}
	\ext \Shift_w(x) \ceq 
	\Shift(w,x)
	\ceq
	\lim_{z \to x} \Shift(w,z)
	=
	\frac{(1- \nabs{w}^2) \, x - 2(1 - \ninnerprod{w,x}) \, w}{1- 2 \, \ninnerprod{w,x} + \nabs{w}^2},
	\label{eq:SphereTrafo}
\end{align}
where the limit is taken with respect the topology of $H$. A short calculation reveals that $\ext \Shift_w$ indeed maps $\Sphere$ to $\Sphere$. There is a nice geometric interpretation of~\eqref{eq:SphereTrafo}: if we extend the (Euclidean) line from $x$ through $w$ until it strikes $\Sphere$ at some point $p$, then $\Shift_w(x) = -p$. We will leave the proof to the reader (see also \autoref{fig:conformalbary}, (c)), but we now use this observation to establish:

\begin{lemma}\label{lem:shift distance}
For $w_1$, $w_2 \in \Disk$, $x \in \Sphere$, we have 
$d_\Sphere(\Shift_{w_1}(x),\Shift_{w_2}(x)) \leq 2\, d_g(w_1,w_2)$.
\end{lemma}

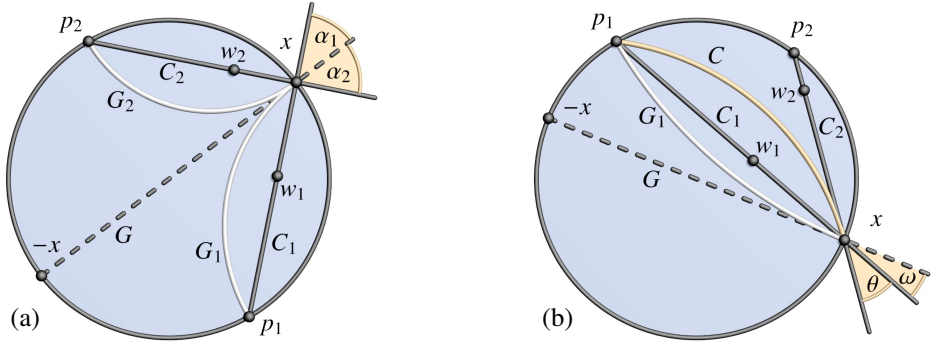
\begin{figure}
\begin{center}
	\capstart %ensures that hyperlink will jump to the top of this image
\newcommand{\inc}[2]{\begin{tikzpicture}
    \node[inner sep=0pt] (fig) at (0,0) {\includegraphics{#1}};
	\node[above right= 0.ex] at (fig.south west) {\begin{footnotesize}(#2)\end{footnotesize}};    
\end{tikzpicture}}%
\presetkeys{Gin}{
	trim = 0 100 10 60, 
	clip = true,  
	width = 0.333\textwidth
}{}%	
	\hfill		
	\inc{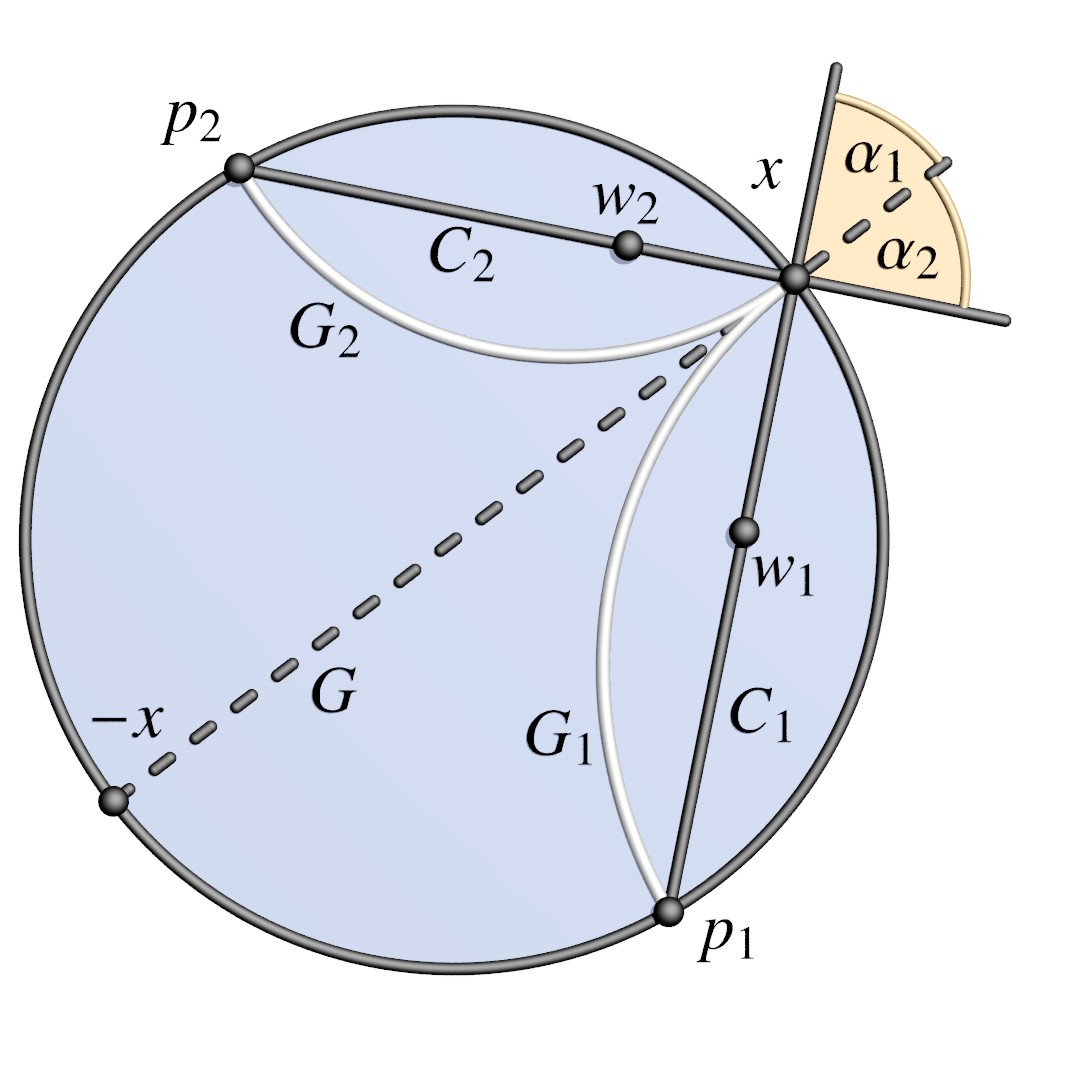}{a}%	
	\hfill		
	\inc{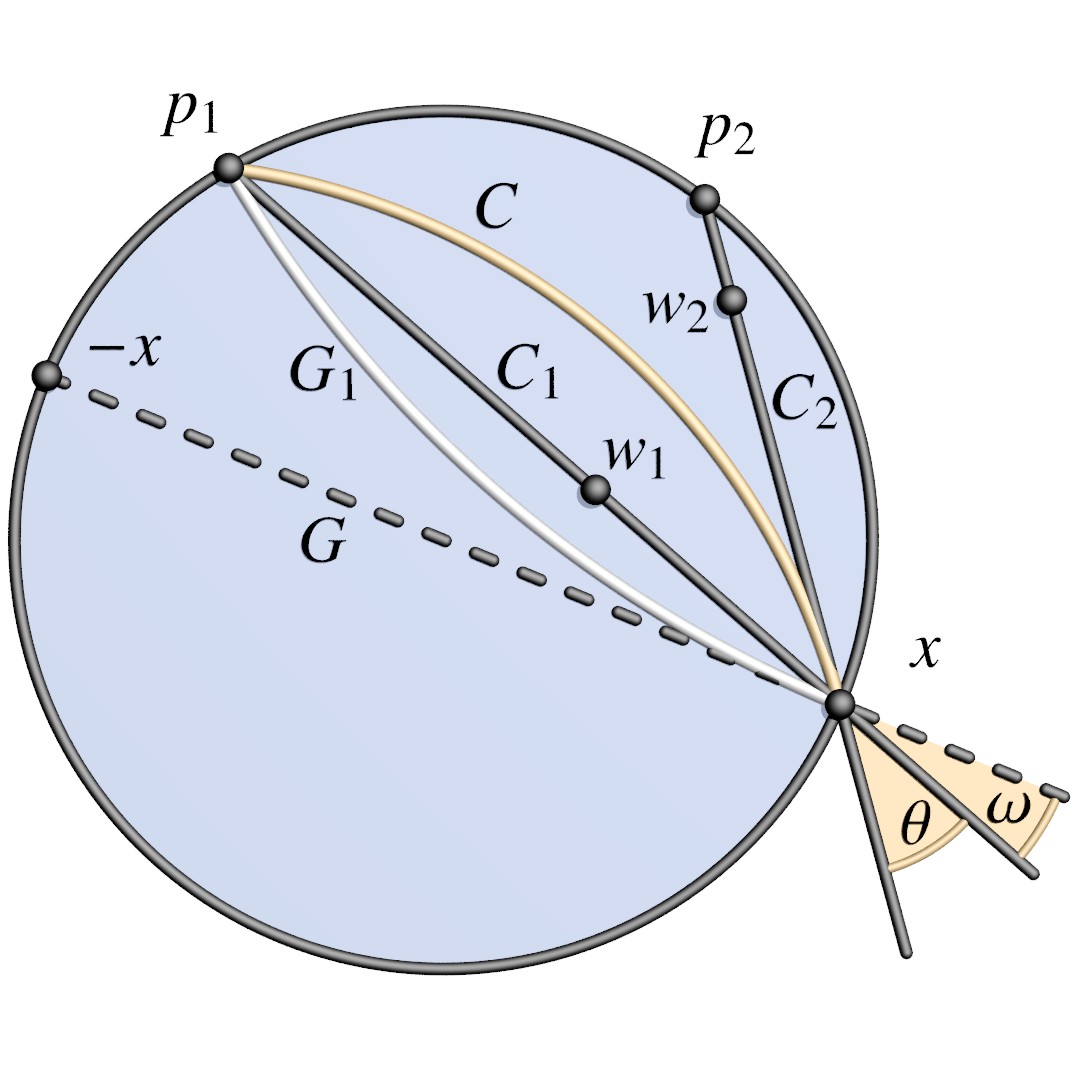}{b}%		
	\hfill{}
	\caption{
	(a)	
	The secants $C_1$ and $C_2$ are hypercycles with axes $G_1$ and $G_2$, respectively.
	(b)
	The hypercycle $C$ with axis $G_1$ that is tangent to $C_2$ meets $C_1$ at the angle $\theta =\theta(x)$.
	}
\label{fig:ProofSketch}
\end{center}
\end{figure}

\begin{proof}
We may focus our attention to the linear space spanned by $w_1$, $w_2$, and $x$,
so without loss of generality, we may assume that $H$ is three-dimensional.
Let $C_i$ denote the Euclidean secant through $x$ containing $w_i$ and let
$p_i$ be the other end point of  $C_i$. The plane containing $w_1$, $w_2$ and $x$
cuts $\Sphere$ in a circle with center $z$ and radius $r \leq 1$. 
From our observation above, we know that $\Shift_{w_i}(x) = -p_i$,
so the central angle theorem implies
\begin{align*}
	d_\Sphere(\Shift_{w_1}(x),\Shift_{w_2}(x))	
	&=
	d_\Sphere(-p_1,-p_2)
	= 
	d_\Sphere(p_1,p_2) 	\\
	&\leq 
	r \angle(p_1,z,p_2) 
	\leq
	\angle(p_1,z,p_2)
	= 
	2 \, \angle(p_1,x, p_2) = 2 \, \angle(w_1,x, w_2),
\end{align*}
leaving us to show $\angle(w_1,x, w_2) \leq d_g(w_1,w_2)$.

It suffices to do so for the maximizer $x = x_*$ of 
$\theta(x) \ceq \angle (w_1 ,x, w_2)$ on $\Sphere$.
Observe that $\theta$ is continuous on $\Sphere$ (the only discontinuities of $\theta$ are $w_1$ and $w_2$. So by compactness, a global maximizer  $x_*$ of $\theta$ on $\Sphere$ must exist.
Next we show that $w_1$, $w_2$, $x_*$, and $0$ are coplanar:
Denote the (Euclidean) straight line through $w_1$ and $w_2$ by $L$.
The minimizers of $\theta$ on $\Sphere$ are precisely the two points in $L \cap \Sphere$, and $\theta$ is smooth away from $L$.
Hence  $x_*$ must be a critical point of $\theta|_\Sphere$. This means that
the surface normals $\nu_\Sphere(x_*) = x_*$ of $\Sphere$ and $\nu_\varSigma(x_*)$ of the levelset $\varSigma$ of $\theta$ at the point $x_*$ are colinear.
Levelsets of $\theta$ are surfaces of revolution about $L$, and so $\nu_\varSigma(x_*)$ is in the plane spanned by $x_*$, $w_1$ and $w_2$. By colinearity, this means that $\nu_\Sphere(x_*) = x_*$ is also in this plane, which must then pass through the origin.

Hence it suffices to show $\theta(x) \leq d_g(w_1,w_2)$ in the case that $H$ is the Euclidean plane.
Denote the diameter from $x$ to $-x$ by $G$.
Then there are two cases:
\newline
\textbf{Case 1:}
\emph{$G$ lies between $C_1$ and $C_2$.}
Denote the geodesic connecting $x$ and $p_i$ by $G_i$ and the angle at $x$ that is enclosed by $G_i$ and $C_i$ by $\alpha_i$ (see \autoref{fig:ProofSketch} (a)).
Each $C_i$ is a \emph{hypercycle} or \emph{equidistant} curve from the corresponding geodesic $G_i$. This means that 
each point $w \in C_i$ has the same minimal distance $\dist_g(w_i,G_i) = f(\alpha_i)$ to $G_i$ (see \cite[Chapter 14, Theorem C]{MR620163}). It is known that $f(\alpha) = \operatorname{arsinh}(\tan \alpha)$~(see \cite[p.\ 24]{MR1695450}). Because $G$ lies between $C_1$ and $C_2$, $G_1$ and $G_2$ also lie between $C_1$ and $C_2$, and we have
\begin{align*}
	d_g(w_1,w_2)
	&\geq 
	\dist_g(w_1,G_1) + \dist_g(w_2,G_2)
	=
	f (\alpha_1) +  f (\alpha_2)
	.
\end{align*}
Since $f(0) = 0$ and $f'(\alpha) = \sec \alpha \geq 1$ we have $f(\alpha) \geq \alpha$ for all $\alpha \in \intervalco{0, \uppi/2}$. Thus in particular, we have
$f(\alpha_1) +  f (\alpha_2) \geq \alpha_1 + \alpha_2 = \theta(x)$.
\newline
\textbf{Case 2:}
\emph{$G$ does not lie between $C_1$ and $C_2$.}
Without loss of generality, we may assume that $C_1$ lies between $G$ and $C_2$ (see \autoref{fig:ProofSketch} (b)).
Denote the angle at $x$ between $G$ and $C_1$ by $\omega$. Observe that $C_1$ lies also between $G_1$ and $C_2$.
Let $C$ be the circle through $x$ and $p_1$ that is tangent to $C_2$ at $x$.
Again, $C$ is equidistant from $G_1$, meeting $G_1$ at the angle $\omega + \theta(x)$.
Moreover, $C$ is equidistant from $C_1$ as well (see again \cite[Chapter 14, Theorem C]{MR620163}).
Hence we have
\begin{align*}
	d_g(w_1,w_2)
	&\geq
	\dist_g(w_1,C)
	=
	\dist_g(C,G_1)	
	-	
	\dist_g(C_1,G_1) 
	=	
	f (\omega + \theta(x))
	-
	f(\omega)
	.
\end{align*}
Since $f' \geq 1$, by the mean value theorem
$f(\omega + \theta(x)) - f(\omega) \geq
(\omega + \theta(x)) - \omega =  \theta(x)$.s
\end{proof}

\subsection{Director Field}\label{sec:DirectorField}

Utilizing the family of isometries $\Shift$ and the identity \eqref{eq:TransformationRulesforDirectors}, we may compute $V_x(w)$ for arbitrary $w \in \Disk$ by
first moving $w$ to $0$ by the shift transformation $\Shift_w$ (and by moving $x$ along with $\ext \Shift_w$), by evaluating $V$ there, and by transporting the result back to $w$ with $\dd \Shift_{(-w)}(0)$:
\begin{align*}
	V_x(w)
	&= \dd \Shift_{(-w)}(0) \cdot  V(0,\ext \Shift_w(x))
%	= \tfrac{1}{2} \, \dd \Shift_{(-w)}(0) \cdot   \ext \Shift_w(x) 
%	=  C(-w,0)\, \bigparen{ \tfrac{1}{2} \, \ext \Shift_w(x) }
	= \tfrac{1}{2} \, (1 - \nabs{w}^2) \, \ext \Shift_w(x)
	,
\end{align*}
where we used~\eqref{eq:pd2f}. Combined with~\eqref{eq:ShiftTrafo}, we obtain the explicit expression
\begin{align}
	V_x(w) 
	&= 
	\frac{1}{2}
	\frac{
		\nparen{1 - \nabs{w}^2}^2 \, x - 2 \, \nparen{1 - \nabs{w}^2} \, (1 - \ninnerprod{w,x}) \, w		
	}{
		1 - 2 \ninnerprod{w,x} + \nabs{w}^2
	}.
\label{eq:vxw}
\end{align}	
Because of
$
	1 - 2 \ninnerprod{w,x} + \nabs{w}^2 = \nabs{x - w}^2 >0
$
for all $w \in \Disk$ and $x \in \Sphere$, the mapping $V \colon \Disk \times \Sphere \to \UnitTangent \Sphere$ is smooth.

\subsection{Approximating conformally centralized measures}

We have already seen above that shifting $\mu$ by one of its conformal barycenters $w_*$ 
leads to the centered measure $(\ext \Shift_{w_*})_\push \mu$. 
We now show that a good approximation $w$ of $w_*$ will also lead to a good approximation $(\ext \Shift_{w})_\push \mu$ of $(\ext \Shift_{w_*})_\push \mu$. We recall a definition first.

\begin{definition}
\label{def:wasserstein}
Given two probability measures $\mu_1$ and $\mu_2$ on a space $X$, a \emph{transport plan} $\gamma$ between the $\mu_i$ is a probability measure $\gamma$ on $X \times X$ with marginals equal to the $\mu_i$. The \emph{Wasserstein distance $W_p$} is defined by
\begin{align*}
	W_p(\mu_1, \mu_2)
	=
	\begin{cases}
		\;
		\displaystyle
		\inf_{\gamma}
		\bigparen{ \textstyle \iint_{\Sphere \times \Sphere} d_\Sphere(y,z)^p \, \dd \gamma(y,z)}^{1/p},
		& \text{for $1 \leq p < \infty$,}
		\\
		\;
		\displaystyle		
		\inf_{\gamma}
		\esssup_{(y,z) \in \supp (\gamma)} d_\Sphere(y,z),
		& \text{for $p = \infty$,}		
	\end{cases}
\end{align*}
where in each case the infimum runs over all transport plans $\gamma$ between $\mu_1$ and $\mu_2$. 
\end{definition}

We can now estimate the Wasserstein distance between shifts of a measure $\mu$:

\begin{lemma}\label{lem:WassersteinDistance}
Let $\mu$ be a Borel probability measure on $\Sphere$.
Then for all $w_1$ and $w_2 \in \Disk$, we have the following estimate for the $p$-Wasserstein distance with respect to the angular distance function $d_\Sphere$ on $\Sphere$:
\begin{align}
	W_p((\ext \Shift_{w_1})_\push \mu,(\ext \Shift_{w_2})_\push \mu) \leq 2 \, d_g(w_1, w_2)
	\quad
	\text{for all $p \in \intervalcc{1,\infty}$}.
	\label{eq:WassersteinDistance}
\end{align}
\end{lemma}
\begin{proof}
We abbreviate $\psi_i \ceq \ext \Shift_{w_i}$ and $\mu_i \ceq (\psi_i)_\push \mu =(\ext \Shift_{w_i})_\push \mu$, and we denote the projection onto  the $i$-th Cartesian factor of $\Sphere \times \Sphere$ by $\pi_i \colon  \Sphere \times \Sphere \to \Sphere$. If we let $\varPhi(x) = (\psi_1(x), \psi_2(x))$ and write $\nu = \varPhi_\push \mu$, then $(\pi_i)_\push \nu = (\pi_i)_\push \varPhi_\push \mu = (\pi_i \circ \varPhi)_\push \mu = (\psi_i)_\push \mu = \mu_i$. Thus $\nu$ is a transport plan between the $\mu_i$ and
\begin{align*}
	W_p(\mu_1, \mu_2)
	\leq
	W_\infty(\mu_1, \mu_2)
	\leq
	\esssup_{(y,z) \in \supp (\nu)} d_\Sphere(y,z)
	=
	\sup_{x \in \Sphere} d_\Sphere(\psi_1(x), \psi_2(x) ),
\end{align*}
and the result then follows from Lemma~\ref{lem:shift distance}.
\end{proof}

\section{Newton's Method with fixed step size}
\label{sec:newton with fixed stepsize}

We are now going to use Newton's method with fixed step size in the Riemannian manifold $\Disk$ with metric $g$ (c.f. \cite[Algorithm 1.2]{MR1895088}) to solve the equation \eqref{eq:ConformalBarycenter}.

\subsection{Basic algorithm}

Starting with an initial guess $w_0 \in \Disk$, we generate a sequence of $w_k \in \Disk$~by
\begin{align}
	v_{k} \ceq - \nabla F_\mu(w_k)^{-1} F_\mu(w_k)
	\qand
	w_{k+1} \ceq \exp_{w_k}(v_{k}).
	\label{eq:Newton1}	
\end{align}
In order to derive an expression for $\nabla F_\mu$, we first compute the covariant derivative $\nabla V_x$ with respect to the metric $g$. 
It helps to recall that if we have conformally equivalent Riemannian metrics $g(X,Y) = f(w) \, \bar g (X,Y)$ with conformal factor $f(w) >0$, then the corresponding Riemannian connections $\nabla$ and $\bar \nabla$ are related by 
\begin{equation}
	\nabla_X Y 
	= 
	\bar \nabla_X Y + (2 \,f)^{-1} \bigparen{ 
		(X f) \, Y + (Y f) \, X - \bar g(X,Y) \, \operatorname{grad}_{\bar g} f 
	},
\end{equation}
where $\operatorname{grad}_{\bar g} f$ obeys $X f = \bar g(X,\operatorname{grad}_{\bar g} f)$ (cf.~\cite{MR1138207}, p.\ 181). 
By \eqref{eq:metric}, our  metric $g$ is conformally equivalent to the metric $\bar g = \ninnerprod{\cdot,\cdot}$ on $\Disk$ with conformal factor $f(w) = 4 \, (1 - \nabs{w}^2)^{-2}$.
Since the covariant derivative $\bar \nabla$ coincides with the Fréchet derivative $D$, this (after a short computation) yields
\begin{equation}
	\nabla_X Y(w) 
	= 
	DY(w) \, X(w) 
	+ 
	\tfrac{2}{1-\nabs{w}^2} \, \bigparen{
		\ninnerprod{X,w} \, Y 
		+ \ninnerprod{Y,w} \, X 
		- \ninnerprod{X,Y} \, w 
	}.
\label{eq:covariant derivative}
\end{equation}
Notice that the two covariant derivatives coincide at $w=0$. Differentiating~\eqref{eq:vxw}, it is straight-forward to compute
\begin{align*}
	\nabla V_x (0)
	= D V_x(0)
	= x \, x\transp - \id_H
	= V_x(0) \otimes V_x(0)^\flat - \id_{\TangentBundle_0 \Disk}.
\end{align*}
To compute $\nabla V_x(w)$, we can either pull back the computation to $w=0$ using covariance or use~\eqref{eq:covariant derivative} directly at $w$ and simplify. Either way, we obtain
\begin{align}
	\nabla V_x = V_x \otimes V_x^\flat - \id_{\TangentBundle \Disk}.
	\label{eq:NablaV}
\end{align}
Differentiating~\eqref{eq:ConformalBarycenter} under the integral sign, we derive from~\eqref{eq:NablaV} that
\begin{align}
	\textstyle
	\nabla F_\mu
	= 
	\int_{x \in \Sphere} \bigparen{ V_x \otimes V_x^\flat} \, \dd \mu(x) - \id_{\TangentBundle_w \Disk}
	.
\label{eq:F and nabla F}
\end{align}

\subsection{Shifted algorithm}
When performing \eqref{eq:Newton1} by computing \eqref{eq:vxw}, \eqref{eq:NablaV}, and $\exp_{w_k}(v_k)$
in machine arithmetic, one is frequently confronted with catastrophic loss of precision. 
We can avoid this by using an idea of Abikoff and Ye~\cite{MR1476978}: shift everything at point $w_k$ to the Euclidean origin, perform the Newton update there, and shift everything back to $w_k$. 
This gives us:
\begin{align}
	\mu_k \ceq (\ext \Shift_{w_k})_\push \mu,
	\;\;
	u_{k} &\ceq - \nabla F_{\mu_k}(0)^{-1} F_{\mu_k}(0),
	\;\;\text{and}\;\;
	w_{k+1} = \Shift(-w_k,\exp_0(u_k)).
	\label{eq:Newton2}	
\end{align}
Indeed, with $\varphi_k \ceq \Shift_{(-w_k)}$ and by \eqref{eq:TransformationRulesforF}, the search directions $v_k$ and $u_k$ are related as follows:
\begin{align*}
	v_k 
	&= - \bigparen{\nabla F_{\mu}^{-1}\, F_{\mu} }\at_{w_k}
	= - \bigparen{
		\nparen{ \nabla F_{(\ext \varphi_k)_\push \mu_k} \circ \varphi_k}^{-1}
		\, 
		\nparen{F_{(\ext \varphi_k)_\push \mu_k} \circ \varphi_k}	
	} \at_0
	\\
	&= - \bigparen{
		\nparen{\dd \varphi_k\, \nabla F_{\mu_k} \, \dd \varphi_k^{-1}}^{-1} 
		\, 
		\dd \varphi_k \, F_{\mu_k}
	}\at_0
	= - \bigparen{
		\dd \varphi_k \,  \nabla F_{\mu_k}^{-1} \,  F_{\mu_k}
	}\at_0
	= \dd \varphi_k(0) \,  u_k.	
\label{eq:vkReconstructionold}	
\end{align*}
Together with $\exp_{w_k}(v_k) = \exp_{\varphi_k(0)}( \dd \varphi_k \,u_k) = \varphi_k(\exp_0(u_k))$,
this shows that both \eqref{eq:Newton1} and \eqref{eq:Newton2}	produce the same new iterate $w_{k+1}$---at least in exact arithmetic.

Computing $F_{\mu_k}(0)$ and $\nabla F_{\mu_k}(0)$ is particularly easy as we have
\begin{align*}
	\textstyle
	F_{\mu_k}(0) = 
	\frac{1}{2} \int_\Sphere x \, \dd \mu_k(x)
	\qand
	\nabla F_{\mu_k}(0)
	=
	DF_{\mu_k}(0)
	= 
	\int_\Sphere  (x \, x\transp)  \, \dd \mu_k(x) - \id_H.
	\label{eq:NablaF0}
\end{align*}
The exponential map $\exp_0 \colon\! \TangentBundle_0 \Disk \to \Disk$ can be computed according to \eqref{eq:RadialGeodesic} by
\begin{align*}
	\exp_0(u_k)
	= \tfrac{\tanh(\nabs{u_k})}{\nabs{u_k}} \, u_k
%	= \frac{\tanh(\nabs{u_k}_g/2)}{\nabs{u_k}_g/2} \, u_k
	.
\end{align*}
 
Up to this point, the measure $\mu$ and the Hilbert space $H$ have been arbitrary. We could have given~$\mu$ by any density supported on $\Sphere$, 
computed $F_{\mu_k}(0)$ and $\nabla F_{\mu_k}(0)$
by numerical integration, and  computed $\mu_{k+1}$ by the transformation formula for measures.
However, for finite, discrete measures, the pushforward is much easier to compute as we see in the following example:

\begin{example}\label{ex:DiscreteMeasure}

Let $H = \R^\AmbDim$ be a finite-dimensional Euclidean space and let $\mu$ be given as a linear combination of Dirac measures 
$\mu = \sum_{i=1}^n \omega_i \, \updelta(x_{i})$ with $x_{i} \in \Sphere^{\AmbDim-1}$, $\omega_i \geq 0$, and $\sum_{i=1}^n \omega_i =1$. 
Then, noting that at $w=0$, we have $V_{x}(w) = \frac{1}{2} x$ and $u^\flat(v) = g(u,v) = 4 \, \ninnerprod{u,v}_{\R^\AmbDim} = 4 \, u\transp \, v$, we get
\begin{align*}
	\textstyle
	F_{\mu}(0) = 
	\frac{1}{2} \sum_{i=1}^n x_{i} \, \omega_i
	\qand
	\nabla F_{\mu}(0)
	= 
	\bigparen{\sum_{i=1}^n (x_{i} \,x_{i}\transp) \, \omega_i } -\id_{\R^\AmbDim}.
\end{align*}
The push-forward $(\Shift_{s})_\pull \, \mu$ of the measure $\mu$ has the same weights $\omega_i$, but the locations of its Dirac measures are transformed by $\Shift_s$. It can be written as
\begin{align*}
	\textstyle
	(\Shift_s)_{\pull} \, \mu
	=
	\sum_{i=1}^n \omega_i \, \updelta(\Shift(s,x_{i})).
\end{align*}
\end{example}

\subsection{Statement of first convergence theorem}

We will see now that Ferreira and Svaiter's Kantorovich theorem on Riemannian manifolds~\cite[Theorem~3.2]{MR1895088} (see also \autoref{theo:KantorovichRiemannian} below) immediately yields a useful result:

\begin{theorem}\label{thm:quadratic convergence}
Denote the smallest eigenvalue of $(-\nabla F_\mu(w_0))$ by  $\lambda_{\min}$ and suppose it is greater then $0$. 
Let $w_k$ be the iterates of the usual Newton iteration~\eqref{eq:Newton1} 
or the ones obtained from the shifted Newton iteration~\eqref{eq:Newton2}.

Suppose that the following \emph{Newton-Kantotovich condition} is satisfied:
\begin{align}
	q \ceq 4 \, \nabs{F_\mu(w_0)}_g/\lambda_{\min}^2  < 1.
	\label{eq:NKCondition}
\end{align}
Then there exists a conformal barycenter $w_*(\mu)$ of $\mu$,
the measure $\mu_* \ceq (\ext \Shift_{w_*(\mu)})_\push \mu$ satisfies $\int_\Sphere \mu_* =0$,
and the iterates $w_k$ and $\mu_k$ converge with
\begin{equation}
	d_g(w_k , w_*(\mu))
	\leq  
	\tfrac{1}{2}\, \lambda_{\min} \, q^{(2^k)}
	\qand
	W_p(\mu_k , \mu_*)
	\leq  
	\lambda_{\min} \, q^{(2^k)},
	\label{eq:ConvergenceRate}
\end{equation}
where $W_p$ denotes the $p$-Wasserstein distance for $p \in \intervalcc{1,\infty}$.
\end{theorem}

The following is a direct consequence of this theorem:
\begin{corollary}\label{cor:time bound under NK conditions}
Let $H = \R^{\AmbDim}$ and denote the Euclidean norm by $\nabs{\cdot}$.
Suppose $\mu = \sum_{i=1}^n \omega_i \, \updelta(x_i)$ with $\omega_i > 0$, $\sum_{i=1}^n \omega_i = 1$.
Denote by $\lambda_{\min}$ the smallest eigenvalue of the matrix
$A \ceq \id_{\R^{\AmbDim}} - \sum_{i=1}^n (x_{i} \, x_{i}\transp) \, \omega_i$
and by $w_{\operatorname{cm}} = \sum_{i=1}^n \omega_i \, x_i$ the center of mass of $\mu$.
Suppose that $q \ceq 4 \, \nabs{w_{\operatorname{cm}}}/\lambda_{\min}^2 < 1$ and $w_0 \ceq 0$.
Then for any $\varepsilon > 0$, 
each of the algorithms \eqref{eq:Newton1}~and~\eqref{eq:Newton2} 
reduces
the hyperbolic distance between $w_k$ and the conformal baycenter $w_*(\mu)$ 
and 
the $p$-Wasserstein distance between $\mu_k$ and the centralized measure $\mu_*$
to less than $\varepsilon$ in at most $k \leq \lceil \log_2 \nabs{\log_2 \varepsilon} - \log_2 \nabs{\log_2 q} \rceil$ iterations.
For this, they require $O(k\,(d^2 n + d^3))$ time and $O(d \, n+d^2)$ memory.
\end{corollary}
\begin{proof} 
The first part of the corollary comes from plugging definitions into \cref{thm:quadratic convergence} for $w_0 = 0$: 
We observe that
$F_\mu(0) = \tfrac{1}{2} \, w_{\operatorname{cm}}$,
$\nabs{F_\mu(0)}_g = 2 \, \nabs{F_\mu(0)} = \nabs{w_{\operatorname{cm}}}$,
and
$-\nabla F_\mu(0) = \id_{\R^{\AmbDim}} - \sum_{i=1}^n (x_{i} \, x_{i}\transp)$.
Hence \cref{thm:quadratic convergence} implies the error bounds \eqref{eq:ConvergenceRate}.
By taking logarithms, we see that $\lambda_{\min} \, q^{(2^k)} < \varepsilon$ as soon as $k \geq \log_2 \nabs{\log_2 (\varepsilon)} - \log_2 \nabs{\log_2 (q)}$.

In each iteration, we have (i)~to compute either $F_\mu(w_k)$ or $\mu_k$ and $F_{\mu_k}(0)$ which both take $O(d \, n)$ time;
(ii)~to compute $\nabla F_\mu(w_k)$ or $\nabla F_{\mu_k}(0)$ which both take $O(d^2 \, n)$ time;
(iii)~solve a linear equation with a matrix of size $d \times d$ which can be performed in $O(d^3)$ time.
We require $O(d\,n)$ memory for storing $\mu$ and $\mu_k$ and
$O(d^2)$ memory to store the $d \times d$ matrix and and for computing its inverse.
\end{proof}

We note that for $\varepsilon = 10^{-16}$ and $q = 0.99$, at most 12 iterations will be required, and for $\varepsilon = 10^{-16}$ and $q = 0.5$, 6 iterations will suffice. 
Thus, the Newton iteration \eqref{eq:Newton1} and the (shifted) Newton iteration~\eqref{eq:Newton2}
are quite efficient when our hypotheses hold. However, we have to point out that accuracy this high
is not obtainable in floating point arithmetic if the conformal barycenter lies close to the boundary of $\Disk$; this is due to rapid precision loss in the computation of $\sigma(w,x)$ as $w \to x \in \Sphere$.

\subsection{Proof of first convergence theorem; Newton-Kantorovich theorem on manifolds}

We now turn to the proof of \cref{thm:quadratic convergence}. We first state a (slightly modified) version of the Kantorovich theorem on a Riemannian manifold $(M,g)$~\cite[Theorem~3.2]{MR1895088}.
The proof can by copied from~\cite{MR1895088} almost word-by-word, but instead of relying on the Hopf-Rinow theorem (which is false for infinite-dimensional Riemannian manifolds, see \cite{MR400283}), we state the required completeness conditions explicitly, but in localized form. We will use $\nabs{\cdot}_g$ for the norm given by $g$ for vectors on $T_x M$ and $\nnorm{\cdot}_g$ for the corresponding operator norm\footnote{For a linear operator $A \colon T_x M \to T_xM$, it is given by $\nnorm{A}_g = \sup_{u \in T_x M\setminus \{0\}} \nabs{A \,u}_g/ \nabs{u}_g$.}  for linear maps $A \!:\! T_x M \rightarrow T_x M$. 

\begin{theorem}\label{theo:KantorovichRiemannian}
Let $(M,g)$ be a (not necessarily finite-dimensional) Riemannian manifold,
let $U \subset M$ be an open set such that 
$U$ is geodesically convex\footnote{This means that every two points $x$, $y \in U$ are connected by a unique length-minimizing geodesic that is contained in $U$. In particular that guarantees that
for every open ball $\OpenBall{x}{\varrho} \subset U$, the Riemannian exponential map $\exp_x \colon \set{u \in T_x M | \nabs{u}_g < \varrho} \to U$ is a well defined diffeomorphism onto $\OpenBall{x}{\varrho}$. 
%This is crucial for the well-definedness of Newton updates.
} 
and 
$\bar U$ is complete with respect to the geodesic distance $d_g$.	
Let $F \in C^{1}(\bar U;TM)$ be a continuously differentiable vector field
and
let $w_0 \in U$ be a given point where the covariant derivative $\nabla F(w_0)  \colon \TangentBundle_{w_0}M \to \TangentBundle_{w_0}M $ is invertible and its inverse is a bounded operator on the Hilbert space $\TangentBundle_{w_0} M$.
Suppose that there are constants $a> 0$, $b \geq 0$, and $L \geq 0$ such that
\begin{equation}
\nnorm{\nabla F(w_0)^{-1}}_g \leq a
\quad\text{and}\quad
\nabs{\nabla F(w_0)^{-1}F(w_0)}_g\leq b
\label{eq:Kantorovich starting conditions}
\end{equation}
and $\nabla F$ obeys the Lipschitz condition that for each two $z$, $w \in U$, one has
\begin{equation}
\nnorm{P_\gamma (\nabla F(w)) - \nabla F(z)}_g \leq L \, d_g(w,z),
\label{eq:Kantorovich Lipschitz condition}
\end{equation}
where $P_\gamma \colon \End(\TangentBundle_w \Disk) \to \End(\TangentBundle_z \Disk)$ is the parallel transport along the (unique) minimizing geodesic $\gamma$ from $w$ to $z$ in~$U$. 
Further, suppose we define the auxiliary constants
\begin{equation*}
	r \ceq \tfrac{1}{a \, L}, 
	\quad
	q \ceq 2 \, a \, b \, L < 1,
	\qand
	r_{-} \ceq r \, \bigparen{1 - \sqrt{1-q}}
	,
\end{equation*}
and that we have $\OpenBall{w_0}{r_-} \subset U$.

Then the sequence of Newton iterates
$
	w_{k+1} \ceq \exp_{w_{k}} \bigparen{-\nabla F(w_{k})^{-1} F(w_{k})}
$
is well-defined and contained in $\OpenBall{w_0}{r_-}$ for all $k$. 
They converge to some $w_* \in \ClosedBall{w_0}{r_-}$ with $F(w_*) = 0$ and
\begin{align*}
	\textstyle
	d_g(w_k,w_*)
	\leq 
	\frac{2 \, b}{q} \, q^{(2^k)}
	=\frac{1}{a \, L} \, q^{(2^k)}.
\end{align*}
Last, the point $w_*$ is the unique zero of $F$ in $\ClosedBall{w_0}{r_-}$.
\end{theorem}

To apply~\cref{theo:KantorovichRiemannian} to $F_\mu$, we start by proving that $\nabla F_\mu$ obeys indeed a Lipschitz condition with $L =2$.
\begin{proposition}\label{prop:nablaF_LipschitzContinuity}\label{prop:L is two}
For any $w$, $z \in \Disk$, let $\gamma \colon \intervalcc{0,1} \to \Disk$ be the unique minimizing geodesic joining $z$ and $w$ with $\gamma(0) = z$ and $\gamma(1) = w$ and let $P_{\gamma,a,b} \colon \End(\TangentBundle_{\gamma(a)} \Disk) \to \End(\TangentBundle_{\gamma(b)}\Disk) $ be the parallel transport along $\gamma$ from $\gamma(a)$ to $\gamma(b)$. Then
\begin{equation*}
	\textstyle
	\nnorm{P_{\gamma,1,0} (\nabla F_{\mu}(w))- \nabla F_{\mu}(z)}_g 
	\leq 
	2 \, d_g(z,w).
\end{equation*}
\end{proposition}
\begin{proof}
By the fundamental theorem of calculus and the properties of parallel transport, we have for a general tensor field $A\in C^1(\Disk;\End(\TangentBundle \Disk))$ that
\begin{align*}
	\textstyle
	P_{\gamma,1,0} (A(w)) - A(z) 
	&=
	\textstyle	
	\int_0^1 \frac{\dd}{\dd t} \bigparen{P_{\gamma,t,0} (A(\gamma(t)))} \, \dd t
	=
	\int_0^1 P_{\gamma,t,0} \, (\nabla_{\gamma'(t)} A) \, \dd t.
\end{align*}
Since  $P_{\gamma,t,0}$ is an isometry, we obtain
\begin{align*}
	\nnorm{P_{\gamma,1,0} (A(w)) - A(z) }_g
	&\leq
	\textstyle	
	\int_0^1 \nnorm{P_{\gamma,t,0} \, (\nabla_{\gamma'(t)} A)}_g \, \dd t	
	=
	\int_0^1 \nnorm{\nabla_{\gamma'(t)} A}_g \, \dd t		
	\\
	&
	\textstyle	
	\leq
	\int_0^1 \nnorm{\nabla A(\gamma(t))}_g \, \nabs{\gamma'(t)}_g \, \dd t		
	\leq 
	\nnorm{\nabla A}_{L^\infty_g} \, d_g(z,w),
\end{align*}
where $\nnorm{\cdot}_{L^\infty_g}$ denotes the supremum-norm with respect to $g$.
Thus, for $A = \nabla F_\mu$, it suffices to show that its covariant derivative $\nabla \nabla F_\mu$ is uniformly bounded by $2$. 
Since $\nnorm{\nabla \nabla F_\mu(w)}_g \leq \int_\Sphere \nnorm{\nabla \nabla V_x(w)}_g \, \dd \mu \leq \nnorm{ \nabla \nabla  V_x}_{L^\infty_g}$,
it suffices to show that $\nabla \nabla  V_x$ is uniformly bounded by~$2$ which is what we do next.
Abbreviate $\alpha_x \ceq V_x^\flat$. 
Because $\flat$ is covariantly constant and isometric, we have $\nnorm{\nabla\nabla \alpha_x}_{L^\infty_g}  = \nnorm{\nabla\nabla V_x}_{L^\infty_g}$.
We deduce from \eqref{eq:NablaV} that $\nabla \alpha_x = \alpha_x \otimes \alpha_x - g$. 
The Leibniz rule yields
\begin{align*}
	\nabla_X \nabla \alpha_x
	&=
	\nabla_X \paren{\alpha_x \otimes \alpha_x - g}
	=
	(\nabla_X \alpha_x) \otimes \alpha_x
	+
	\alpha_x \otimes (\nabla_X  \alpha_x)
	\\
	&=	
	(\alpha_x(X) \, \alpha_x - g(X,\cdot)) \otimes \alpha_x
	+
	\alpha_x \otimes (\alpha_x(X) \, \alpha_x - g(X,\cdot))
	.
\end{align*}
So with respect to the operator norm on $T^*\Disk \otimes T^*\Disk \otimes T^*\Disk$ given by
\begin{align}
	\nnorm{B}_g \ceq 
	\sup_{X,\,Y,\, Z \in T_w\Disk \setminus \set{0} } 
	\frac{\nabs{B(X,Y,Z)}}{\nabs{X}_g \, \nabs{Y}_g\, \nabs{Z}_g}
	\quad
	\text{for}
	\quad
	B \in T_w^*\Disk \otimes T_w^*\Disk \otimes T_w^*\Disk,
	\label{eq:OperatorNorm}
\end{align}
we have the following uniform bound for the second derivative of the director field:
\begin{align*}
	\nnorm{\nabla\nabla V_x}_{L^\infty_g} 
	=
	\nnorm{\nabla \nabla \alpha_x}_{L^\infty_g}
	&\leq 
	{\textstyle
		2 \, \nnorm{\alpha_x}_g \, \nnorm{g - \alpha_x \, \otimes \alpha_x}_g
	}
	\leq 
	2.
\end{align*}
\end{proof}

We are now ready to prove \cref{thm:quadratic convergence}. 
\begin{proof}
To apply \cref{theo:KantorovichRiemannian}, we must find $a$, $b$ and $L$ so that~\eqref{eq:Kantorovich starting conditions} and~\eqref{eq:Kantorovich Lipschitz condition} are satisfied. 
We have
$
\nnorm{\nabla F(w_0)^{-1}}_g = \nnorm{-\nabla F(w_0)^{-1}}_g = 1/\lambda_{\min},
$
where $\lambda_{\min}$ is the minimum eigenvalue of $-\nabla F(w_0)$.
Thus we can set $a = 1/\lambda_{\min}$. 
We can then choose $b$ and check that~\eqref{eq:Kantorovich starting conditions} is satisfied by letting
\begin{equation*}
	b 
	= \nabs{F(w_0)}_g/\lambda_{\min} 
	= a \, \nabs{F(w_0)}_g 
	= \nnorm{\nabla F(w_0)^{-1}}_g \, \nabs{F(w_0)}_g 
	\geq \nabs{\nabla F(w_0)^{-1} F(w_0)}_g.
\end{equation*}
Since $(\Disk,g)$ is geodesically convex and complete, we may choose $U = \Disk$.
\cref{prop:L is two} shows that \eqref{eq:Kantorovich Lipschitz condition} is satisfied with $L=2$. 
The rest follows from plugging these constants into the statement of~\cref{theo:KantorovichRiemannian}.
\end{proof}

\section{How often do the Newton-Kantorovich conditions hold?}
\label{sec:how often do NK conditions hold}

Having established in \cref{cor:time bound under NK conditions} that if the Newton-Kantorovich conditions hold, then the conformal barycenter is easy to find, we now turn to the question of how often we find ourselves in these favorable circumstances. Experimentally, the answer is ``almost always'', even for surprisingly small values of $n$ (see \cref{fig:CDF_of_Q}).
Theoretically, the answer is ``on all but an exponentially small fraction of the space of possible measures''. We now prove this result, though we won't try to be very sharp in our estimates.

\begin{figure}
	\capstart
	\begin{center}
		\includegraphics[width=0.8\textwidth]{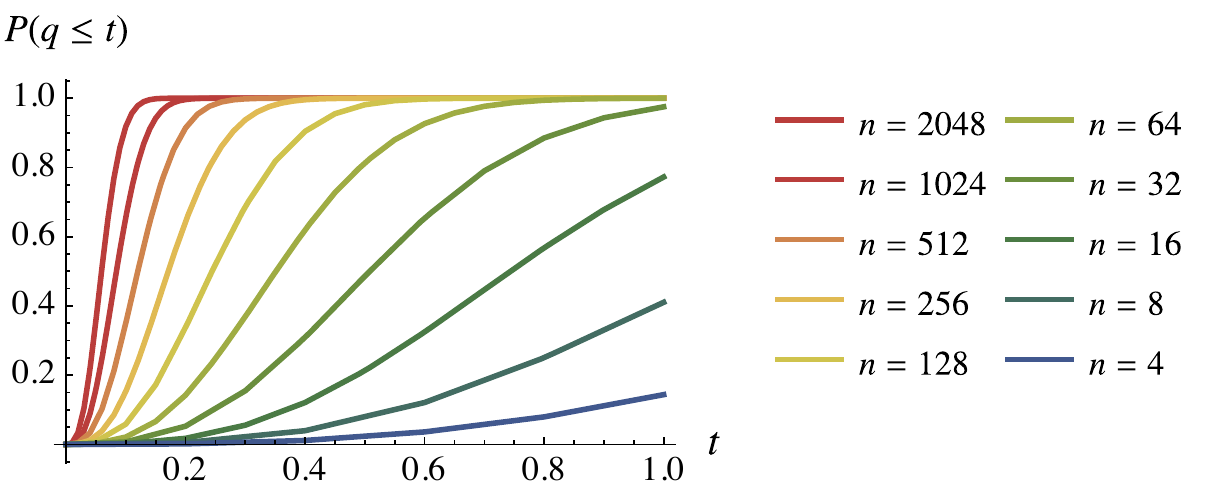}
	\end{center}
	\caption{Empirical distribution functions of $q = 4\nabs{w_{\operatorname{cm}}}/\lambda_{\min}^2$ for various values of $n$, derived from samples of size $N = 10^6$.
	Here, we fixed the weights $\omega_1= \dotsm = \omega_n = \tfrac{1}{n}$ 
	and sampled the point cloud $x$ uniformly from $(\Sphere^{2})^n$.
	As can be seen, the chance of $q$ being close to or greater than $1$ rapidly decays with increasing $n$.
	For example, for $n =64$, the empirical probablity of $P(q > 0.99)$ was lower than $0.03$ percent.
	}
	\label{fig:CDF_of_Q}
\end{figure}

\begin{theorem}
Suppose $H = \R^{\AmbDim}$, $\mu = \sum_{i=1}^n \omega_i \, \updelta(x_i)$ with $\omega_i > 0$, $\sum_{i=1}^n \omega_i = 1$. Let $\lambda_{\min}$ be the smallest eigenvalue of 
$\id_{\R^{\AmbDim}} - \sum_{i=1}^n (x_{i} \, x_{i}\transp) \, \omega_i$, let $w_{\operatorname{cm}} = \sum_{i=1}^n \omega_i \, x_i$ be the center of mass of $\mu$, and let $w_*$ be the conformal barycenter of $\mu$. Let $\nabs{\cdot}$ be the standard norm on $\R^\AmbDim$. 
Suppose we let $\rho_i = n \, \omega_i$ be the relative weight of each $\updelta(x_i)$, and have $\rho_{\max}$ be the maximum of these weights.

Consider the space $(\Sphere^{\AmbDim-1})^n$ of possible measures $\mu$ of this type with given weights $\omega_i$. For each $\AmbDim \geq 2$ there is a universal constant $C(\AmbDim) > 1$ so that the Newton-Kantorovich condition $q = 4 \, \nabs{w_{\operatorname{cm}}}/\lambda_{\min}^2 \leq 1/2$ holds on at least the fraction
\begin{equation*}
	\mathcal{P} \bigparen{q \leq \tfrac{1}{2}}
	\geq 
	1 - 2 \, \AmbDim \, C(\AmbDim)^{-\frac{n}{\rho_{\max}}}
\end{equation*}
of the space (by volume). We may take the constant $C(\AmbDim) = \exp\bigparen{\frac{3}{5} \nparen{1 - \frac{1}{\AmbDim}}^4 \frac{1}{12+\sqrt{\AmbDim}}}$, while $C(2) \geq \exp \nparen{\frac{1}{400}}$ and $C(3) \geq \exp\nparen{\frac{1}{125}}$, respectively.
\label{thm:how often}
\end{theorem}

\begin{proof}
We know that 
\begin{equation*}
	\mathcal{P} \bigparen{q \leq \tfrac{1}{2}} 
	=
	\mathcal{P} \bigparen{
		2 \, \nabs{w_{\operatorname{cm}}} < \lambda_{\min}^2 
	} 
	\geq 
	\mathcal{P} \bigparen{
		\text{$
		2 \, \nabs{w_{\operatorname{cm}}}  
		\leq
		\tfrac{4}{5} \bigparen{1 - \tfrac{1}{\AmbDim}}^2
		$
		and 
		$\tfrac{4}{5} \bigparen{1 - \tfrac{1}{\AmbDim} }^2 
		\leq \lambda_{\min}^2$ }
	}
\end{equation*}
Applying the union bound, we have
\begin{equation}
	\mathcal{P} \bigparen{q \leq \tfrac{1}{2}} 
	\geq
	1 - 
	\mathcal{P}\bigparen{\nabs{w_{\operatorname{cm}}} > \tfrac{4}{10} \bigparen{1 - \tfrac{1}{\AmbDim} }^2 }
	- 
\mathcal{P}\bigparen{\lambda_{\min} < \sqrt{\tfrac{4}{5}} \bigparen{1 - \tfrac{1}{\AmbDim} }  }
	.
\label{eq:union bound}
\end{equation}
Therefore, we must separately bound the probability that $\nabs{w_{\operatorname{cm}}}$ is large and that $\lambda_{\min}$ is small. Proposition~14 in~\cite{MR3871192} uses Bernstein's inequality to prove (for $t > 0$) 
\begin{equation*}
	\mathcal{P}\nparen{ \nabs{w_{\operatorname{cm}}} > t }
	\leq 
	\AmbDim \, \exp \biggparen{ - \frac{3 \, n \, t^2}{2 \, t \, \rho_{\max} \sqrt{\AmbDim} + 6\, (1 + n^2 \operatorname{Var} \omega_i)} }
\end{equation*}
Using the Bhatia-Davis inequality, we can overestimate $n^2 \operatorname{Var} \omega_i \leq \rho_{\max} - 1$ to simplify the bound to  
\begin{equation*}
	\mathcal{P}\nparen{ \nabs{w_{\operatorname{cm}}} > t } 
	\leq 
	\AmbDim \, \biggparen{ 
		\exp\biggparen{\frac{3 \, t^2}{2 \, t \sqrt{\AmbDim} + 6}}
	}^{-\frac{n}{\rho_{\max}}}.
\end{equation*}
Substituting in the value for $t$ given in~\eqref{eq:union bound}, the function inside the exponential becomes a complicated algebraic function of $\AmbDim$ which can be underestimated (for $\AmbDim \geq 2$) by $C(\AmbDim)$. 

For the second part, we will use Proposition~15 in~\cite{MR3871192} (noting that the sign of the inequality is wrong in the statement of the Proposition; the proof is correct),
which uses the matrix Bernstein inequality to prove (for $t > 0$) that
\begin{equation*}
	\mathcal{P}\bigparen{\lambda_{\min} < \bigparen{ 1 - \tfrac{1}{\AmbDim}} - t }
	\leq 
	\AmbDim \,
	\exp \biggparen{ -\frac{\AmbDim}{\AmbDim-1} \cdot \frac{3 \, \AmbDim \, t^2 \, n}{2 \, t \, \AmbDim \, \rho_{\max} + 6 \, (1 + n^2 \, \Var \omega_i)} }
\end{equation*}
Using the Bhatia-Davis inequality as before, we simplify the bound to 
\begin{equation*}
	\mathcal{P}\bigparen{\lambda_{\min} < \bigparen{ 1 - \tfrac{1}{\AmbDim}} - t }
	\leq 
	\AmbDim \, \biggparen{  \exp \biggparen{ \frac{3 \AmbDim^2 t^2}{2 (\AmbDim-1)(3+\AmbDim t)} } }^{-\frac{n}{\rho_{\max}}}
\end{equation*}
Setting $t = \bigparen{1 - \frac{1}{\AmbDim}} \,  \bigparen{1 - \sqrt{4/5}}$ to make the left hand side match the corresponding term in~\eqref{eq:union bound}, we again get a complicated algebraic expression in $\AmbDim$ which can again be underestimated by $C(\AmbDim)$. 
The estimates of $C(2)$ and $C(3)$ are numerical. 
\end{proof}

We see that while the largest (relative) mass $\rho_{\max}$ and the dimension $\AmbDim$ both affect our bound on the fraction of the space of measures where the Newton-Kantorovich conditions hold, this fraction always converges to 1 exponentially quickly in $n$.

\section{When the Newton-Kantorovich conditions don't hold}\label{sec:NewtonRegularized}

In extreme cases, e.g., when $\mu_k$ is very concentrated around the two endpoints of a single geodesic, $\nabla F_{\mu_k}(0)$ may be very ill-conditioned (see \cref{lem:Uniform Convexity}) so that the Newton-Kantorovich conditions may not be satisfied. This leads to an oversized search direction $u_{k+1}$ and the need for backtracking. We can mend this problem by utilizing a regularized variant of Newton's method (see \cite{MR3800474}).
Instead of $u_k$ from \eqref{eq:Newton2}, one may employ the following search direction:
\begin{align}
	u_{k} &\ceq - \bigparen{\nabla F_{\mu_k}(0) - \regparam\, \nabs{F_{\mu_k}(0)}_g^2 \, \id_H}^{-1} F_{\mu_k}(0)
	\quad \text{with} \quad \alpha \geq 0.	
	\label{eq:Newton3}
\end{align}
However, to guarantee convergence, we will have to explicitly account for the possibility of backtracking. We do this in \cref{thm:main}, which gives an algorithm which converges under the most general input conditions possible. To prepare for the proof of the theorem, we will now recast the search for the conformal barycenter  as a convex optimization problem and study the behavior of the objective function.

\subsection{Potentials}

We will now find (hyperbolic) potential functions $\psi_x$ for $V_x$ and $\varPsi_\mu$ for $F_\mu$; that is, functions $\psi_x \colon \Disk \to \R$ so that $V_x(w) = -\grad_g (\psi_x)(w)$ and $F_\mu = -\grad_g (\varPsi_\mu)(w)$. 
The unique potential for $V_x$ that is gauged to $\psi_x(0) = 0$ can be computed as follows:
\begin{equation}
	\psi_{x}(w) 
	= 
	\textstyle
	\int_0^1  \ninnerprod{- V_x(t \,w), \tfrac{\dd}{\dd t}( t \,w )}_g \, \dd t
	\\
%	&= \int_0^1  \frac{
%		2 \, 
%		\nparen{
%			 2 \, t \, \nabs{w}^2 - \ninnerprod{w,x} \, \nparen{1 + t^2 \, \nabs{w}^2}
%		}
%	}{
%		\paren{
%			1 - t^2 \, \nabs{w}^2
%		}
%		\,
%		\paren{
%			1 - 2 \, t \, \ninnerprod{w,x} + t^2 \, \nabs{w}^2
%		}
%	}
%	\, \dd t
%	\notag	
%	\\
%	&
%	= \log(1- 2 \, \ninnerprod{x,w} + \nabs{w}^2) - \log(1- \nabs{w}^2)
	=
	\log \bigparen{
		\tfrac{\nabs{x-w}^2}{1- \nabs{w}^2}
	}.
	\label{eq:VPotential}
\end{equation}
The potential $\varPsi_\mu$ for $F_\mu$ with the gauge $\varPsi_\mu(0) = 0$, is given by
\begin{align}
	\varPsi_\mu(w) 
	\ceq 
	\textstyle
	\int_\Sphere \psi_x(w) \, \dd \mu(x)
	=
	\int_\Sphere \log \bigparen{
		\tfrac{\nabs{x-w}^2}{1- \nabs{w}^2}
	} \, \dd \mu(x)
	.
	\label{eq:Potential}
\end{align}
As pointed out by Douady and Earle in \cite[Section 11]{MR0857678}, $\exp(-\psi_{x}(w))$ is proportional to the Poisson kernel for $\AmbDim=2$. However, for general $\AmbDim >2$, $\exp(-\psi_{x}(w))$ is not harmonic in $w$.
With~\eqref{eq:NablaV}, we have
\begin{align}
	\textstyle
	\Hess[g](\varPsi_\mu)
	= \int_\Sphere  \Hess[g](\psi_x) \, \dd \mu(x)
	= -\int_\Sphere \nabla V_x \, \dd \mu(x)
	= g - \int_\Sphere  V_x^\flat \otimes V_x^\flat \, \dd \mu(x).
\label{eq:hessian of psi}
\end{align}
Because $ V_x$ is a unit vector field, $\Hess[g](\varPsi_\mu)(w)$ is positive semi-definite and the spectrum of $\Hess[g](\varPsi_\mu)(w)$ is contained in $\intervalcc{0,1}$. 
In particular, $\varPsi_\mu$ is convex.
Thus the conformal barycenters of $\mu$ coincide with the minimizers of the potential $\varPsi_\mu$.
Moreover, we observe that $\Hess[g](\varPsi_\mu)(w)$ is not definite if and only if 
$\supp(\mu) \subset \set{x_+, x_-}$
and if there is a geodesic that contains $w$, $x_+$, and $w_-$. 
This motivates the following terminology that we borrow from Kapovich and Millson~\cite{MR1431002}.

\begin{definition}\label{def:stable}
A measure $\mu$ on $\Sphere$ is \emph{stable} if no point $x \in \Sphere$ supports half (or more) of the mass of $\mu$, \emph{nice semi-stable} if it is stable or one point $x_+$ supports half the mass of $\mu$ and another point $x_-$ supports half of the mass of $\mu$, \emph{semi-stable} if it is nice semi-stable or if one point $x_+$ supports exactly half the mass of $\mu$, and \emph{unstable} otherwise. 
\end{definition}
In particular, $\varPsi_\mu$ is strictly convex if $\mu$ is stable.
Since the critical points of $\varPsi_\mu$ coincide with the conformal barycenters of $\mu$, this implies that each stable $\mu$ has at most one conformal barycenter.
Indeed, it is known that the conformal barycenter of a stable $\mu$ exists (see \cite{MR0857678} and Lemma 2.9 in~\cite{MR1431002}).
Also each nice semi-stable $\mu$ has conformal barycenters, but in this case they are not unique as every point on the geodesic joining $x_+$ and $x_-$ is a conformal barycenter. And all other $\mu$ have no conformal barycenter.

For the numerical optimization of $\varPsi_\mu$ and in particular for Newton's method, it would be most desirable if $\varPsi_\mu$ where \emph{uniformly} convex in the sense that there is a fixed constant $c>0$ such that $\Hess[g](\varPsi_\mu)(w) \geq c \, g\at_w$ holds for all $w \in \Disk$.
But as we will see soon in \cref{prop:SmallestEigenvaluesAndCone}, such a constant cannot exist, at least globally. 
Instead, we show that such a constants exists for every ball $B(w_0;r) \subset \Disk$ of finite radius and we give a lower bound for this constant (see \cref{lem:Uniform Convexity}). 
This lower bound is in terms of a quantitative measurement of stability for $\mu$ (with respect to the point $w_0$) that we have to develop next.
To this end, we introduce the notion of a ``viewing cone''.

\subsection{Viewing cones and the Hessian of the potential}

\begin{definition}
For $w \in \Disk$, $X \in \UnitTangent_w \Disk$, and $\delta \in \nintervalcc{0,\uppi}$, we define the \emph{viewing cone} (see~\cref{fig:viewing cone}, (a)).
\begin{align*}
	A(w,X;\delta) \ceq 
	\set{y \in \Sphere |  \ninnerprod{X, V_y(w)}_g \geq \cos(\delta)}.
\end{align*}
\end{definition}
From the transformation rules \eqref{eq:TransformationRulesforDirectors} for $V_x$, it follows immediately that the viewing cones transform as follows under the M\"obius transformation $\varphi \in \Aut(\Disk,g)$:
\begin{align}
	\ext \varphi \paren{A ( w, X;\delta )}
	=
	A \bigparen{ \varphi(w),\dd\varphi(w) \, X ; \delta }.
	\label{eq:ViewingConeMappingProperties}
\end{align}
This notion of viewing cone allows us to provide bounds for the first eigenvalue of the Hessian of $\varPsi_\mu$ in terms of concentration of $\mu$ in a symmetric pair of viewing cones of given angle:

\begin{figure}
	\capstart %ensures that hyperlink will jump to the top of this image
\newcommand{\inc}[2]{\begin{tikzpicture}
    \node[inner sep=0pt] (fig) at (0,0) {\includegraphics{#1}};
	\node[above right= 0ex] at (fig.south west) {\begin{footnotesize}(#2)\end{footnotesize}};    
\end{tikzpicture}}%
\presetkeys{Gin}{
	trim = 45 100 15 100, 
	clip = true,  
	width = 0.333\textwidth
}{}
\begin{center}
	\hfill
	\inc{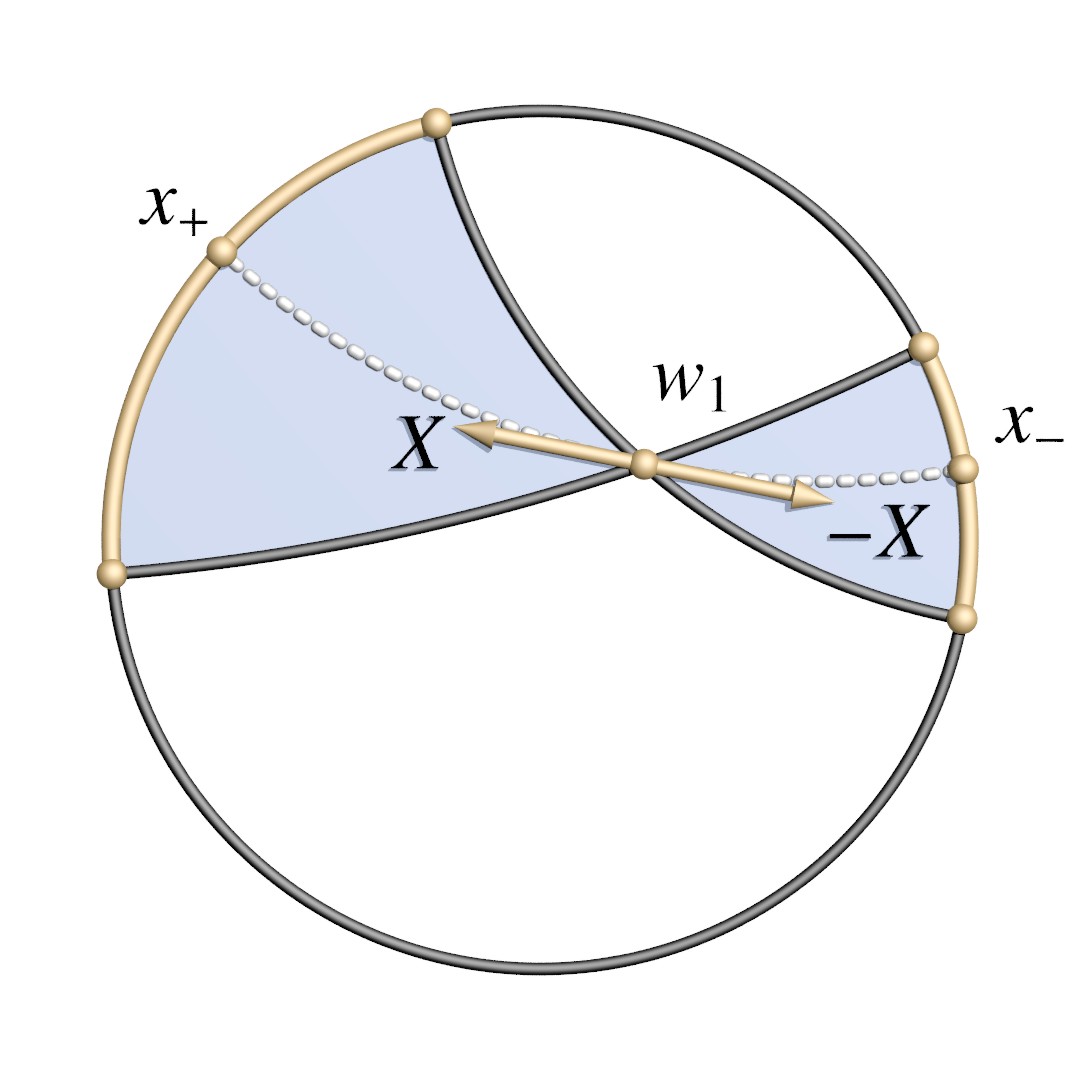}{a}%
	\hfill
	\inc{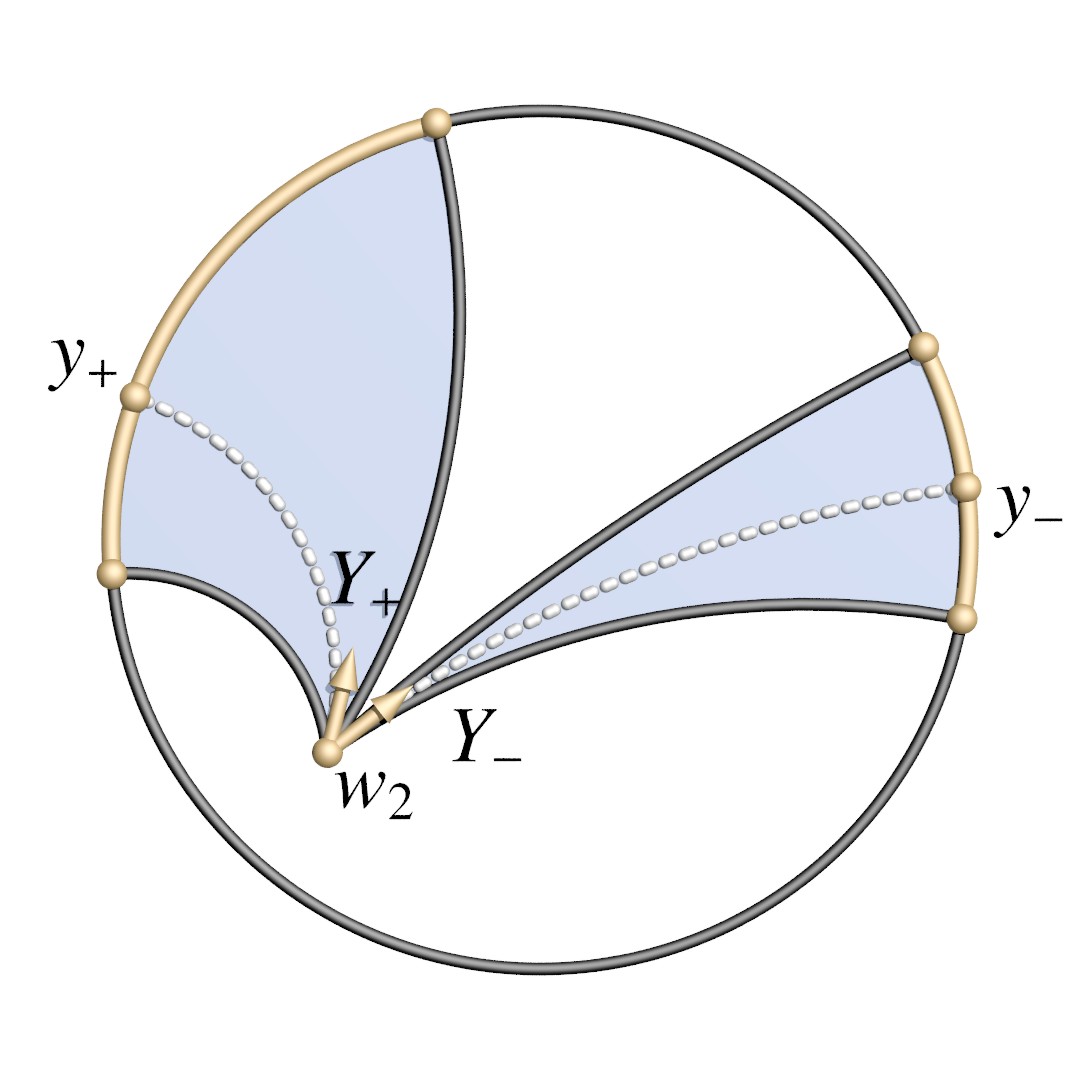}{b}%
	\hfill
	\inc{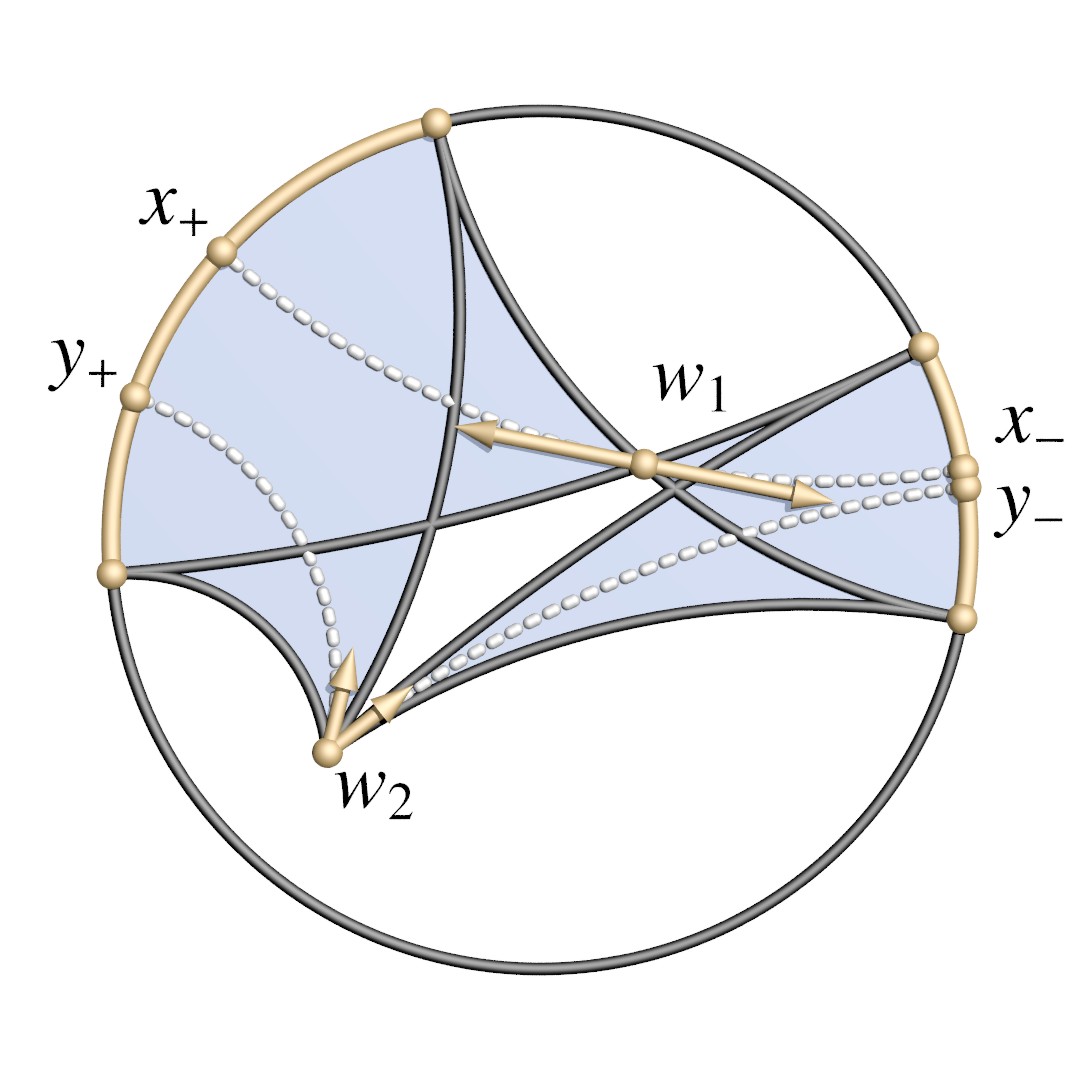}{c}%
	\hfill	
	\caption{(a) shows a symmetric viewing cone $A(w_1,X;\delta) \cup A(w_1,-X;\delta)$, with the center points $\lim_{t \rightarrow \infty} \exp(t \, X) = x_+$ and $\lim_{t \rightarrow -\infty} \exp(t \, X) = x_-$. In (b), we see that the same subset of $\Sphere$ is also the union of two viewing cones $A(w_2,Y_+;\beta_+) \cup A(w_2,Y_-;\beta_-)$ when viewed from $w_2$. However, as can be seen in (c), the direction vectors $Y_{\pm}$ are not antipodal to each other, the center points $y_{\pm}$ are not the same as $x_{\pm}$ and the angles $\beta_{\pm}$ are neither equal to each other nor to $\delta$.} 
	\label{fig:viewing cone}
\end{center}
\end{figure}
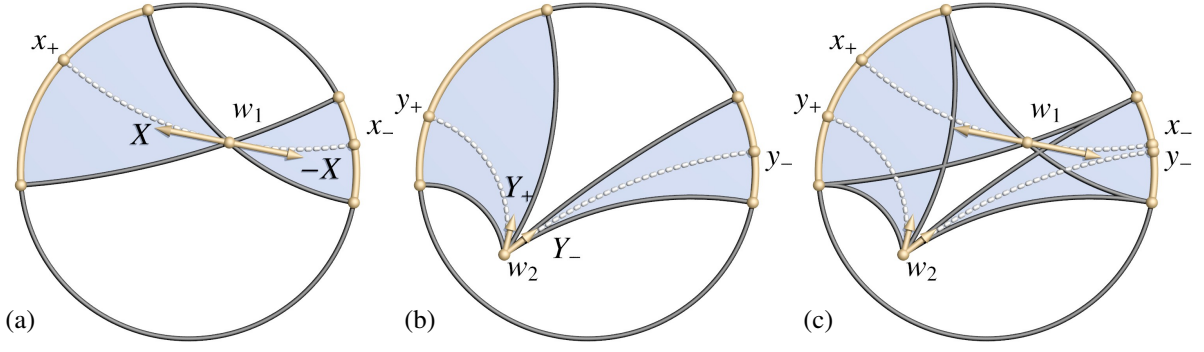

\begin{proposition}\label{prop:SmallestEigenvaluesAndCone}
Let $\beta >0$ and  define
\begin{align*}
	a(w) 
	\ceq 
	\inf_{X \in \UnitTangent_w \Disk} \mu \bigparen{\Sphere \setminus (A(w,X;\beta) \cup A(w,-X;\beta))}.
\end{align*}
Then the smallest eigenvalue $\lambda_\mu(w)$ of $\Hess[g](\varPsi_\mu)(w)$ with respect to $g\at_w$ satisfies
\begin{align*}
	a(w) \, \sin^2(\beta) 
	\leq 
	\lambda_\mu(w)
	\leq
	\sin^2(\beta) + a(w) \, \cos^2(\beta).
\end{align*}
So if $a(w) >0$, the condition number of the Hessian is bounded by $\kappa(\Hess[g](\varPsi_\mu)) \leq \tfrac{1}{a(w)} \csc^2(\beta)$.
\end{proposition}
In particular this shows that $\lambda_\mu(w)$ can be arbitrarily small when almost all the measure of $\mu$ is supported in a single double viewing cone $A(w,X;\beta) \cup A(w,-X;\beta)$ with small opening angle $\beta$.
Indeed, if $\mu$ is supported in $A(w,X;0) \cup A(w,-X;0)$, then $\lambda_\mu(w) = 0$.

\begin{proof}
Fix $\beta >0$ and $w \in \Disk$. 
For $X \in \UnitTangent_w \Disk$, define the symmetric double cone $A(X) \ceq A(w,X;\beta) \cup A(w,-X;\beta)$
and $b(X) \ceq  \mu(\Sphere \setminus A(X))$. 
Using~\eqref{eq:hessian of psi}, we have
\begin{equation*}
	\Hess[g](\varPsi_\mu)(w) (X,X)  
	= 
	1 
	- \smallint_{\Sphere \setminus A(X)} \ninnerprod{X,V_y(w)}_g^2 \, \dd \mu (y)
	- \smallint_{A(X)} \ninnerprod{X,V_y(w)}_g^2 \, \dd \mu (y).
\end{equation*}
For $y$ outside of $A(X)$, we have the bound $\ninnerprod{X,V_y(w)}^2_g \leq \cos^2 (\beta)$, 
while we only know that $\ninnerprod{X,V_y(w)}^2_g \leq 1$ in case of $y \in A(X)$. Thus we obtain
\begin{align*}
	\MoveEqLeft
	\Hess[g](\varPsi_\mu)(w) (X,X)  
	\geq 
	\textstyle
	1 
	- \smallint_{\Sphere \setminus A(X)} \cos(\beta)^2 \, \dd \mu (y)
	- \smallint_{A(X)} \, \dd \mu (y)	
	\\
	&= 1 
	-b(X) \, \cos(\beta)^2 
	- (1-b(X))
	=
	b(X) \, \sin^2(\beta) \geq a(w) \,  \sin^2(\beta).
\end{align*}
It is well-known that $\lambda_\mu(w) = \inf_{X \in \UnitTangent_w \Disk} \Hess[g](\varPsi_\mu)(w)(X,X)$, so $\lambda_\mu(w)$ is the greatest lower bound for $\{ \Hess[g](\varPsi_\mu)(w)(X,X) \mid X \in \UnitTangent_w \Disk \}$. We have just proved that $a(w) \sin^2(\beta)$ is some lower bound for the same set. Thus we obtain
$
	\lambda_\mu(w) \geq a(w) \sin^2(\beta)
$.
Since the maximum eigenvalue of $\Hess[g](\varPsi_\mu)$ is at most $1$, our estimate of the condition number $\kappa( \Hess[g](\varPsi_\mu))$ follows immediately. 

For $y$ inside of $A(X)$, we know $\ninnerprod{X,V_y(w)}^2_g \geq \cos^2 (\beta)$, while outside of $A(x)$ we know only $\ninnerprod{X,V_y(w)}^2_g \geq 0$. Thus
\begin{align*}
	\MoveEqLeft
	\lambda_\mu(w)  \leq \Hess[g](\varPsi_\mu)(w) (X,X)  
	\leq 
	\textstyle	
	1 
	- 0
	- \int_{A(X)} \cos^2(\beta) \, \dd \mu (y)	
	\\
	&\leq 1 
	- (1-b(X)) \cos^2(\beta)
	= \sin^2(\beta) + b(X) \cos^2(\beta).
\end{align*}
As before we now know that $\lambda_\mu(w)$ is a lower bound for $\{ \sin^2(\beta) + b(X) \cos^2(\beta) \mid X \in \UnitTangent_w \Disk\}$. Since $\sin^2(\beta) + a(w) \cos^2(\beta)$ is the greatest lower bound for the same set, we have
$
	\lambda_\mu(w)  
	\leq 
	\sin^2(\beta) + a(w) \cos^2(\beta),
$
as desired.
\end{proof}

\subsection{Quantitative stability}

We now want to show that~\cref{prop:SmallestEigenvaluesAndCone} can be applied to any stable Borel probability measure $\mu$ on $\Sphere$. To that end, we will first show
\begin{lemma}\label{lem:ConcentrationLemma2}
Let $\mu$ be a Borel probability measure on the complete metric space $(M,d)$, such that no singleton has measure $1/2$ or more, i.e. $\mu(\{x\}) < 1/2$ for all $x \in M$ or $\mu$ is stable. Then there are $r>0$ and $\varepsilon >0$ such that for all $x\in M$ one has $\mu(\bar B(x;r)) \leq (1-\varepsilon)/2$, where $\bar B(x;r)$ is the closed ball of radius $r$ around $x$.
\end{lemma}
Recall that the viewing cone $A(0,x/2;\delta) \subset \Sphere$ is just a closed spherical ball around $x \in \Sphere \cong \UnitTangent_0 \Disk$ of radius $\delta$.
Thus, \cref{lem:ConcentrationLemma2} shows that every stable Borel probability measure on $\Sphere$ has to satisfy the following property for $w = 0$:
\begin{align}
	\text{
	There are $\varepsilon>0$, $\delta >0$ s.t.
	$\mu ( A(w,X;\delta)) \leq \tfrac{1-\varepsilon}{2}$ holds for all $X \in \UnitTangent_{w} \Disk$.
	}
	\label{eq:NoConcentration property}
\end{align}
Now M\"obius transformations map stable measures to stable measures, which immediately yields:
\begin{lemma}\label{lem:ConcentrationLemma}
Let $\mu$ be a Borel probability measure on $\Sphere$. Then the following three statements are equivalent:
\begin{enumerate}
	\item $\mu$ is stable; 
	\item condition \eqref{eq:NoConcentration property} is satisfied for just one point $w \in \Disk$; and 
	\item a condition like \eqref{eq:NoConcentration property} holds for each point $w \in \Disk$ (with varying $\varepsilon$ and $\delta$).
\end{enumerate}
\end{lemma}

We now turn to the proof of~\cref{lem:ConcentrationLemma2}:
\begin{proof}
\emph{Assume} that this were false. Then for each $n \in \N$ there is an $x_n \in M$ with
$\mu(B(x_n;4^{-n})) > (1-4^{-n})/2$.
We abbreviate $\varOmega_n \ceq \bar B(x_n;4^{-n})$ and note that each $\varOmega_n$ has measure greater than~$1/3$.

\textbf{Claim :} \emph{There is a subsequence $(n_k)_{k \in \N}$ such that
$\varOmega_{n_k} \cap \varOmega_{n_m} \neq \emptyset$ for each $k$ and each $m \geq k$ .}
\newline
We thin out the sequence $a_1 \ceq (\varOmega_n)_{n \in \N}$ recursively 
so obtain sequences $a_1$, $a_2$, $a_3 \dotsc$ such that 
$a_{k+1}$ is a subsequence of $a_{k}$ and such that $a_{k,1}$  of each $a_k$ has nontrivial intersection with all $a_{k,i}$, $i \geq 1$.
Then $\{ a_{k,1} \}_{k \in \N}$ is the subsequence of $(\varOmega_n)_{n \in \N}$ that we are looking for.
To this end we consider $k \in \N$ and the subsequence $a_k$ of $( \varOmega_n)_{n \in \N}$.
It suffices to show that this sequence must contain an $a_{k,i_0}$ that is intersected by infinitely many of the $a_{k,j}$, $j \geq i_0$ because we can then define $a_{k+1}$ as the subsequence of $a_k$ that contains $a_{k,i_0}$ and all the $a_{k,j}$, $j \geq i_0$ that intersect $a_{k,i_0}$.
Indeed, such an $a_{k,i_0}$ does exist: Either $a_{k,1}$ intersects infinitely many $a_{k,i}$ or not. If not, then there is a first $a_{k,i}$ that is disjoint from $a_{k,1}$. For every $j \geq i$ we have
$
	\mu\nparen{ a_{k,1}\cup a_{k,i}} + \mu \nparen{ a_{k,j}}
	=
	\mu\nparen{a_{k,1}} + \mu \nparen{a_{k,i}} +\mu \nparen{a_{k,j}}
	> 1.
$
This means that each $a_{k,j}$, $j \geq i$ must intersect $a_{k,1}$ or $a_{k,i}$, since the total measure of $\mu$ is $1$. Thus at least one of $a_{k,1}$ and $a_{k,i}$ is intersected by infinitely many $a_{k,j}$. Altogether, this proves the claim.

Having found the sequence $(n_k)_{k \in \N}$, we observe that $(x_{n_k})_{k \in \N}$ must be a Cauchy sequence.
Indeed, for each $N \in \N$ and all $j \geq i  \geq N$, the balls
$\bar B(x_{n_i};4^{-n_i})$ and $\bar B(x_{n_j};4^{-n_j})$ have nontrivial intersection, thus we have
$d(x_{n_i},x_{n_j}) \leq 2 \cdot  4^{-n_N}$.
Thus there exists a limit point $x$ of $(x_{n_k})_{k \in \N}$ and we have $d(x,x_{n_k}) \leq 2 \cdot  4^{-n_k}$.
In particular, this implies $B(x; 4 \cdot  4^{-n_k}) \supset \bar B(x_{n_k}; 4^{-n_k})$ and, because every  finite Borel measure on metric spaces is outer regular, we have
\begin{align*}
	\mu(\{x\}) 
	= \lim_{k \to \infty} \mu \bigparen{ B(x; 4 \cdot  4^{-n_k})}
	\geq \lim_{k \to \infty} \mu \bigparen{ \bar B(x_{n_k};4^{-n_k})}
	\geq \lim_{k \to \infty} \tfrac{1 - 4^{n_k}}{2} = \tfrac{1}{2}
\end{align*}
which is a \emph{contradiction} to the stability of $\mu$. So the initial assumption must be wrong and we have proven the lemma.
\end{proof}

\subsection{Uniform convexity of the potential}

We now prove that the potential $\varPsi_\mu$  of a stable Borel probability measure $\mu$ is \emph{uniformly} convex on balls of finite radius and provide an explicit bound on the lowest eigenvalue of the Hessian.

\begin{lemma}\label{lem:Uniform Convexity}
Let $w_0 \in \Sphere$ and suppose that $\mu$ satisfies \eqref{eq:NoConcentration property} for $w = w_0$.
Let $\lambda_\mu(w)$ denote the smallest eigenvalue of $\Hess(\varPsi_\mu)(w)$. Then for all $w \in \Disk$ we have 
\begin{align*}
	\lambda_\mu(w) \geq \varepsilon \, \sin^2(\delta) \exp (- 2 \, d_g(w_0,w )).
\end{align*}
In particular, $\varPsi_\mu$ is uniformly convex and the condition number of $\Hess[g](\varPsi_\mu)(w)$ is uniformly bounded on hyperbolic balls around $w_0$ of finite radius. 
\end{lemma}

Proving \cref{lem:Uniform Convexity} will require us to investigate a further property of viewing cones:
Each viewing cone  $A(w,X,\delta)$ is a closed ball with respect to the angular metric on the sphere~$\Sphere$.
This can be most easily seen by applying the Möbius transformation $\varphi \in \Aut(\Disk,g)$ and by recalling that the induced M\"obius transformation $\varphi \colon \Sphere \to \Sphere$ maps spherical balls to spherical balls.
Thus every viewing cone $A(w_1,X_1,\delta_1)$ at point $w_1$ coincides with another viewing cone $A(w_2,X_2,\delta_2)$ at point $w_2$. However, the center and angle of the cone change as we move from $w_1$ to $w_2$ (see \autoref{fig:viewing cone}). 
We now estimate this change of the angles:

\begin{lemma}\label{lem:ViewingAngleEstimate}
Let $w_1 \in \Disk$, $X_1 \in \UnitTangent_{w_1}\Disk$, and $\delta_1 \in (0,\uppi)$.
Then for each $w_2 \in \Disk$ there exist unique $X_2 \in \UnitTangent_{w_2}\Disk$ and $\delta_2 \in (0,\uppi)$ such that the viewing cones $A(w_1,X_1;\delta_1)$ and $A(w_2,X_2;\delta_2)$ coincide. Further, we have
\begin{align*}
	\exp(- d_g(w_1,w_2))   \leq  \frac{\sin(\delta_2)}{\sin(\delta_1)}  \leq  \exp(d_g(w_1,w_2)).
\end{align*}
\end{lemma}
\begin{proof}
By applying the shift transformation $\Shift_{w_1}$, we may assume that \hbox{$w_1=0$}.
We fix the shift transformation $\varphi \ceq \Shift_{w_2} \colon \Disk \to \Disk$ which extends to a M\"obius transformation of~$\Sphere$. Observe that $A(0,X_1;\delta_1)$ is a ball on the sphere with center $x_1 \ceq \lim_{t\to \infty}\exp_0(t \, X_1)$.
Since $\ext \varphi$ is a M\"obius transformation on $\Sphere$ and because Möbius transformations map balls to balls, also $\ext \varphi(A(0,X_1;\delta_1))$ has to be a ball. Thus it can be written as a viewing cone from the origin:
$\ext \varphi(A(0,X_1;\delta_1)) = A(0,Y;\delta_2)$ with some $Y \in \UnitTangent_0 \Disk$ and $\delta_2 > 0$. Letting $X_2 \ceq \dd\varphi(y)^{-1} Y$, \eqref{eq:ViewingConeMappingProperties} shows that
\begin{align*}
	A(w_1,X_1;\delta_1)
	&=
	A(0,X_1;\delta_1)
	=
	\ext\varphi^{-1}(A(0,Y;\delta_2))
	\\
	&=
	A(\varphi^{-1}(0), \dd\varphi(y)^{-1} Y;\delta_2)
	=
	A(w_2, X_2;\delta_2),
\end{align*}
which completes the proof of the first statement of the lemma.

We still need to estimate $\delta_2$. Observe that $w_1 = 0$, $w_2$ and $x_1$ lie in a common two-dimensional plane $E \subset H$. 
Recalling~\eqref{eq:ShiftTrafo}, we observe that the shift $\varphi$ taking $w_2$ to $0$ takes any point $z \in E$ to a linear combination of $z$ and $w_2$; that is, to a point in $E$. The inverse map $\varphi^{-1} = \Shift_{-w_2}$ has the same property, so $\varphi$ maps $E \cap \Disk$ (and $E \cap \Sphere$) to itself. 

An easy computation shows that if $z_1 \in E$ and $z_2 \in E^\perp$ then $\varphi(z_1+z_2) + \varphi(z_1-z_2) \in E$.
Since the ball $A(0,X_1;\delta_1)$ is centrally symmetric with respect to $E$, this shows that the ball $\varphi(A(0,X_1;\delta_1)) = A(0,Y;\delta_2)$ is also centrally symmetric with respect to $E$, and in particular its center $y \ceq \lim_{t \to \infty} \exp_0(t\,Y) = 2 \,Y$ must lie in $E$. 

We can conclude that $A(0,X_1;\delta_1) \subset \Sphere$  and $A(0,Y;\delta_2) \subset \Sphere$ are spherical balls centered at points $x_1$ and $y$ in $E \cap \Sphere$. Thus $E$ intersects each of the balls in a diameter and $\varphi$ maps the diameter $\{x_+,x_-\}$ of $A(0,X_1;\delta_1) \cap E$ to the diameter $\{\varphi(x_+),\varphi(x_-)\}$ of $A(0,Y;\delta_2) \cap E$. 
Without loss of generality, we may identify $E$ with $\R^2$ and assume that  $X_1 = (1,0)$, $w_2 = r \, ( \cos(\theta) , \sin(\theta))$,
and $x_\pm =  (\cos(\pm \delta_1), \sin(\pm \delta_1))$.
Then we have
$
	\sin^2(\delta_2) = (1 - \cos(2 \, \delta_2))/2 = (1 - \ninnerprod{ \varphi(x_+), \varphi(x_-)})/2
$
and a short computation involving \eqref{eq:ShiftTrafo} and $\cos(\delta_1 \pm \theta) \geq -1$ lead to
\begin{align*}
	\sin^2(\delta_2) 
	=
	\frac{
		(1 - r)^2 
		\,
		(1 + r)^2 
		\, 
		\sin^2(\delta_1)
	}{
		\bigparen{1 - 2 \, r \cos(\delta_1 - \theta) + r^2} 
		\, 
		\bigparen{1 - 2 \, r \cos(\delta_1 + \theta) + r^2}
	}
	\geq
	\frac{
		(1 - r)^2
	}{
		(1 + r)^2
	} 
	\, \sin^2(\delta_1)
	.
\end{align*}
Both $\delta_1$ and $\delta_2$ are contained in the interval $\intervalcc{0,\uppi}$ where $\sin$ is nonnegative, thus we may apply the square root;
substituting $r = \tanh(d_g(w_1,w_2)/2)$ leads to 
\begin{align*}
	\sin(\delta_2)\geq \exp(- d_g(w_1,w_2)) \, \sin(\delta_1).
\end{align*}
This shows the first inequality stated by the lemma. The second follows from swapping the roles of $w_1$ and $w_2$.
\end{proof}

We are now ready to prove \cref{lem:Uniform Convexity}.
\begin{proof}
Fix $w \in \Disk$ and $X \in \UnitTangent_w \Disk$.
By \cref{lem:ViewingAngleEstimate}, there are unit tangent vectors $Y_\pm  \in \UnitTangent_{w_0} \Disk$ and angles $\beta_\pm > 0$ so that  $A(w,\pm X; \beta_\pm) = A(w_0,Y_{\pm}; \delta)$, as shown in~\cref{fig:viewing cone}. Notice that the two vectors $Y_{\pm}$ are likely \emph{not} some pair $\pm Y$ of antipodal vectors. 
We put $\beta \ceq \min \set{\beta_-,\beta_+, \uppi/2}$
%\uppi/2 as upper bound is important to guarantee that \mu(A(w, X; \beta) \cup A(w,-X; \beta)) = \mu(A(w, X; \beta)) +  \mu(A(w,-X; \beta))
and observe that \cref{lem:ViewingAngleEstimate} implies
$
	\sin^2(\beta)
	\geq
	\exp(- 2 \, d_g(w,w_0)) \, \sin^2(\delta).
$
By the choice of $\beta$, we have
$
	\mu(A(w,\pm X;\beta)) \leq \mu(A(w,\pm X; \beta_\pm) = \mu(A(w_0,Y_{\pm}; \delta)) \leq (1- \varepsilon)/2
$.
Now $0 \leq \beta \leq \uppi/2$ implies
\begin{equation*}
	b(X) 
	\ceq  
	\mu( \Sphere \setminus (A(w, X; \beta) \cup A(w,-X; \beta) ) 
	=
	1 - \mu(A(w, X; \beta)) - \mu(A(w,-X; \beta))
	\geq
	\varepsilon
	.
\end{equation*}
Since $X$ was arbitrary, $a(w) = \inf_{X\in\UnitTangent_w\Disk} b(X) \geq \varepsilon$ holds as well. \cref{prop:SmallestEigenvaluesAndCone} then yields
\begin{align*}
	\lambda_\mu(w) 
	\geq 
	a(w) \, \sin^2(\beta) 
	\geq
	\varepsilon \, \sin^2(\delta) \exp(-2\,d_g(w,w_0))
	.
\end{align*}
The largest eigenvalue of $\Hess[g](\varPsi_\mu)$ is at most $1$, so this provides a uniform estimate on the condition number of this matrix as well.
\end{proof}

\subsection{A general bound on the location of the conformal barycenter}
Recall that conformal barycenters of $\mu$ coincide with minimizers of the potential $\varPsi_\mu$.
Thus the following provides an	 \emph{a priori} bound on the distance that a conformal barycenter (if existent) can have from a given point. 

\begin{lemma}\label{lem:generalbound}
Let $w_0 \in \Sphere$ and suppose that $\mu$ satisfies \eqref{eq:NoConcentration property} for $w = w_0$.
Then we have $\varPsi_\mu(w) > \varPsi_\mu(w_0)$ for each $w$ outside the closed ball $B$ around $w_0$ of radius
$
	r(\varepsilon, \delta) \ceq - (2/\varepsilon) \log \bigparen{\sin(\delta)/2}.
$
\end{lemma}

\begin{proof}
By applying a shift $\sigma_{w_0}$, we may assume without loss of generality that $w_0 = 0$.
We recall that $\varPsi_\mu$ was chosen so that $\varPsi_\mu(0) = 0$. 
We now prove that $\varPsi_\mu$ is strictly positive outside of the closed ball $B \ceq \set{w \in \Disk | d_g(w,0) \leq r(\varepsilon,\delta)}$. 

For $x \in \Sphere$, put $X \ceq V_x(0) = \frac{1}{2} \, x$.
Denote the unit speed geodesic ray emanating from $0$ in direction $X$ by $\gamma_x(r) \ceq \exp_0(r \, X) = \tanh(r/2) \, x$ (see \eqref{eq:RadialGeodesic}).
With \eqref{eq:VPotential},
we obtain for each $y \in \Sphere$ that
\begin{align*}
	\psi_{y}( \gamma_x(r) ) 
%	=
%		\log \biggparen{
%		\frac{1- 2 \,\cos(\theta) \, \tanh(r/2) + \tanh^2(r/2)}{1- \tanh^2(r/2)}
%	}
	=
	\log \paren{ \cosh(r) - \cos(\theta) \, \sinh(r)},
	\quad
	\text{where $\cos(\theta) = \ninnerprod{x,y}$.}
\end{align*}
As a restriction of a convex function to a geodesic, $r \mapsto \psi_{y}( \gamma_x(r) )$ is convex. 
Further, a direct computation shows that it has slant asymptotes which we may employ as lower bounds:
\begin{align*}
	\psi_{y}( \gamma_x(r) ) 
	\geq
	\min \Set{
		r + 2 \, \log \nparen{\sin(\tfrac{\theta}{2})}
		,
		-r + 2 \, \log \nparen{\cos(\tfrac{\theta}{2})}
	}.
\end{align*}
For $0\leq \theta \leq \uppi$, $\log \nparen{\sin(\tfrac{\theta}{2})}$ is monotonically increasing, while $\log \nparen{\cos(\tfrac{\theta}{2})}$ is monotonically decreasing.
So we have the bounds
\begin{align*}
	\psi_{y}( \gamma_x(r) )
	\geq
	\begin{cases} 
		\phantom{-}r + 2 \, \log \nparen{ \sin (\tfrac{\delta}{2})}, & 
		\text{for $\delta \leq \theta \leq \uppi$,}\\
		-r + 2 \, \log \nparen{ \cos (\tfrac{\delta}{2})}, & \text{otherwise}.
	\end{cases}
\end{align*}
We abbreviate $a \ceq \mu(A(0,X;\delta)) \leq \frac{1- \varepsilon}{2} $.
Thus, we have for all $r > r(\varepsilon,\delta) > 0$ that
\begin{align*}
	\MoveEqLeft
	\varPsi_\mu( \gamma_x(r) )
	=
	\textstyle
	\int_{\Sphere \setminus A(0,X;\delta)} \psi_{y}( \gamma_x(r) ) \, \dd \mu(y)
	+
	\int_{A(0,X;\delta)} \psi_{y}( \gamma_x(r)  ) \, \dd \mu(y)
	\\
	&\geq
	\textstyle	
	\int_{\Sphere \setminus A(0,X;\delta)}  \bigparen{ r + 2 \, \log \nparen{ \sin (\tfrac{\delta}{2})} } \, \dd \mu(y)
	+
	\int_{A(0,X;\delta)} \bigparen{- r + 2 \, \log \nparen{\cos(\tfrac{\delta}{2})}} \, \dd \mu(y)	
	\\
	&=
	(1-a) \, \bigparen{ r + 2 \, \log \nparen{ \sin (\tfrac{\delta}{2})} }
	+
	a \, \bigparen{ - r + 2 \, \log \nparen{\cos(\tfrac{\delta}{2})}}	
	\\
%	&=
%	\paren{ 1- 2 \, a} \, r
%	+
%	2 \, (1-a) \, \log \nparen{ \sin (\tfrac{\delta}{2})} 
%	+ 2 \, a \, \log \nparen{\cos(\tfrac{\delta}{2})} 
%	\\
	&
	\geq 
	\varepsilon\, r
	+
	2 \log  \bigparen{ \min \nparen{ \sin ( \tfrac{\delta}{2}), \cos ( \tfrac{\delta}{2})}}
	\geq
	\varepsilon\, r 
	+
	2 \log( \sin(\delta)/2)
	>0
	.
\end{align*}
\end{proof}

\subsection{Main result: A fast, robust, and globally convergent algorithm}
\label{subsec:main result}

We are now prepared to state a damped and regularized version of Newton's method \eqref{eq:Newton1}.
By performing an appropriate line search, we can guarantee that the method converges for each stable measure and each starting values.

Choosing a regularization parameter $\regparam \geq 0$,\footnote{In our experiments, the choice $\alpha = 1$ turned out to work best.} we start with some initial guess $w_0 \in \Disk$ and iteratively define:
\begin{align}
	v_{k} &\ceq - \bigparen{\nabla F_{\mu}(w_k) - \regparam \, \nabs{F_{\mu}(w_k)}_g^2 \, \id_{\TangentBundle_{w_k}\Disk}}^{-1} F_{\mu}(w_k),
	\label{eq:Newton3.1}
	\\
	w_{k+1} &= \exp_{w_k}(\tau_k \, v_{k}),
	\label{eq:Newton3.2}
\end{align}
where the $\tau_k > 0$ are chosen by a line search so that for some given constants $0< c_1 \leq \tfrac{1}{2}$ and $c_1 < c_2 < 1$, 
	the following conditions are met along the \emph{search line} $\gamma_{k}(t) \ceq \exp_{w_k}(t \, v_{k})$ for the \emph{merit function} $\MeritFun_{k}(t) \ceq \varPsi_{\mu_k}( \gamma_{k} (t))$:
\begin{align}
	\MeritFun_{k}(\tau_{k}) &\leq \MeritFun_{k}(0) + c_1 \, \tau_k \, \MeritFun_{k}'(0)
	&&\text{(Armijo condition)}
	\label{cond:armijo}
	\\
	\MeritFun_{k}'(\tau_{k}) &\geq c_2 \, \MeritFun_{k}'(0). 
	&&\text{(weak Wolfe condition)} 
	\label{cond:wolfe}		
\end{align}
Moreover, we require that
\begin{gather}
	\text{$\tau_{k} = 1$ whenever this choice satisfies \eqref{cond:armijo} and \eqref{cond:wolfe}.}
	\label{cond:otherwise}	
\end{gather}
We note that the Armijo and weak Wolfe condition are standard conditions to guarantee global convergence in unconstrained optimization algorithms.
It is well-known how to realize these conditions in a line search algorithm (see for example \cite[Chapter~3]{MR2244940}).
Such algorithms only involve evaluation of $f_k(t)- f_k(0)$ and $f_k'(t)$.
For an isometry $\varphi \in \Aut(\Disk)$, the potential $\varPsi_{\varphi_\push \mu} \circ \varphi$ differs from $\varPsi_{\mu}$ only by a constant, and so the line search can also be pushed to the origin via the shift transformation $\Shift_{w_k}$. This leads us to the following shifted variant of \eqref{eq:Newton3}:
\begin{align}
	\mu_k &\ceq (\ext \Shift_{w_k})_\push \mu
	\label{eq:Newton4.1}	
	\\
	u_{k} &\ceq - \bigparen{\nabla F_{\mu_k}(0) - \regparam \, \nabs{F_{\mu_k}(0)}_g^2 \, \id_{H}}^{-1} F_{\mu_k}(0)
	\label{eq:Newton4.2}
	\\
	w_{k+1} &\ceq \Shift(-w_k , \exp_{0}(\tau_k \, u_{k})),
	\label{eq:Newton4.3}	
\end{align}
where $\tau_k$ is determined as in \eqref{cond:armijo}--\eqref{cond:otherwise} but with the merit function $f_k$ replaced by
$f_k(t) \ceq \varPsi_{\mu_k}(\exp_0(t \, u_{k}))$.
	
Next we show that the sequence $(w_k)_{k\in\N}$ created by these algorithms converges to the conformal barycenter for each stable Borel probability measure $\mu$ on $\Sphere$.	
As a side effect, we obtain a proof of the \emph{existence} and uniqueness of the conformal barycenter of any stable Borel probability measure. In contrast to the original proof by Douady and Earle \cite[Proposition 1]{MR0857678}, this proof does not rely on the Poincaré-Hopf index theorem and thus also works in nonseparable Hilbert spaces (for which even the Browder-Minty theorem cannot be applied to show existence).
	
\begin{theorem}\label{thm:main}
Suppose $\mu$ is a Borel probability measure on $\Sphere$ which is stable in the sense of \cref{def:stable}.
Then there exists a unique conformal barycenter $w_*(\mu)$ of $\mu$ and the iterates $(w_k)_{k \in \N}$ defined by \eqref{eq:Newton3.1}-- \eqref{eq:Newton3.2} or \eqref{eq:Newton4.1}--\eqref{eq:Newton4.3} converge quadratically to $w_*(\mu)$
in the sense that there is a $C > 0$ (depending on $\mu$) so that
\begin{equation*}
%	d_g(w_{k+1},w_*(\mu))
%	\leq
%	C \, d_g(w_k,w_*(\mu))^2
%	\quad
%	\text{holds for all sufficiently large $k$.} 
	\limsup_{k \rightarrow \infty} \frac{d_g(w_{k+1},w_*(\mu))}{d_g(w_k,w_*(\mu))^2} \leq C.
\end{equation*}
\end{theorem}
\begin{proof}
We employ the techniques from \cite{MR2968868} utilizing the Riemannian exponential map $\exp$ as retraction.
Since $\mu$ is stable, \cref{lem:ConcentrationLemma} and \cref{prop:SmallestEigenvaluesAndCone} imply that the operator $\nabla F_\mu(w) \colon \TangentBundle_w \Disk \to \TangentBundle_w^* \Disk$ is continuously invertible at each point $w \in \Disk$.
By \cite[Proposition 1]{MR2968868}, there is always a step size $\tau_k$ satisfying \eqref{cond:armijo}--\eqref{cond:otherwise}. Thus the sequence $(w_k)_{k \in \N}$ is well-defined.
The Armijo condition \eqref{cond:armijo} guarantees $\varPsi_\mu(w_k) < \varPsi_\mu(0)$ and thus \cref{lem:generalbound} implies that $(w_k)_{k \in \N}$ stays within a ball $B$ of finite hyperbolic radius.
By \cref{lem:Uniform Convexity}, this in turn implies that there are $0 < \lambda \leq \varLambda < \infty$ such that
the self-adjoint linear operators $A_{\alpha,k} \ceq -\nabla F_\mu (w_k) +  \regparam \, \nabs{F_{\mu}(w_k)}_g^2 \, \id_{\TangentBundle_{w_k}\Disk}$ satisfy the uniform bound
\begin{align*}
	\lambda \, \id_{\TangentBundle_{w_k}\Disk}
	\preceq	
	A_{\alpha,k} 
	\preceq
	\varLambda \, \id_{\TangentBundle_{w_k}\Disk}
	\quad \text{for all $k \in \N$.}
\end{align*}
Notice that $\varPsi_\mu$ is Lipschitz continously differentiable.
Since it is convex with $\varPsi_\mu(0) =0$ and finite slope at $\nabs{\grad_g(\varPsi_\mu(0))}_g < \infty$,  it is also bounded from below on the ball $B$.
So \cite[Corollary 3]{MR2968868} implies so-called \emph{global convergence in the sense that}
\begin{align*}
	\nabs{F_\mu (w_k)}_g   = \nabs{\dd \varPsi_\mu (w_k)}_g \converges{k \to \infty} 0.
\end{align*}
Notice that these eigenvalues are uniformly bounded from below by $\lambda$.
So there must be a $k_0$ such that $q_k \ceq 4 \, \nabs{F_\mu (w_k)} \, \lambda_k^{-2} \leq 4 \, \nabs{F_\mu (w_k)} \, \lambda^{-2} < 1$ holds all $k \geq k_0$.
For those $k$, the Newton-Kantorovich theorem (for the Newton algorithm started at $w_k$, see \autoref{thm:quadratic convergence})
implies the existence of a conformal barycenter.  Moreover, since $\varPsi_\mu$ is strictly convex, there is exactly one such conformal barycenter $w_*(\mu)$. Hence the estimate from \autoref{thm:quadratic convergence} (with $w_0$ replaced by $w_k$) implies
\begin{align*}
	d_g(w_k , w_*(\mu)) 
	\leq  
	\tfrac{1}{2} \,\lambda_k \,  q_k^{(2^0)}
	=
	2 \, \nabs{F_\mu (w_k)}_g /  \lambda_k
	\leq 2 \, \nabs{F_\mu (w_k)}_g / \lambda
	\converges{k \to \infty} 0.
\end{align*}
We are left to show \emph{quadratic} convergence. 
Since we know now that the limit point $w_*(\mu)$ exists, this is fairly standard:
Because $F_\mu$ is \hbox{$1$-Lipschitz}, we obtain that
\begin{align*}
	\MoveEqLeft
	\nabs{\dd \varPsi_\mu(w_k) + \Hess(\varPsi_\mu)(w_k)(v_k,\cdot)}_g
	=
	\nabs{- F_\mu(w_k)  + A_{0,k} \, v_k}_g
	=
	\\
	&\leq
	\nabs{- A_{\alpha,k} \, v_k + A_{0,k} \, v_k}_g		
	=
	\alpha \, \nabs{F_\mu(w_k)}^2 \, \nabs{v_k}_g
	\\
	&=
	\alpha \, \nabs{A_{\alpha,k}  \, v_k}_g^2 \, \nabs{v_k}_g
	\leq 
	\alpha \, (1 +\alpha)^2\,\nabs{v_k}_g^3
	\converges{k \to \infty}
	0
	.
\end{align*}
Now \cite[Proposition 5]{MR2968868} shows that the condition \eqref{cond:otherwise}	 enforces $\tau_k = 1$ for all sufficiently large~$k$.
Finally, Propositions 7~and~8 from \cite{MR2968868} imply quadratic convergence of $w_k$ towards $w_*(\mu)$.
\end{proof}

Due to the regularization, the line search is seldom required in practice. In fact, it may cause some problems when run with finite precision: All computations involving the shift transformation suffer a slight loss of precision. So when the slope $f_k'(0) = \dd \varPsi_{\mu_k} \, u_k$ is already very close to $0$, the Armijo condition may just not be justifiable due to fact that $f_k(\tau_k) = \varPsi_{\mu}(w_{k+1})$ cannot be computed arbitrarily well. Fortunately, this typically happens only when the Newton-Kantorovich condition $q_k <1$ is already satisfied (here $q_k \ceq 4 \, \nabs{F_{\mu_k}(0)} / \lambda_k^2$ and $\lambda_k$ is the smallest eigenvalue of $-\nabla F_{\mu_k}(0)$, see also \autoref{thm:quadratic convergence}).
So one is better off  by just putting $\tau_k = 1$ and skipping the check for the Armijo condition whenever $q_k<1$. Putting also $\alpha = 0$ lets the method fall back to Newton's method and its convergence is then guaranteed by \autoref{thm:quadratic convergence}.
Since one has to compute $\nabla F_{\mu_k}(0)$ anyway, computing the smallest eigenvalue does not really increase the complexity of the algorithm when $d$ is small. 
Moreover, it allows to use the condition 
\begin{align}
	q_k<1 \qand 2 \, |\nabla F_{\mu_k}(0)| / \lambda_k < \varepsilon
	\label{eq:StoppingCriterion}
\end{align} 
as reliable stopping criterion and as an a\,posteriori error bound.
Indeed, the residual $|F_{\mu_k}(0)|$ is a very bad predictor of the distance between $w_k$ and $w_*(\mu)$: It typically underestimates the distance and it does so by orders of magnitude when the smallest eigenvalue of $\Hess(\varPsi_\mu)(w_*(\mu))$ is tiny.

It might also be noteworthy that the Abikoff-Ye iteration is of the form
\begin{align*}
	u_k = F_{\mu_k}(0) = - \grad(\varPsi_{\mu_k}(0))
	\qand
	w_{k+1} = \Shift(-w_k, \exp_0(\tau_k \, u_k)) =  \Shift(-w_k, 2\, u_k)
\end{align*}
with step size $\tau_k \to 2$, for $k \to \infty$. 
Thus it is basically the method of steepest descent, and adding a line search as above would also make this method globally convergent (but of course, only with linear convergence rate).

\section{Experimental results}
\label{sec:experiments}

We now give some examples of the performance of our methods in practice. We start with the Douady-Earle extension which was the original motivation for studying conformal barycenters (see \cite{MR0857678}).
Afterwards, we show a couple of examples for polygonal closure which was our initial motivation. 

\subsection{Douday-Earle extension}

\begin{definition}\label{def:DouadyEarleExtension} 
Suppose we are given a nonconstant, continuous map $\gamma \colon \Sphere_1 \to \Sphere_2$ between the unit spheres $\Sphere_1$ and $\Sphere_2$ in the Hilbert spaces $H_1$ and $H_2$. 
Suppose that $H_1$  is finite-dimensional; then
there is a uniques rotation-invariant Borel probability measure~$\nu_0$ on $\Sphere_1$.
For each $z \in \Disk_1$, one defines the measure $\nu_z \ceq \ext \Shift_z^\pull \, \nu_0$ via pullback along the shift transformation $\Shift_z$.
Clearly, this measure has the property that the conformal barycenter of $\nu_z$ is $z$ itself, i.e., $w_*(\nu_z) = z$. 
Since $\gamma$ is nonconstant and continuous, the measure $\gamma_\push \, \nu_z$ cannot be concentrated within two single points. So it is stable and there is a unique conformal barycenter $w_*(\gamma_\push \, \nu_z)$. Thus we may define the \emph{Douady-Earle extension} $\DouadyEarle(\gamma) \colon \Disk_1 \to \Disk_2$ of $\gamma$ by setting
$E(\gamma)(z) \ceq w_*(\gamma_\push \, \nu_z)$.
\end{definition}

\begin{figure}
	\capstart %ensures that hyperlink will jump to the top of this image
\newcommand{\inc}[2]{\begin{tikzpicture}%
    \node[inner sep=0pt] (fig) at (0,0) {\includegraphics{#1}};
	\node[above = -1ex, right= 0ex] at (fig.south west) {{\scriptsize #2}};    
\end{tikzpicture}}%
\presetkeys{Gin}{
	trim = 0 0 0 0, 
	clip = true,  
	angle = 0,
	width = 0.19\textwidth
}{}%
\begin{center}%
	\raisebox{-0.5\height}{\includegraphics{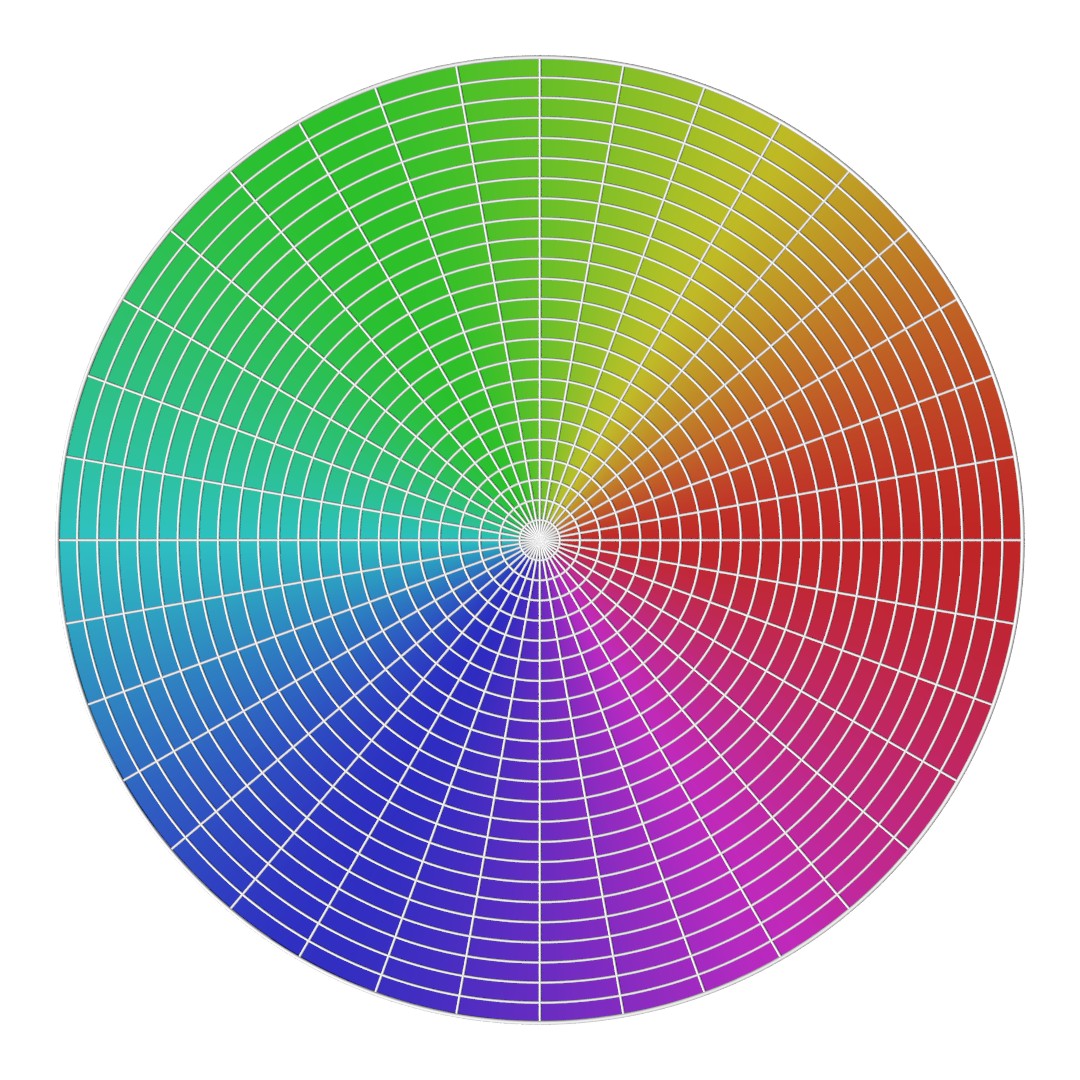}}%
	$\stackrel{\varphi_1}{\longrightarrow}$%
	\raisebox{-0.5\height}{\includegraphics{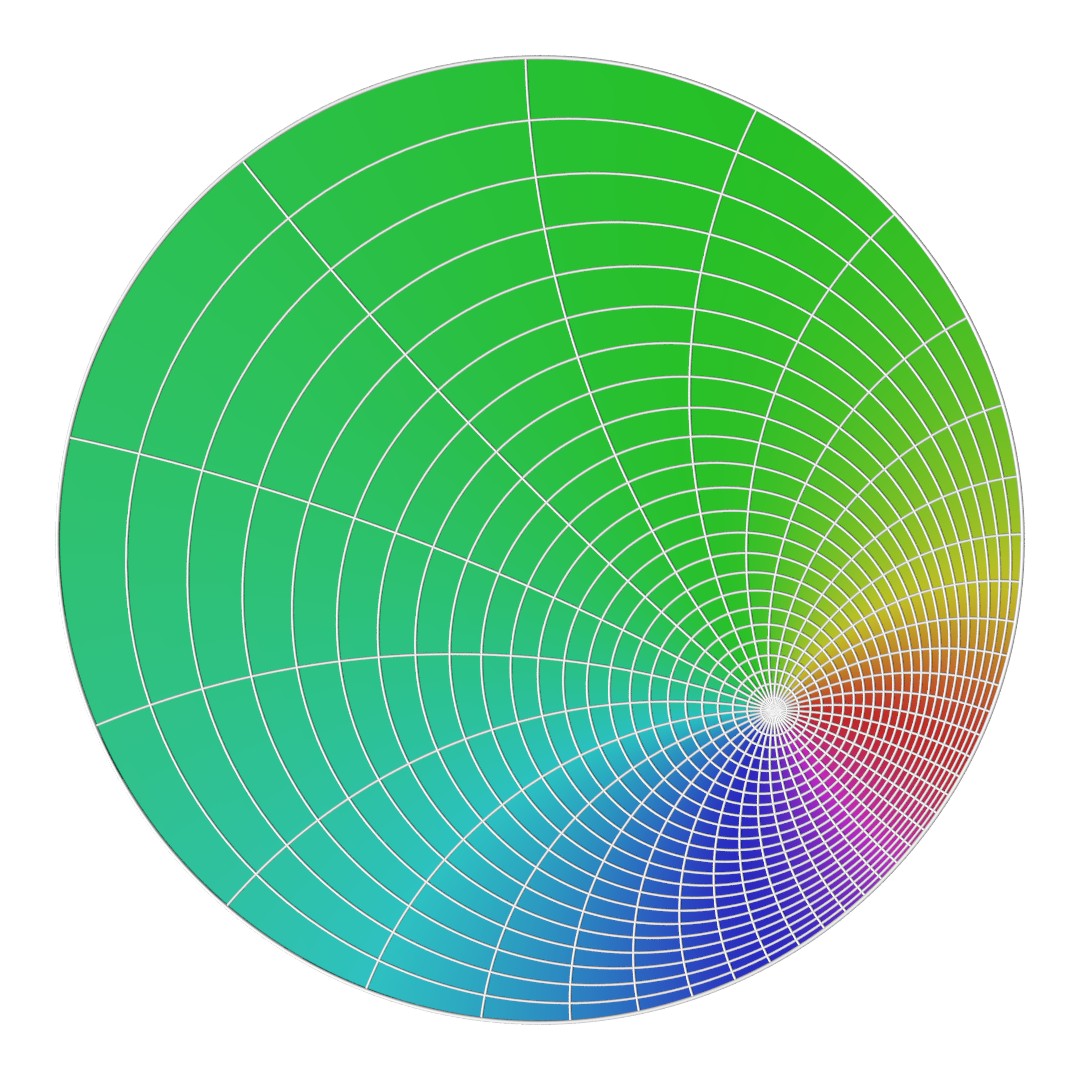}}%
	\raisebox{-0.5\height}{\inc{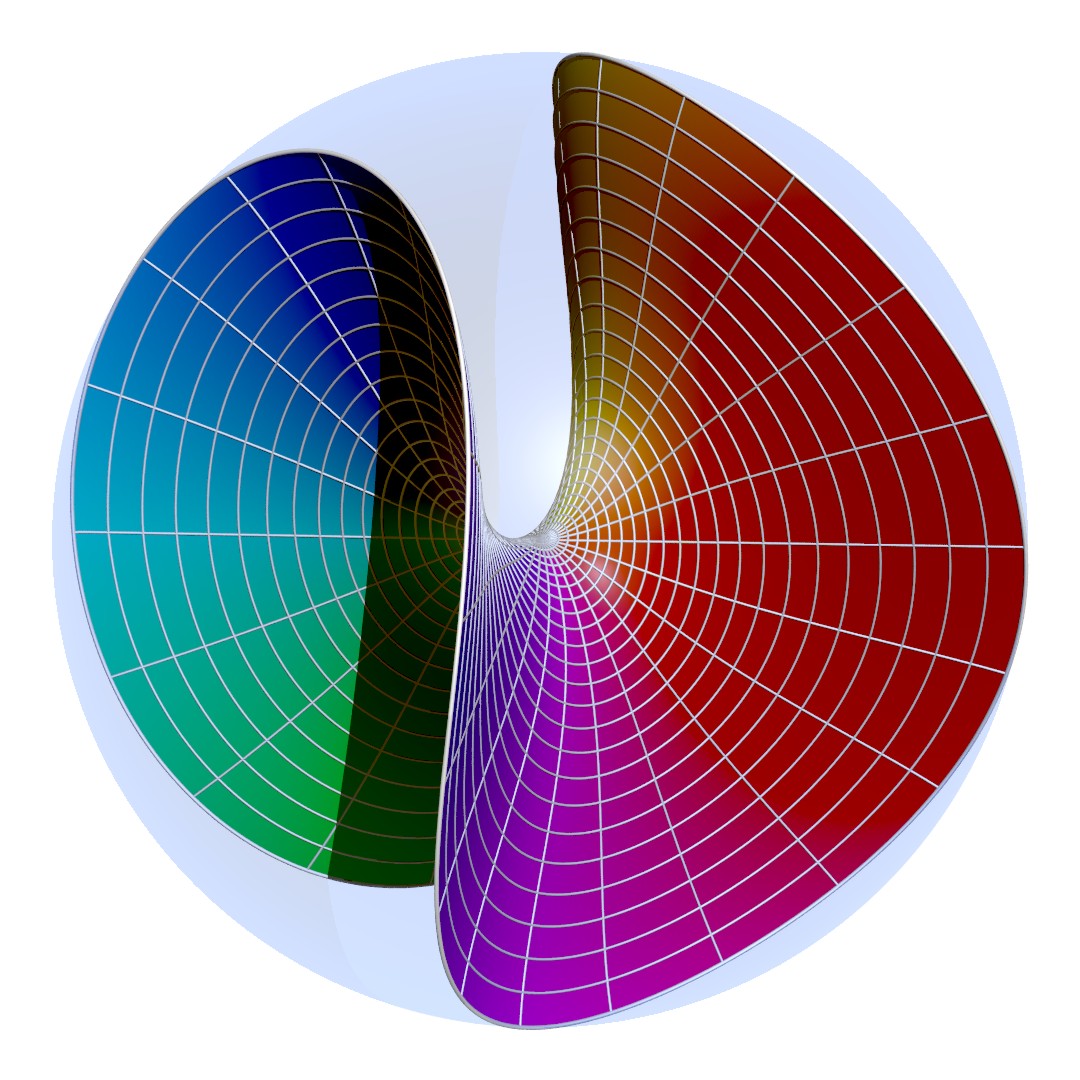}{$E(\gamma)$}}%
	\raisebox{-0.5\height}{\inc{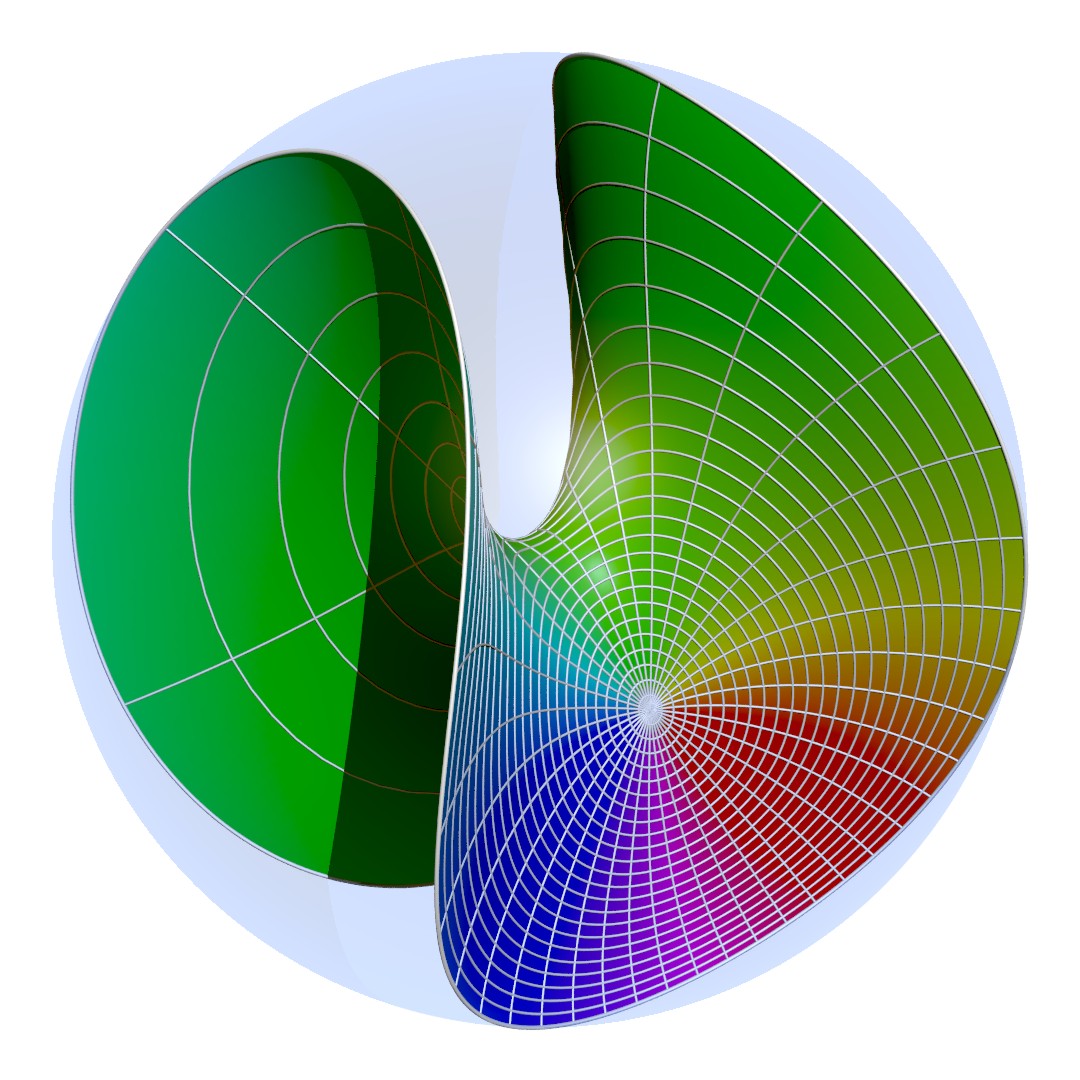}{$E(\gamma\circ \varphi_1)$}}%
	\raisebox{-0.5\height}{\inc{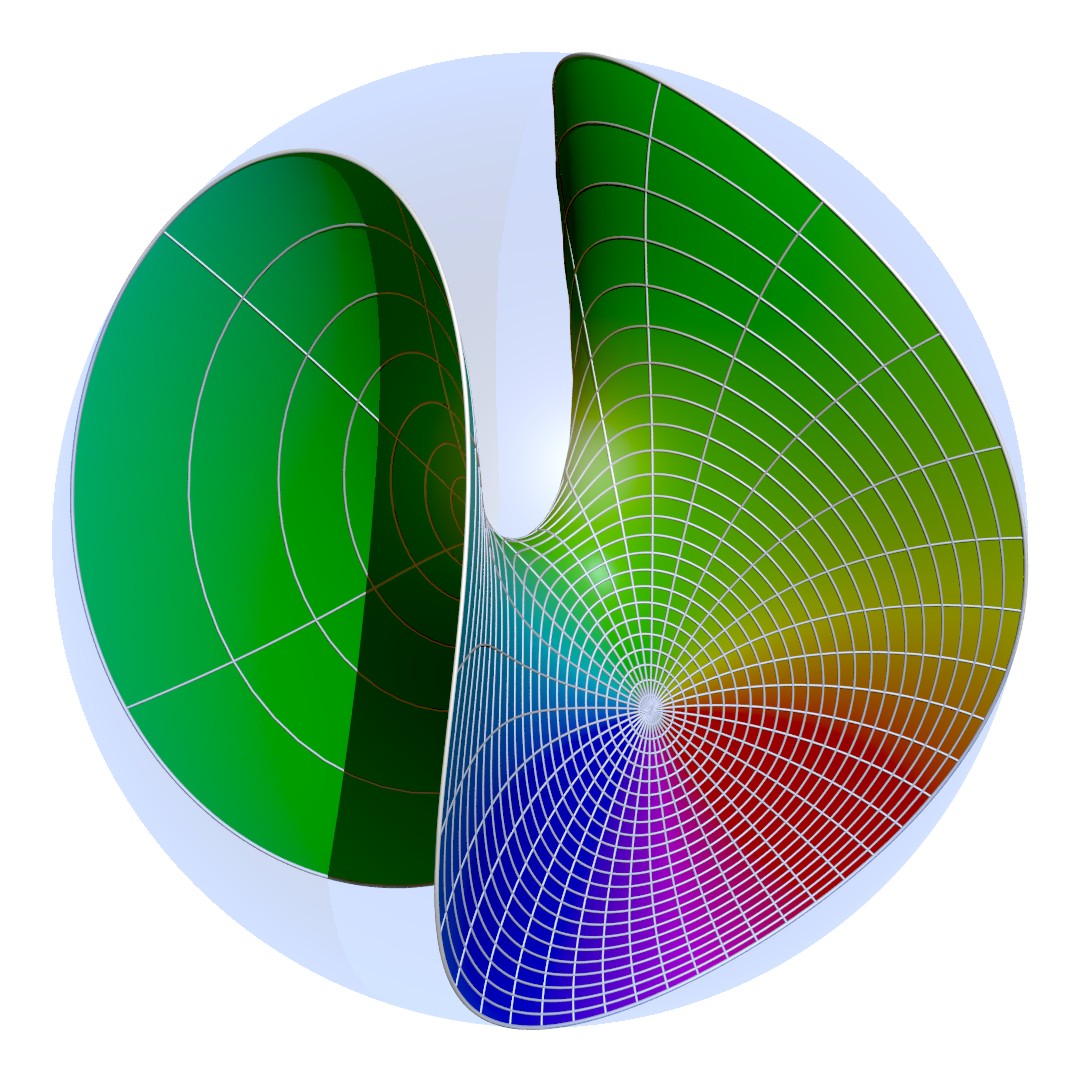}{$E(\gamma)\circ \varphi_1$}}%	
	\\%
	\raisebox{-0.5\height}{\includegraphics{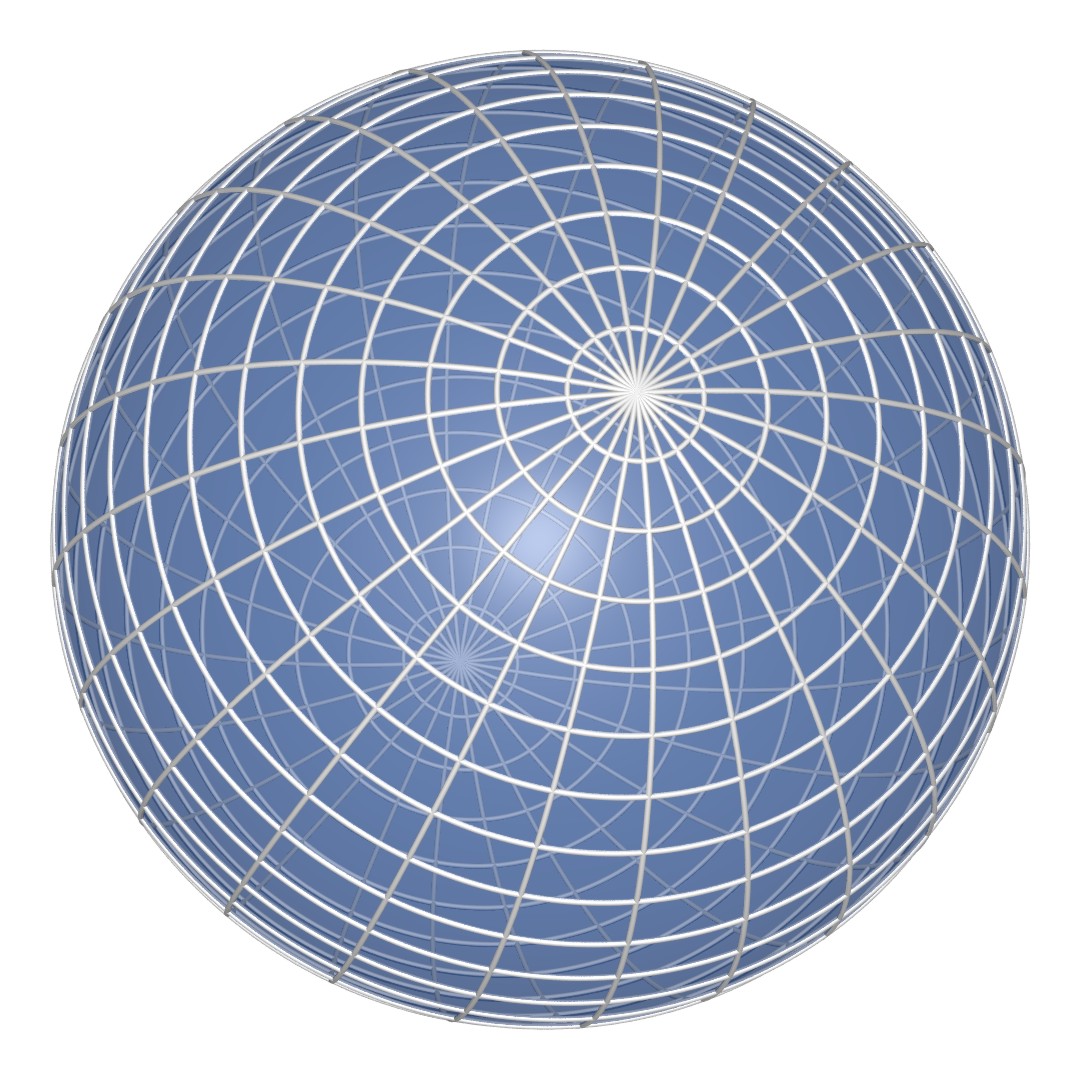}}%
	$\stackrel{\varphi_2}{\longrightarrow}$%
	\raisebox{-0.5\height}{\includegraphics{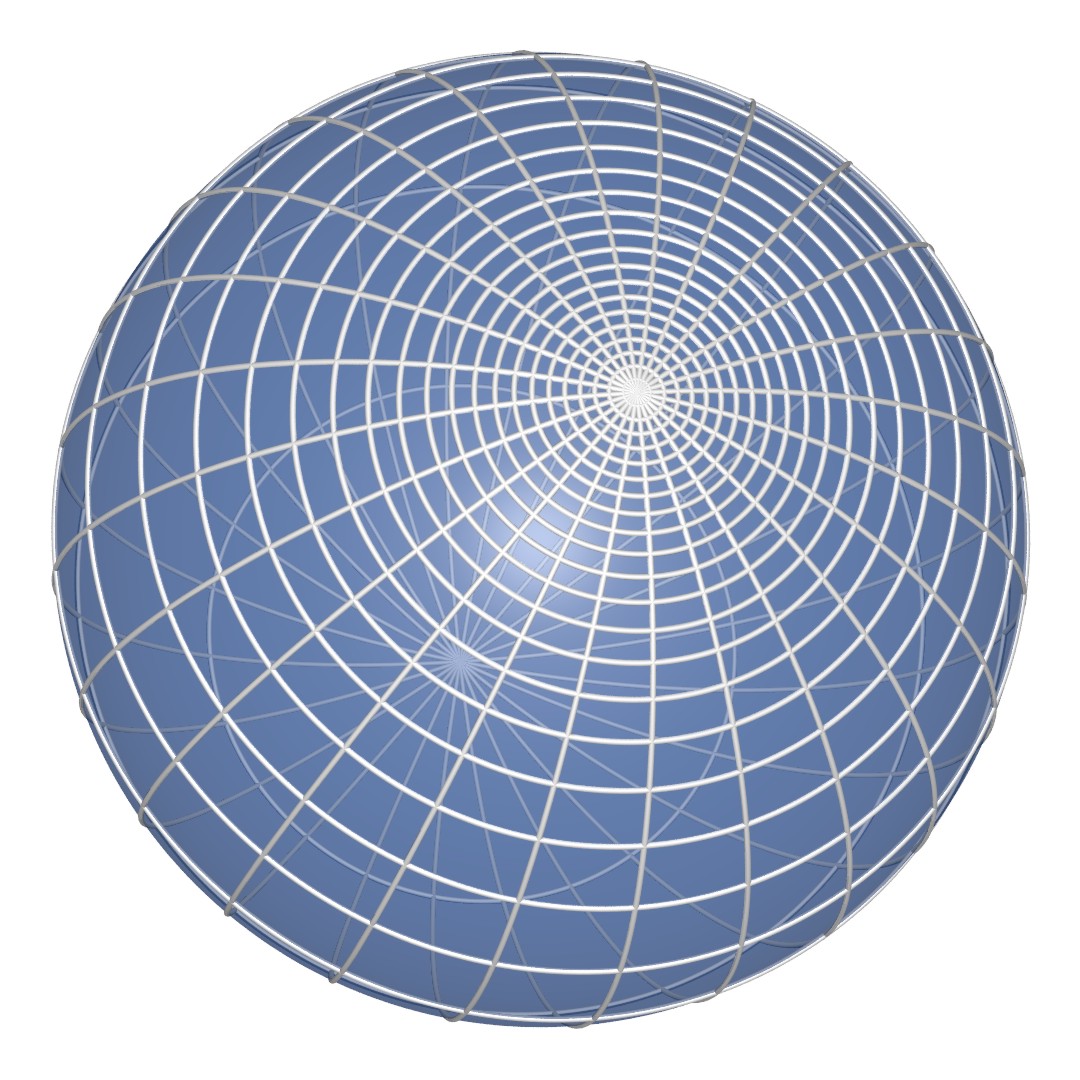}}%
	\raisebox{-0.5\height}{\inc{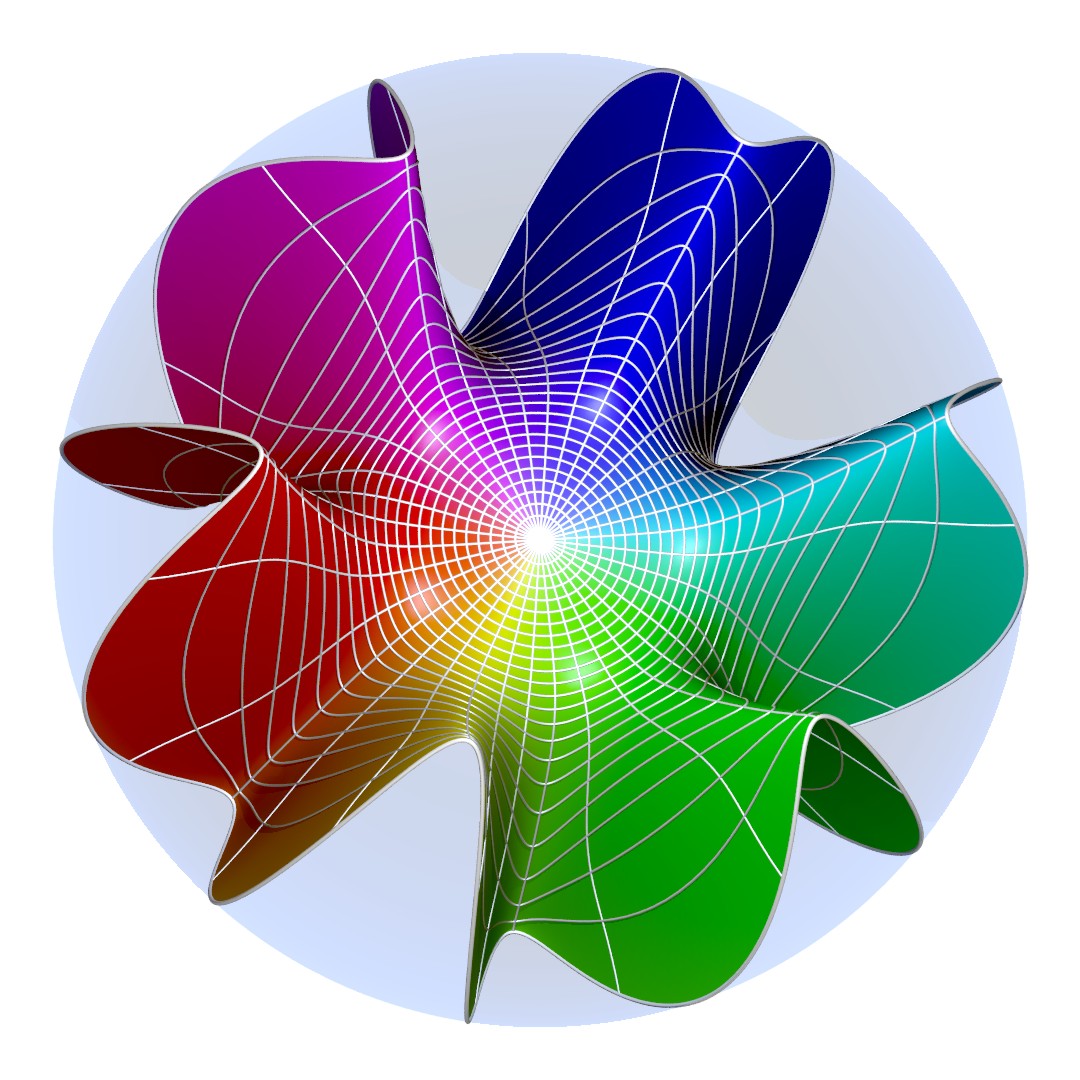}{$E(\gamma)$}}%
	\raisebox{-0.5\height}{\inc{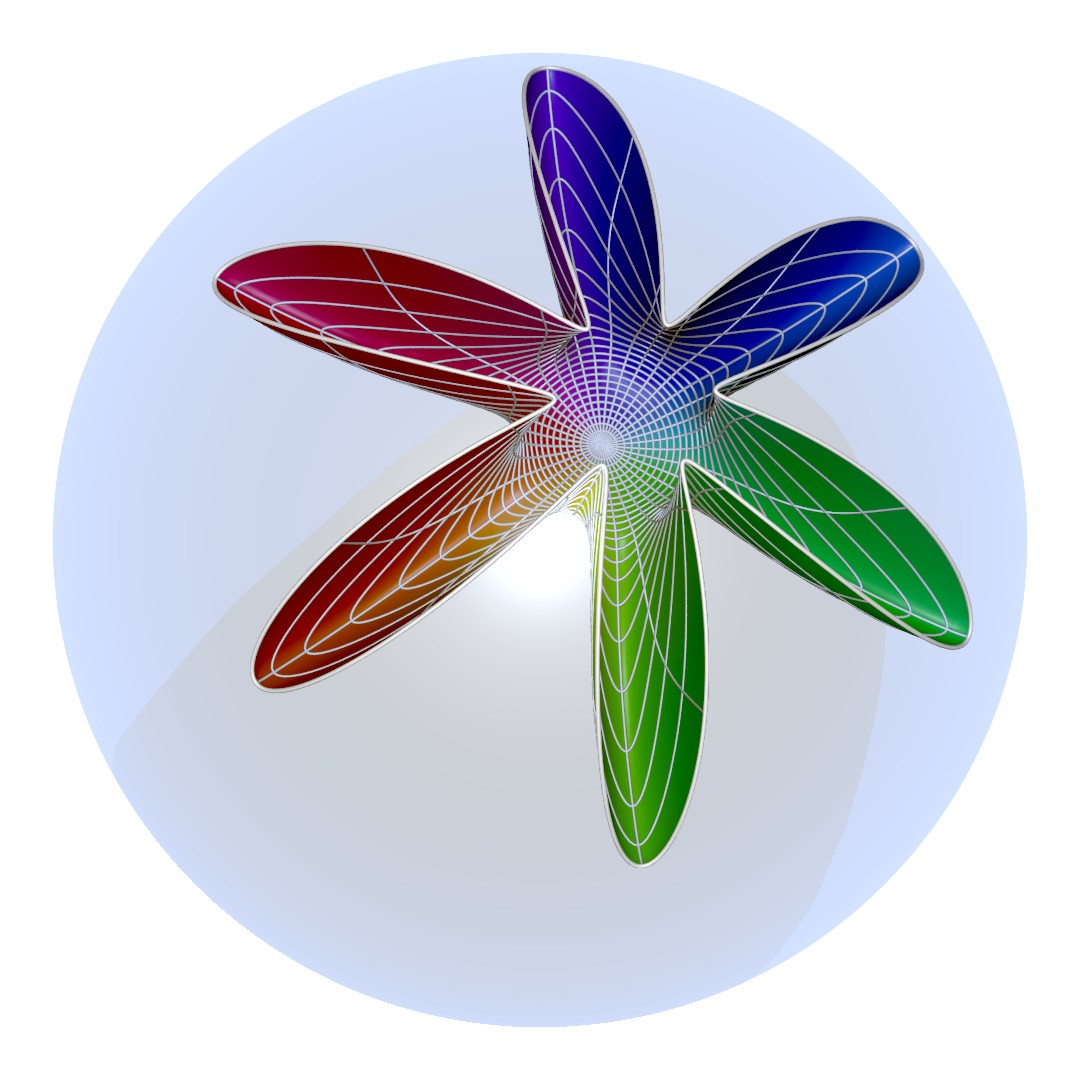}{$E(\varphi_2 \circ \gamma)$}}%
	\raisebox{-0.5\height}{\inc{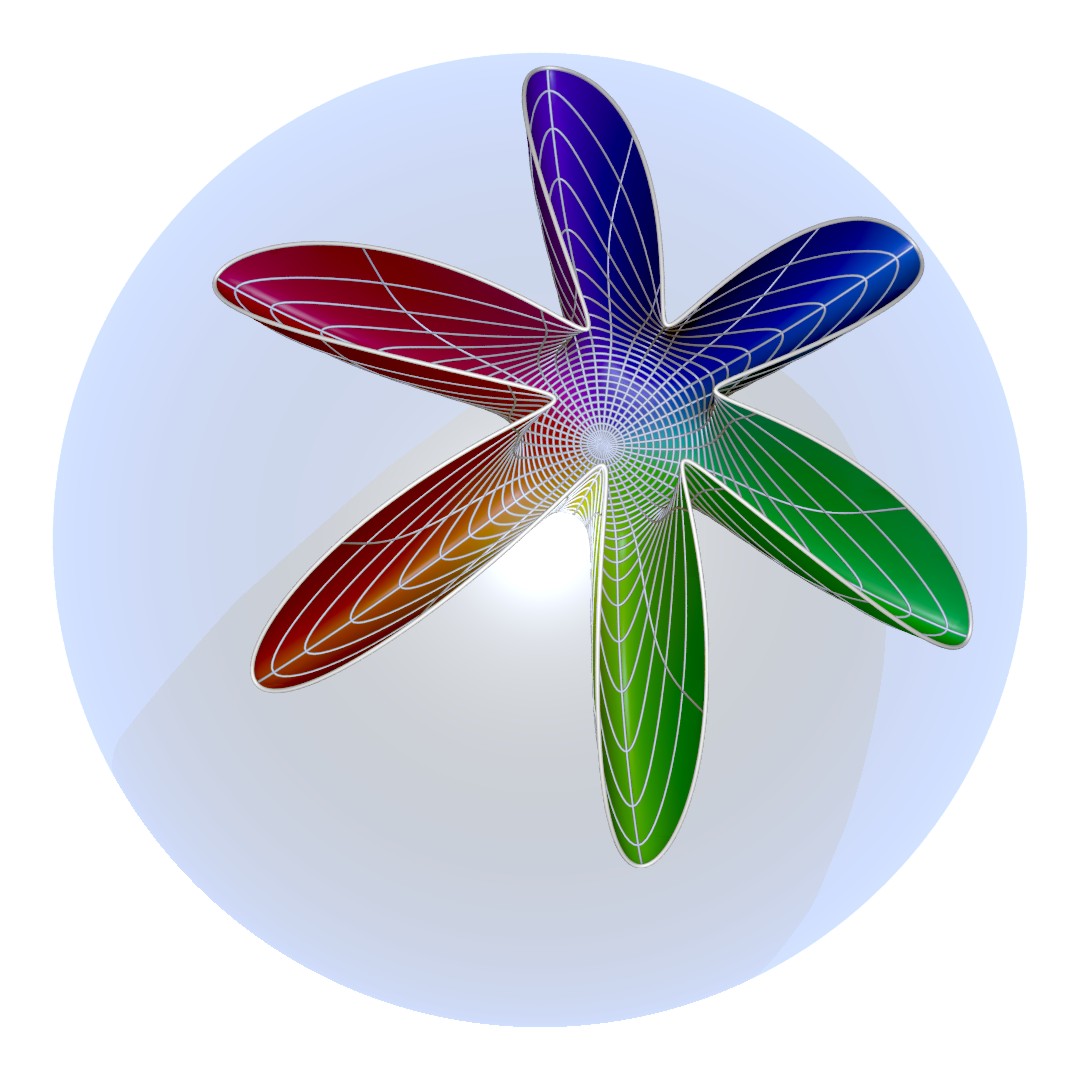}{$\varphi_2 \circ E(\gamma)$}}%
	\caption{The Douady-Earle extensions of several maps from $\Sphere^1$ to $\Sphere^2$. On the top row, we see the effect of precomposing with a M\"obius transformation of $\Disk^2$; on the bottom row the effect of postcomposing with a M\"obius transformation of $\Disk^3$. These surfaces are computed by damped, regularized Newton iterations in hyperbolic 3-space (\cref{subsec:main result}).
}%
	\label{fig:DouadyEarleExtensions}%
\end{center}
\end{figure}
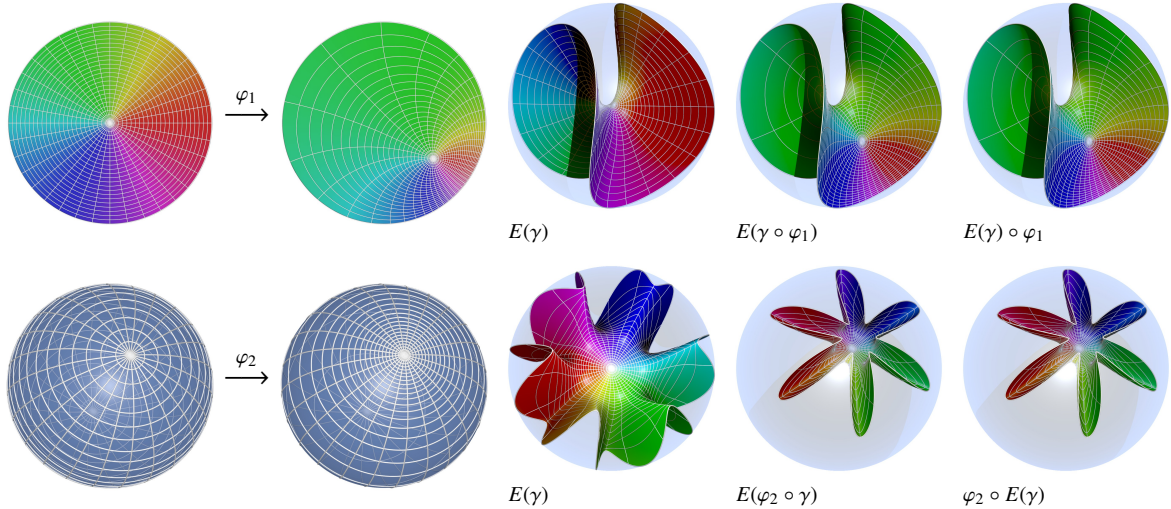

By \eqref{eq:TransformationRulesforF}, the extension operator $E$ is ``conformally natural''.
This means that if $\varphi_1$ is a M\"obius transformation of (the closure of) $\Disk_1$ then $\DouadyEarle(\gamma \circ \varphi_1) = \DouadyEarle(\gamma) \circ \varphi_1$. Further, if $\varphi_2$ is a M\"obius transformation of (the closure of) $\Disk_2$ then $\DouadyEarle(\varphi_2 \circ \gamma) = \varphi_2 \circ \DouadyEarle(\gamma)$. These properties are illustrated in~\cref{fig:DouadyEarleExtensions}.

One can approximate $\DouadyEarle(\gamma)$ by approximating $\nu_0$ by a discrete $n$-point measure $\nu_{0,n}$.
If $\gamma$ is a sufficiently smooth and for fixed $z \in \Disk_1$, 
the conformal barycenter $\DouadyEarle_n(\gamma)(z)$ 
of $\gamma_\push \, \ext \Shift_z^\pull \,\nu_{0,n}$ will converge to $\DouadyEarle(\gamma)(z)$ provided that $\nu_{0,n}$ converges in \hbox{$1$-Wasserstein} distance to $\nu_0$.
Approximating $\nu_0$ is particularly easy if $\Sphere_1 = \Sphere^1$ is the $1$-dimensional sphere: We may choose uniformly distributed quadrature points and put $\nu_{0,n} = \tfrac{1}{n} \sum_{i=1}^n \updelta( \cos(\tfrac{2 \pi}{n}),\sin(\tfrac{2 \pi}{n}))$.
Since $\gamma_\push \, \ext \Shift_z^\pull \,\nu_{0,n}$ is also a discrete measure, we may compute its conformal barycenter by the method outlined in \cref{sec:NewtonRegularized}.
The reader might find it intriguing to try out the \emph{Mathematica} routine \texttt{DouadyEarleExtension}
provided by the package \texttt{ConformalBarycenter.m} in the electronic supplement.\footnote{See  \href{https://github.com/HenrikSchumacher/ConformalBarycenter}{https://github.com/HenrikSchumacher/ConformalBarycenter} for a maintained version.}
This routine computes $\DouadyEarle_n(\gamma)$ of a ``piecewise-linear'' curve $\gamma \colon \Sphere^1 \subset \R^2 \to \Sphere^2 \subset \R^3$. It is the very routine that we used to produce \cref{fig:DouadyEarleExtensions} and \cref{fig:DouadyEarleExtensions_by_Milnor}. 

This application is actually quite challenging: For $z$ close to the boundary of $\Disk_1$, the measure $\nu_z$ on $\Sphere_1$ and its pushforward $\gamma_\push \nu_z$ on $\Sphere_2$ are highly concentrated.
This means that the Newton-Kantorovich condition from~\cref{thm:quadratic convergence} may be far from being satisfied at the starting point because the Hessian is very degenerate. 
The method of steepest descent is notorious for having very oscillatory behavior
and for not making good progress in such regions.
Hence the method of Milnor-Abikoff-Ye should suffer immensely from degenerate Hessians.
Indeed, the left hand side of \cref{fig:DouadyEarleExtensions_by_Milnor} shows that the Milnor-Abikoff-Ye iteration has problems with computing the Douady-Earle extension in the boundary regions of the surfaces. There, 
the algorithm did not converge even after $1000$ iterations. 
Despite this, the damped, regularized Newton method (with regularization parameter $\alpha = 1$ and initialized with the Euclidean center of mass of $\gamma_\push \, \nu_z$) required typically less than a dozen iterations to decrease the Newton-Kantorovich error bound $\tfrac{1}{2} \, q \, \lambda_{\min}$ (see \cref{thm:quadratic convergence}) below $10^{-8}$ (see right hand side of \cref{fig:DouadyEarleExtensions_by_Milnor}). 
In the depicted setting, the surface is discretized by triangle meshes with $64309$ vertices and 
the uniform measure $\nu_0$ on $\Sphere^1$ is discretized by $n = 720$ quadrature points.
Our test machine\footnote{Intel Core i7 4980HQ CPU (2,8 GHz Quad-Core) with 16\,GB RAM.}
performed the task in about 2.6 seconds.

\begin{figure}
\begin{center}
	\capstart %ensures that hyperlink will jump to the top of this image
\presetkeys{Gin}{
	trim = 0 0 0 0, 
	clip = true,  
	angle = 0,
	width = 0.25\textwidth
}{}%
	\hfill
	\includegraphics{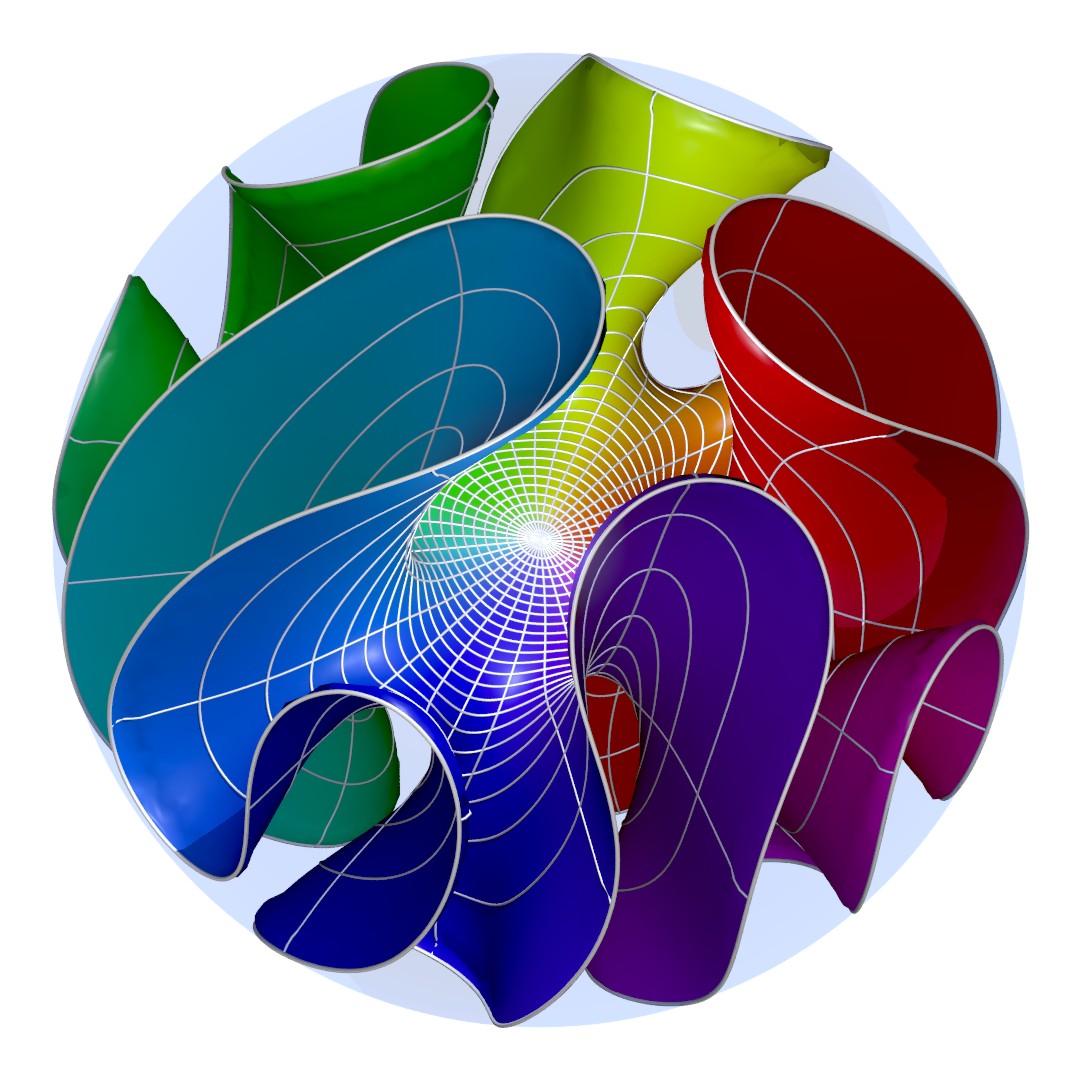}%
	\hfill	
	\includegraphics{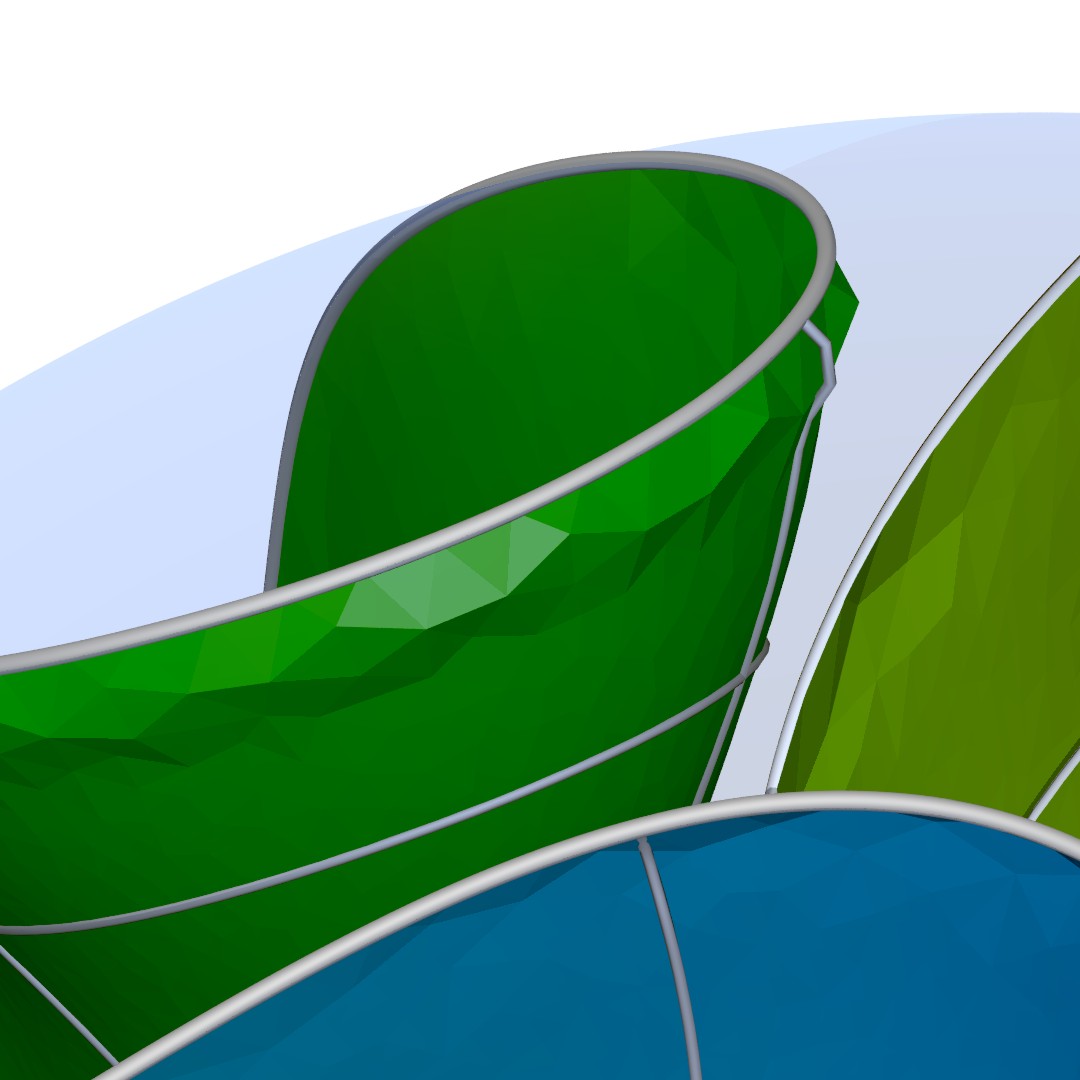}%	
	\hfill	
	\includegraphics{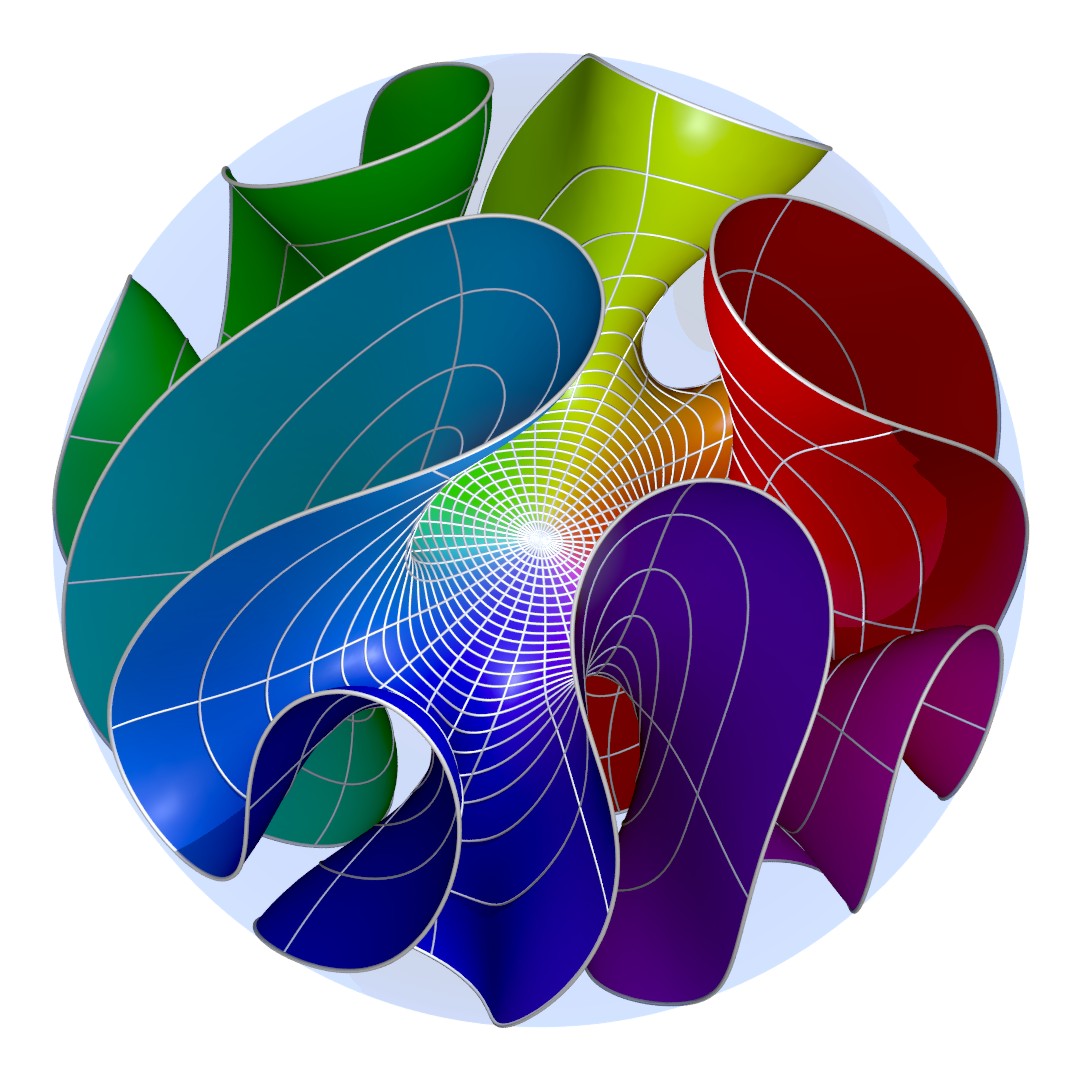}%
	\hfill	
	\includegraphics{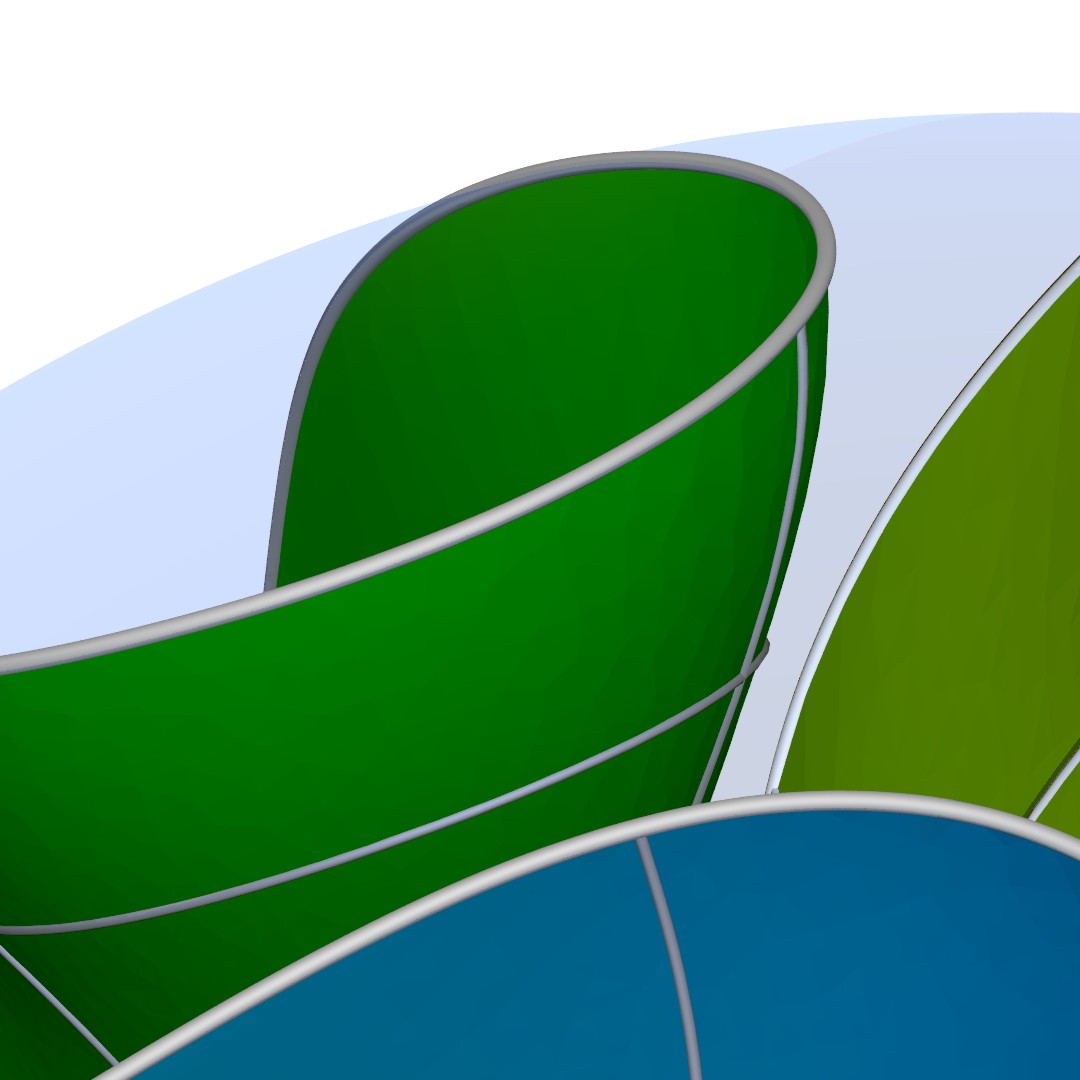}%	
	\hfill	
	\caption{\emph{Left:} Discrete Douady-Earle extension of curve on the 2-sphere, computed by Milnor-Abikoff-Ye method which was stopped after at most $1000$ iterations. 
	The close-up (with interpolation of surface normals deactivated) shows a lot of noise in the vicinity of the boundary and reveals that the algorithm has failed to converge.
	\emph{Right:} 
	In contrast, our regularized Newton method computes the points of the extension robustly and quickly (in at most four iterations and in about two iterations on average for a tolerance of order $10^{-8}$).
	}
	\label{fig:DouadyEarleExtensions_by_Milnor}
\end{center}
\end{figure}

\subsection{Polygon Closure}

\begin{figure}
	\capstart %ensures that hyperlink will jump to the top of this image
\begin{center}%
\begin{minipage}[c]{0.18\textwidth}%
\begin{center}%
	\includegraphics[trim = 0 60 0 60, 
		clip = true,  
		angle = 0,
		width = \textwidth
	]{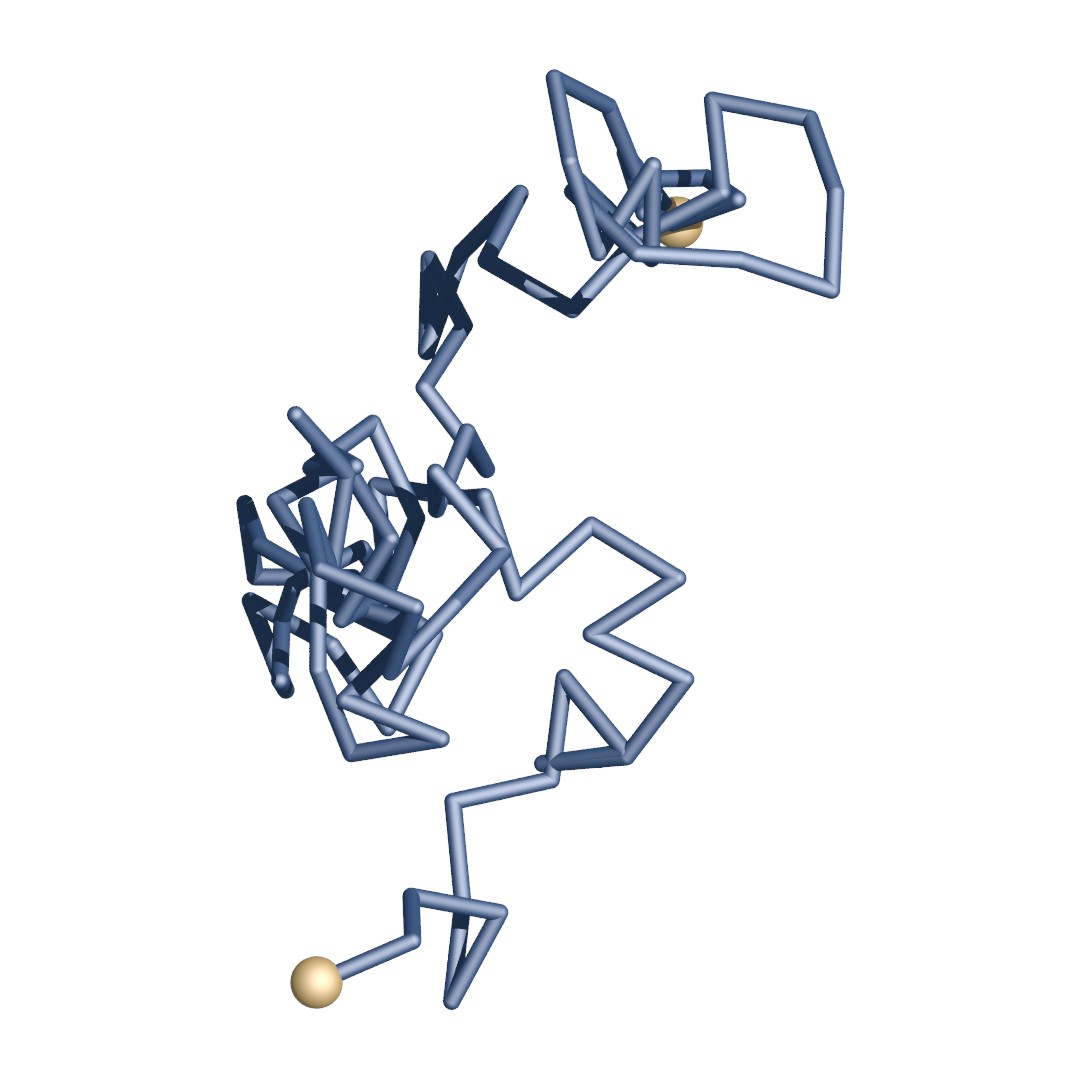}%
\\[-1ex]%
$\downarrow$%
\\[-.1ex]%
	\includegraphics[trim = 0 250 0 250, 
		clip = true,  
		angle = 0,
		width = \textwidth
	]{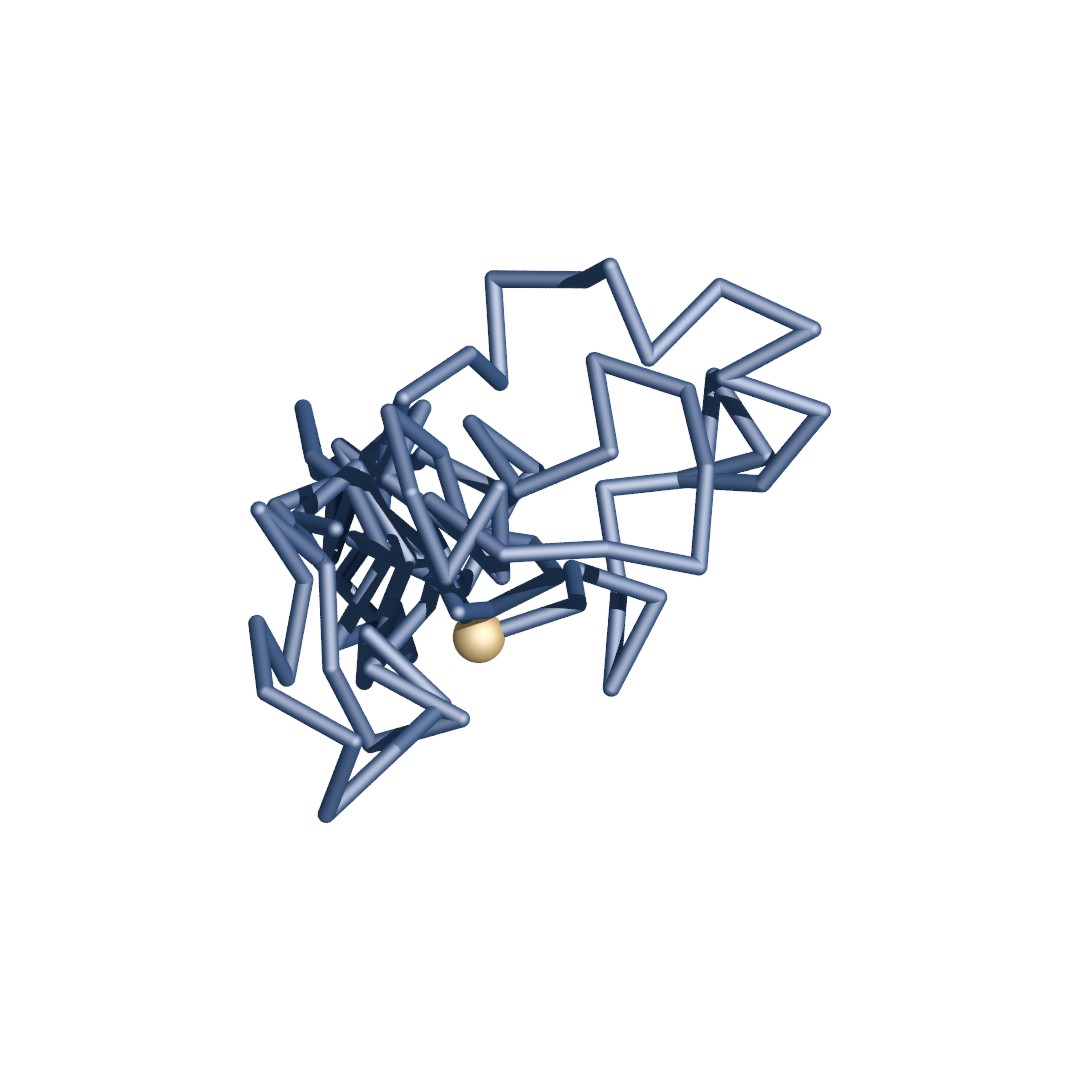}%
\end{center}%
\end{minipage}%
	\hfill
\presetkeys{Gin}{
	trim = 0 0 0 0, 
	clip = true,  
	angle = 0,
	height = 0.18\textheight
}{}%
	\raisebox{-0.5\height}{\includegraphics{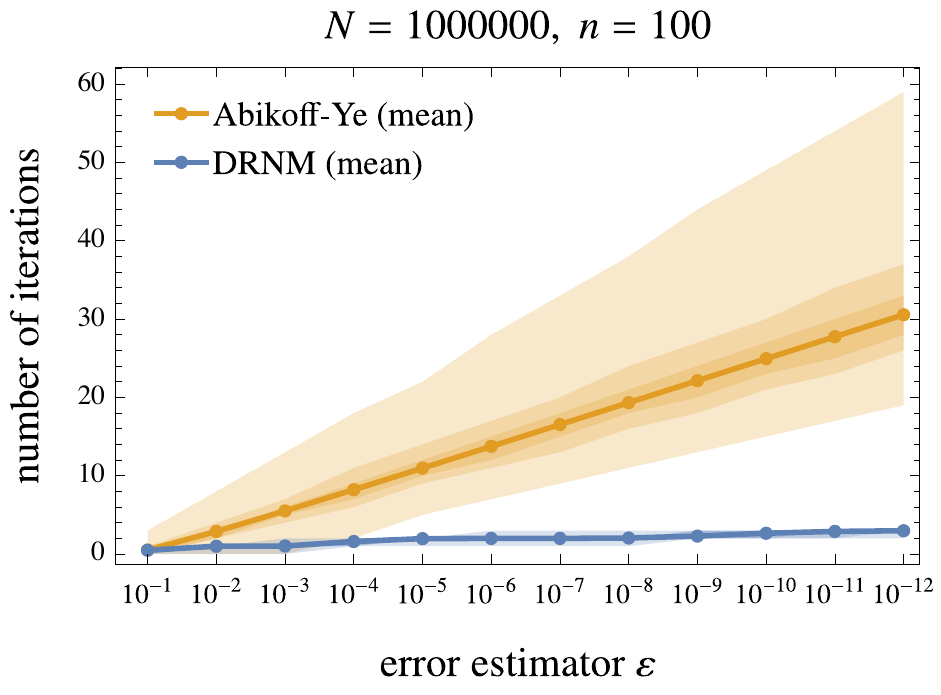}}%
	\hfill
	\raisebox{-0.5\height}{\includegraphics{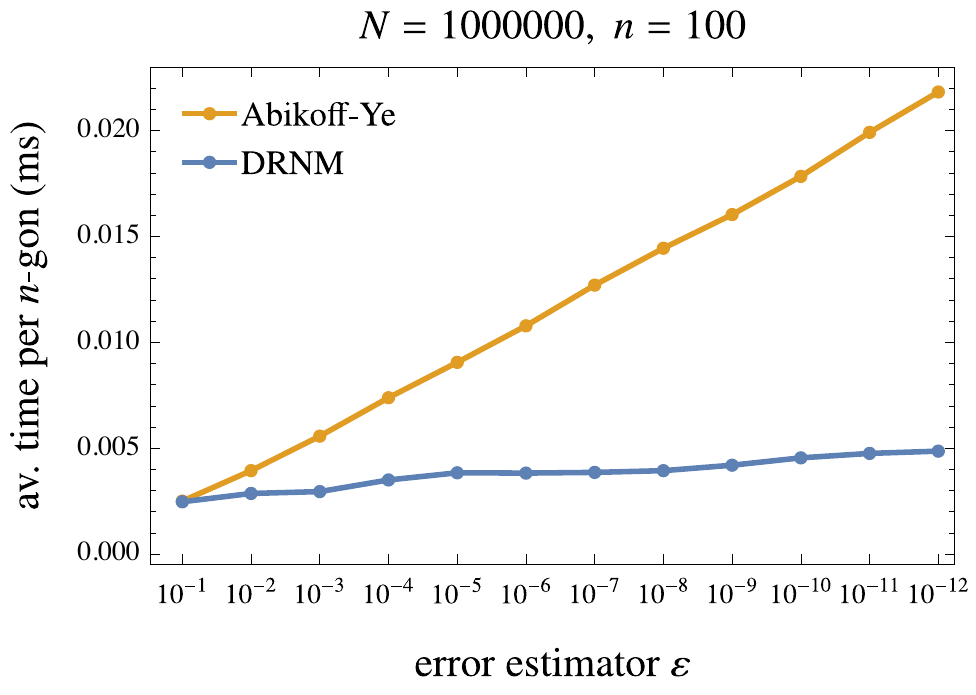}}%	
	\\[1ex]
\begin{minipage}[c]{0.18\textwidth}%
\begin{center}%
	\includegraphics[trim = 0 300 0 300, 
		clip = true,  
		angle = 0,
		width = \textwidth
	]{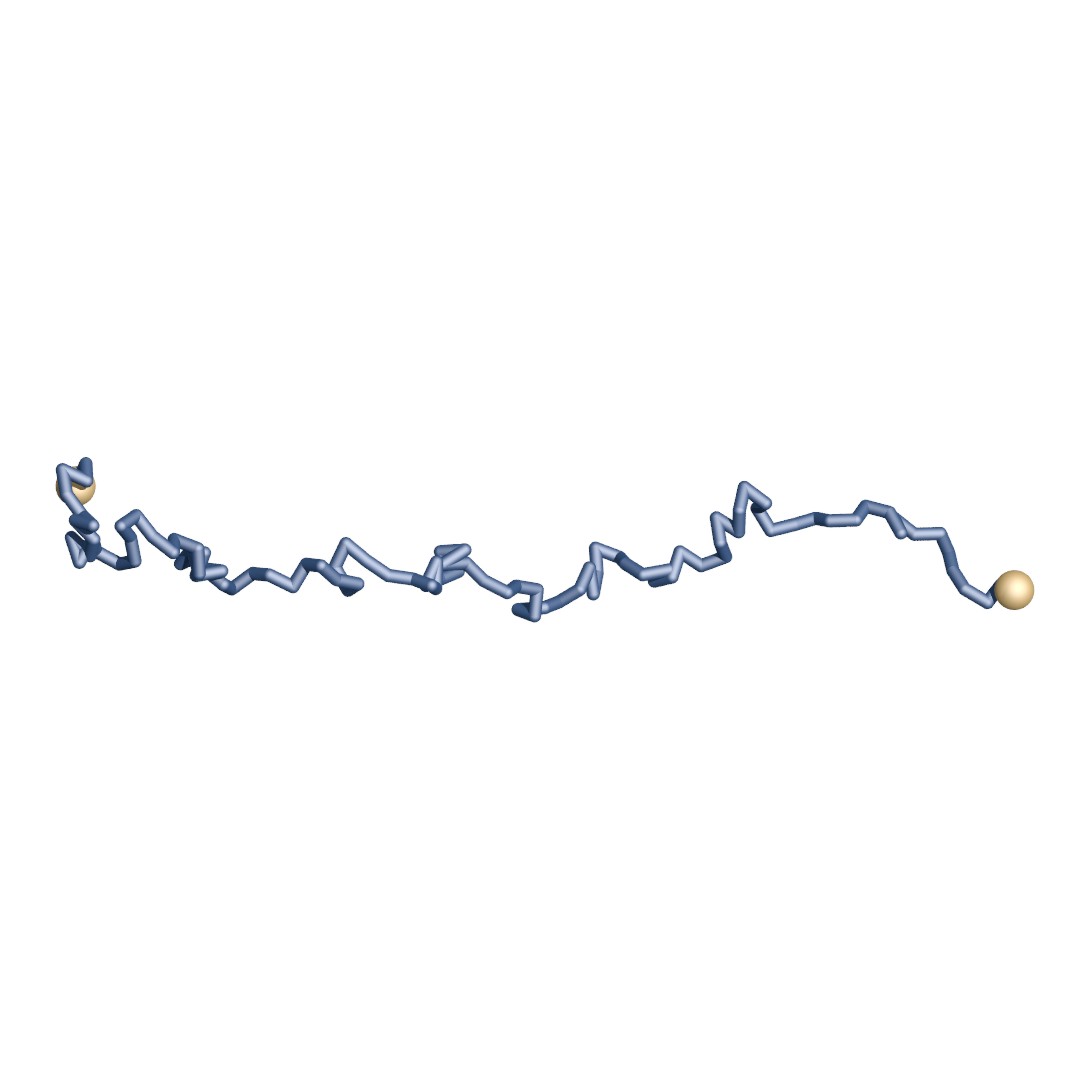}%
\\[-1ex]%
$\downarrow$%
\\[-.1ex]%
	\includegraphics[trim = 0 320 0 300, 
		clip = true,  
		angle = 0,
		width = \textwidth
	]{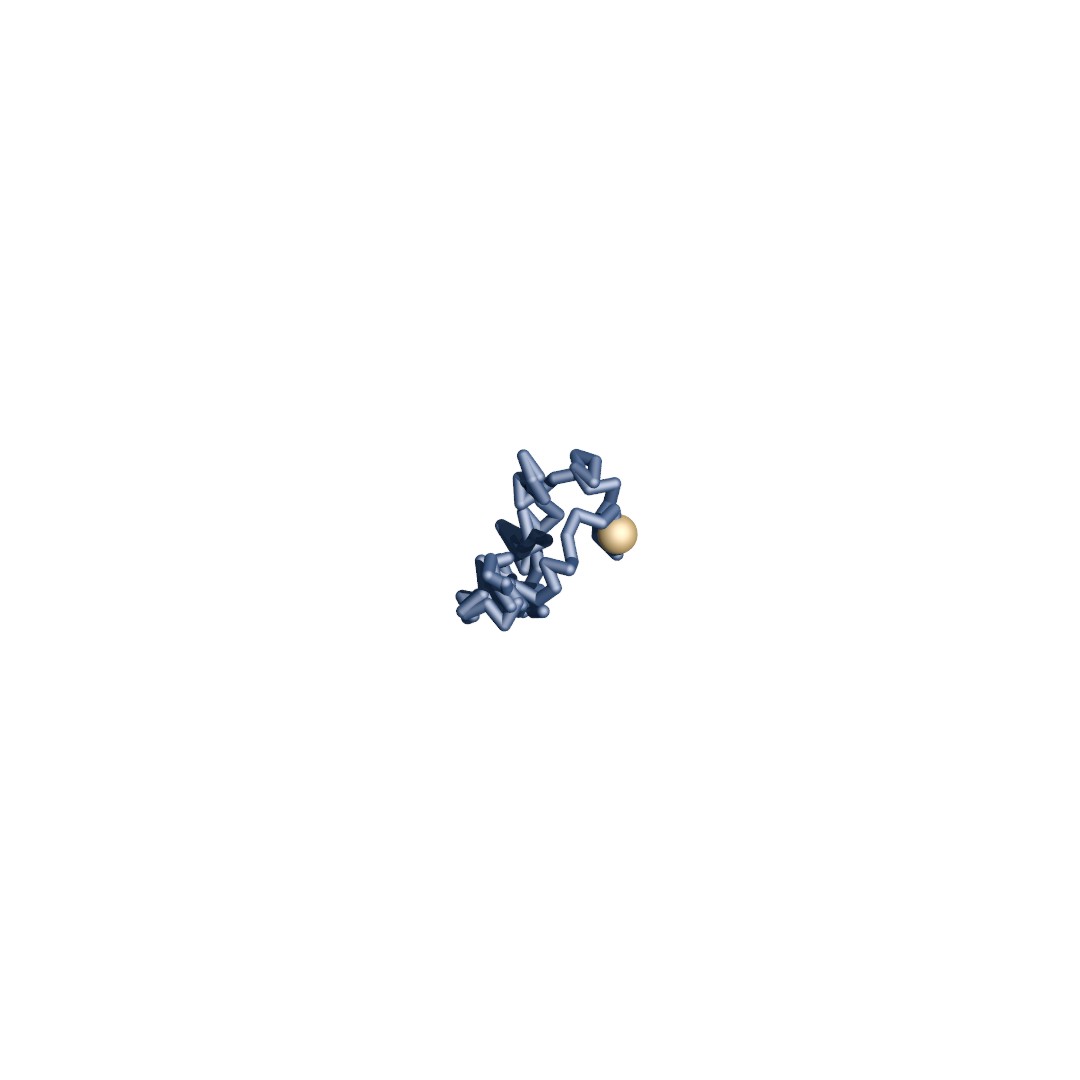}%
\end{center}%
\end{minipage}%
	\hfill
\presetkeys{Gin}{
	trim = 0 0 0 0, 
	clip = true,  
	angle = 0,
	height = 0.18\textheight
}{}%
	\raisebox{-0.5\height}{\includegraphics{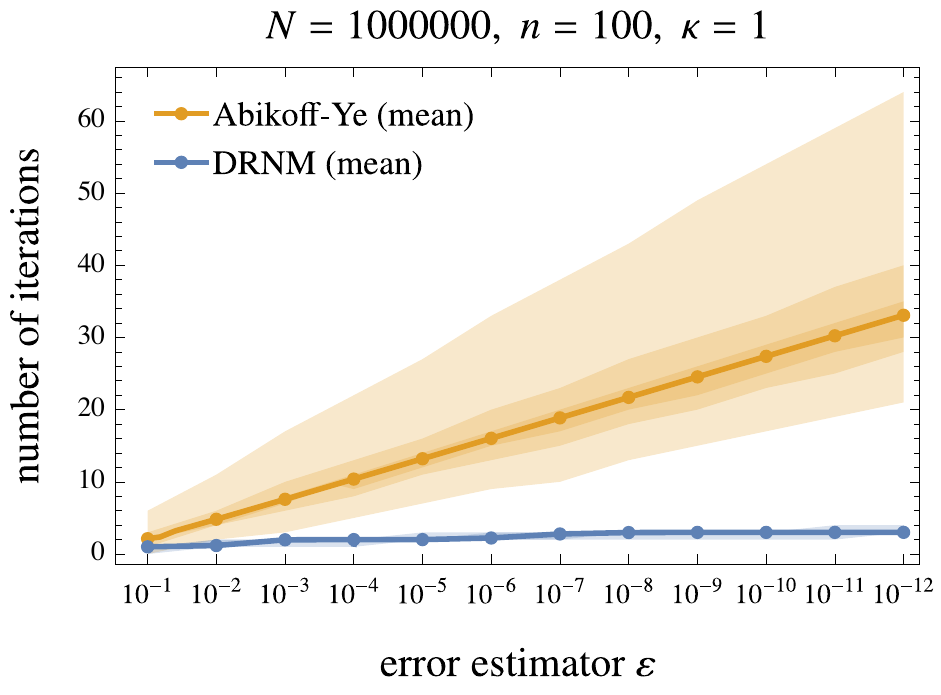}}%
	\hfill
	\raisebox{-0.5\height}{\includegraphics{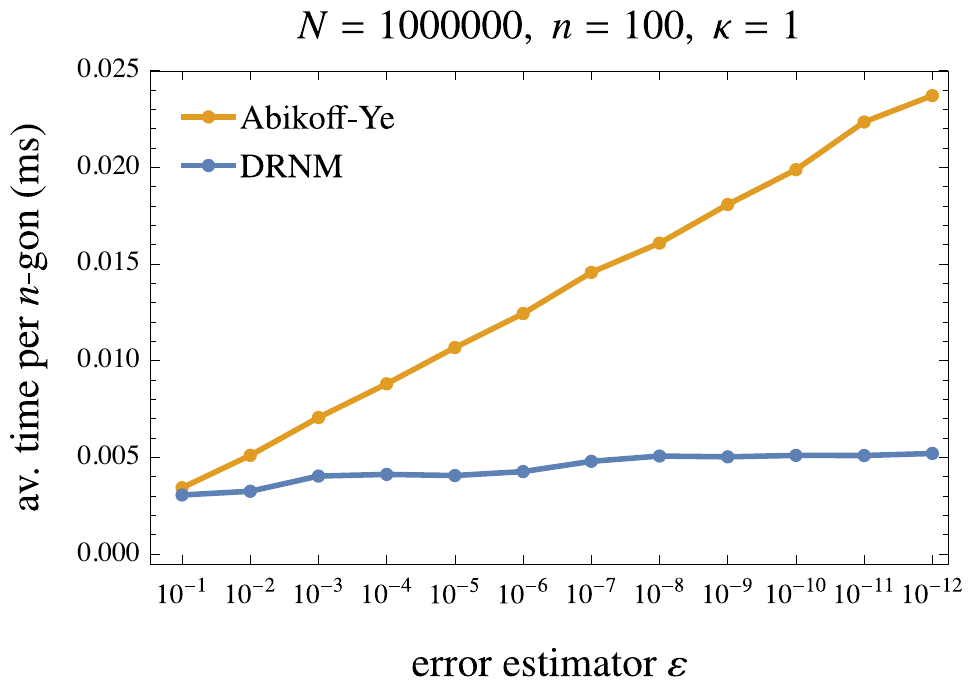}}%	
\\[1ex]%
\begin{minipage}[c]{0.18\textwidth}%
\begin{center}%
	\includegraphics[trim = 0 300 0 300, 
		clip = true,  
		angle = 0,
		width = \textwidth
	]{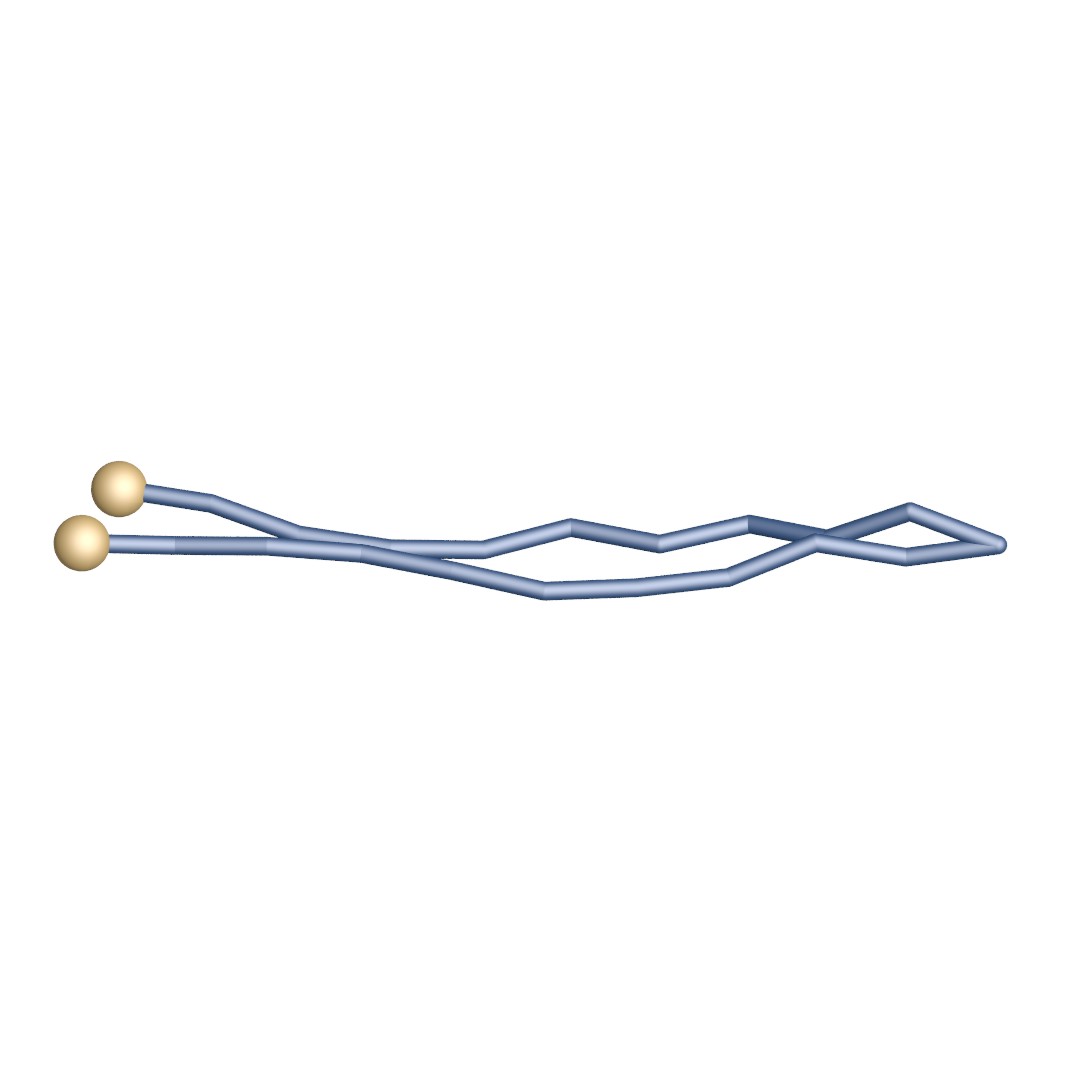}%
\\[-1ex]%
$\downarrow$%
\\[-.1ex]%
	\includegraphics[trim = 0 320 0 300, 
		clip = true,  
		angle = 0,
		width = \textwidth
	]{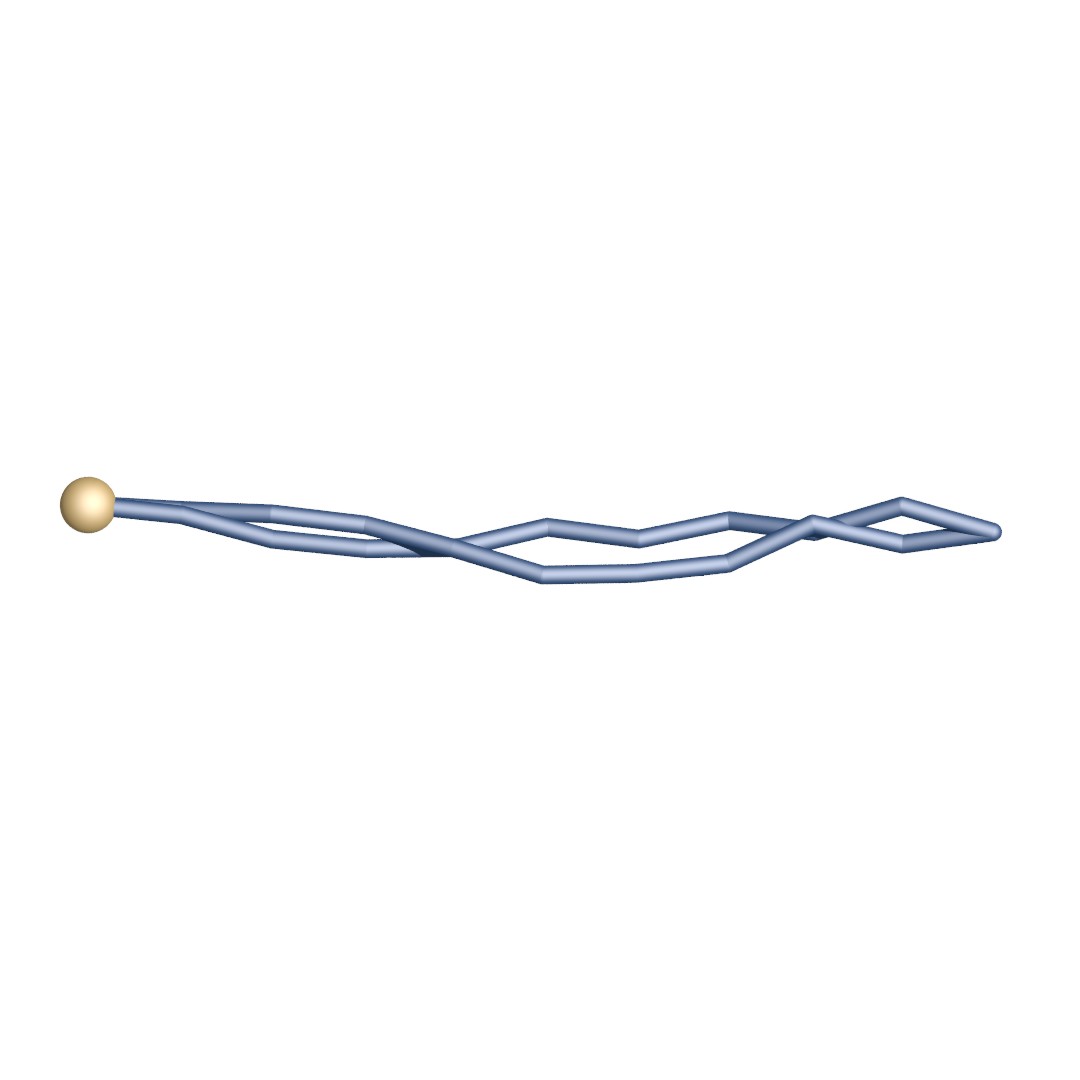}%
\end{center}%
\end{minipage}%
	\hfill
\presetkeys{Gin}{
	trim = 0 0 0 0, 
	clip = true,  
	angle = 0,
	height = 0.18\textheight
}{}%
	\raisebox{-0.5\height}{\includegraphics{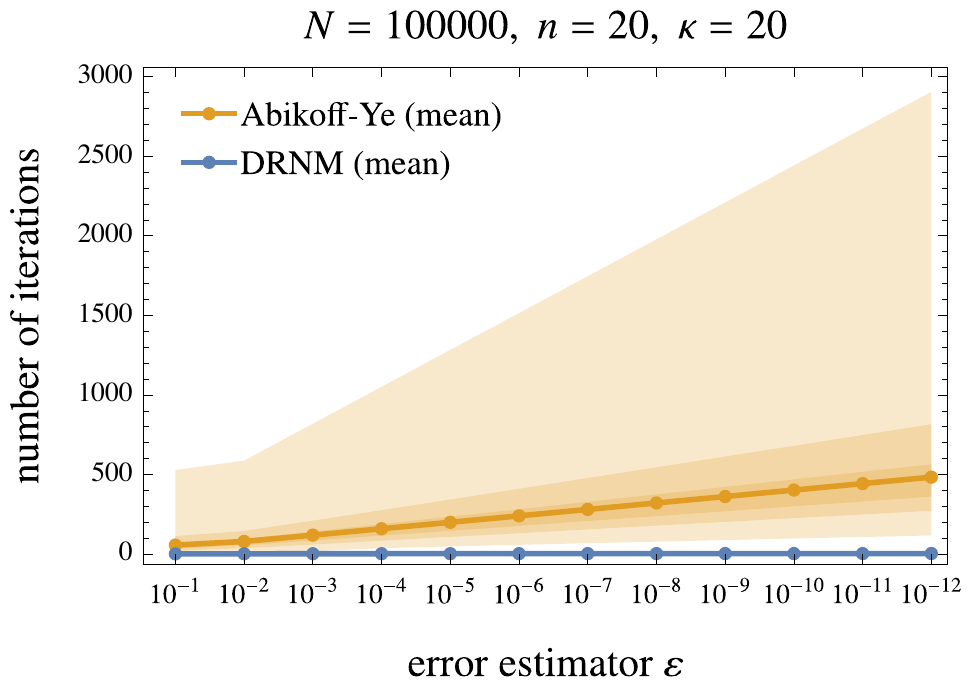}}%
	\hfill
	\raisebox{-0.5\height}{\includegraphics{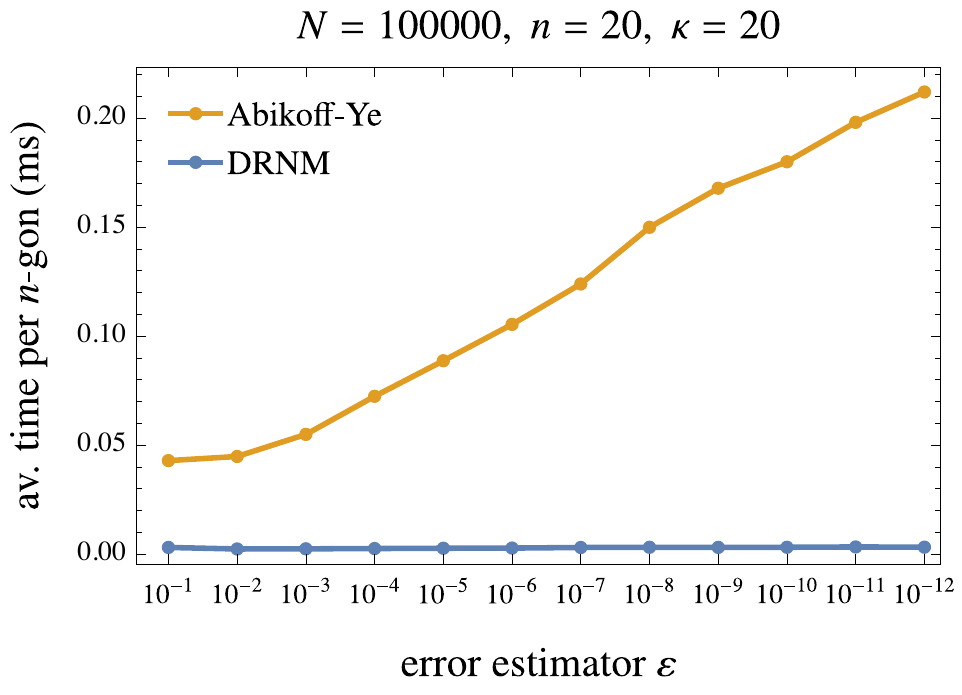}}%	
	\label{fig:PolygonClosureExamples}%
\end{center}%
\caption{%
Performance comparison between Abikoff-Ye iteration and the damped, regularized Newton method (DRNM).
We applied both methods to $N$ randomly generated polygonal lines with $n$ equal length edges and recorded the average number of iterations and the average time per polygonal line.
The shaded regions in the iteration plots indicate the region between the $p/2$- and $(100-p/2)$-percentiles, where $p \in \set{0,10,50}$.
\emph{Top:} Unit edge vectors uniformly sampled over $\Sphere^2$. 
\emph{Center:} 
Unit edge vectors sampled from the von Mises-Fisher distribution with $\kappa =1$.
\emph{Bottom:} 
An example that illustrates that ``almost closed'' does not necessarily imply ``easy to close''. 
DRNM required at most 5 iterations per polygon whereas Abikoff-Ye required almost 3000 iterations in some cases.
}
\end{figure}

We next present the results of experiments on polygon closures. 
Recall from~\cref{sec:intro} that Millson-Kapovich showed that the space of polygons with fixed edgelengths is a quotient of the space of nice semi-stable point measures on the sphere by the action of the conformal group. 
So we can close a polygonal line by computing the conformal barycenter $w_*(\mu)$ of $\mu = \sum_{i=1}^n \omega_i \, \updelta(x_i)$ and by applying the shift transformation $\ext \Shift_{w_*(\mu)}$ to the unit edge vectors $x_1,\dotsc,x_n$. Here $\omega_1,\dotsc,\omega_n$ are the edgelengths of the polygonal line, normalized to satisfy $\omega_1 + \dotsm + \omega_n =1$.

\cref{fig:PolygonClosureExamples}	 shows performance results of Abikoff-Ye iteration and of the damped, regularized Newton method (DRNM) that we analyzed here.
Both algorithms are accessible from the routine \texttt{ConformalBarycenter} in the aformentioned \emph{Mathematica} package.
They were applied to polygonal lines whose unit edge vectors were randomly sampled over the unit sphere $\Sphere^2$: 
In the first example we sampled the unit edge vectors uniformly over $\Sphere^2$. 
In the second example, we sampled from the spherical von Mises-Fisher distribution $\varrho(x) = C_{\kappa} \exp(\kappa \, \ninnerprod{\xi,x})$ with $\kappa =1$. Here $\xi \in \Sphere^{2}$ is an arbitrary unit vector and $C_{\kappa}$ is a normalizing constant such that $\int_{\Sphere^{2}} \varrho(x) \, \dd \mathcal{H}^{2}(x) = 1$.
For this distribution, the unit edge vectors are slightly concentrated around $\xi$, resulting in ``straighter'' polygonal lines. Surprisingly, this does not substantially increase the difficulty of the closing problem.
We see that both iteration count and runtime of the Abikoff-Ye iteration grow linearly in terms of the desired ``accuracy'' $\log(1/\varepsilon)$, where we used \eqref{eq:StoppingCriterion} as stopping criterion.
This is expected as the Abikoff-Ye iteration boils down to steepest descent. In contrast, DRNM performs like Newton's method (because we disabled line search as soon as $q_k<1$) with iteration count and runtime depending sublinearly on $\log(1/\varepsilon)$.
It can also be seen  from \cref{fig:PolygonClosureExamples} that the spread of iteration counts for DRNM is substantially smaller than for Abikoff-Ye iteration.\footnote{The per-polygon timings of both methods are just too small to be timed accurately, so we refrained from determining their spread.}

The last example does probably not reflect the typical use case, but it points directly onto the weak spot of the Abikoff-Ye iteration.
Here we sampled from a much more concentrated spherical von Mises-Fisher distribution ($\kappa =20$), but we also introduced a sharp kink of about $180^\circ$ after half the number of edges. This produces ``almost closed''  polygonal lines (i.e., $w_*(\mu) \approx 0$) whose unit edge vectors are concentrated around two distinct points on the sphere.
As we saw in \cref{prop:SmallestEigenvaluesAndCone}, this leads to a very high condition number of the Hessian of $\varPsi_{\mu}$. It is well-known that this case is particularly bad for steepest descent if this happens at the minimizer $w_*(\mu)$. We can see the problem illustrated in the third row of \cref{fig:PolygonClosureExamples}.
In contrast, DRNM is largely unaffected by this severe loss of conditioning; it converges in no more than five iterations.

\section{Conclusion and Future Directions}

We have now given two algorithms for computing the conformal barycenter: Newton's method with fixed stepsize~\eqref{eq:Newton2}
and the regularized Newton method with line 
search~\eqref{eq:Newton3.1}--\eqref{eq:Newton3.2}. The first algorithm is primarily interesting as a device for proving theorems:~\cref{cor:time bound under NK conditions} and~\cref{thm:how often} show that in all but exponentially few cases, the conformal barycenter can be approximated to fixed accuracy in linear time. In practical implementations, we recommend the use of~\eqref{eq:Newton4.1}--\eqref{eq:Newton4.3}, which converges for all input cases with a solution~(\cref{thm:main}) and in practice does so even faster than~\eqref{eq:Newton2}. 
We note that checking the Newton-Kantorovich condition of~\cref{cor:time bound under NK conditions} reduces to computing the smallest eigenvalue of a $\AmbDim \times \AmbDim$ matrix, so one can switch between algorithms if a hard bound is desirable. 

We think that it is probably possible to establish an explicit time bound for our second algorithm in terms of the parameters $\varepsilon$ and $\delta$ from~\cref{lem:Uniform Convexity} and~\cref{lem:generalbound}. However, we have deferred this problem until a compelling reason to solve it arises.

Our original interest in the conformal barycenter was motivated by its presence in the polygon and arclength parametrized closed curve constructions of Millson and coauthors~\cite{MR1376245,MR1431002}. We intend to follow up on making this construction effective in computational geometry using the methods presented above. The Douady-Earle surfaces above were an intriguing surprise. What geometric properties do they have? This seems an avenue worth more investigation.

\section*{Acknowledgments}
This work was supported by a postdoc fellowship of the German Academic Exchange Service (DAAD). We are also grateful for the support of the Simons Foundation 
(\#524120 to Cantarella). In addition, many colleagues and friends contributed  insightful discussions about the conformal barycenter, including Kyle Chapman, Philipp Reiter, Erik Schreyer, and Clayton Shonkwiler.

\bibliographystyle{unsrt}
\bibliography{Literature}

\end{document}